\documentclass[a4paper,11pt,reqno]{amsart}
\usepackage{graphics}
\usepackage{amsmath,amssymb,amsthm,graphicx,verbatim,psfrag,multirow}
\usepackage[colorlinks=true]{hyperref}
\usepackage{color}
\usepackage{ulem}
\graphicspath{{figs/}}

\theoremstyle{plain}
\newtheorem{theo}{Theorem}[section]
\newtheorem{lem}[theo]{Lemma}

\newtheorem{cor}[theo]{Corollary}
\newtheorem{prop}[theo]{Proposition}
\theoremstyle{remark}
\newtheorem{rem}[theo]{Remark}
\theoremstyle{definition}
\newtheorem{defi}{Definition}[section]

\newcommand{\asm}{\operatorname{ASM}}

\newcommand{\vhsasm}{\operatorname{VHSASM}}

\newcommand{\dasasm}{\operatorname{DASASM}}
\newcommand{\inw}{\operatorname{in}}
\newcommand{\out}{\operatorname{out}}

\renewcommand{\ast}{\operatorname{AST}}
\newcommand{\qast}{\operatorname{QAST}}

\newcommand{\dast}{\operatorname{DAST}}
\newcommand{\mast}{\operatorname{MAST}}
\newcommand{\pf}{\operatorname{Pf}}
\newcommand{\n}{\operatorname{N}}
\newcommand{\cspp}{\operatorname{CSPP}}
\newcommand{\da}{\nabla}
\newcommand{\inv}{\operatorname{inv}}
\newcommand{\w}{\operatorname{W}}

\renewcommand{\l}{\operatorname{L}}
\renewcommand{\r}{\operatorname{R}}

\newcommand{\schur}{s}
\renewcommand{\sp}{sp}
\newcommand{\so}{so}
\newcommand{\oeven}{{o}^\text{even}}

\newcommand{\dpp}{\operatorname{DPP}}
\newcommand{\oosasm}{\operatorname{OOSASM}}
\newcommand{\qoosasm}{\operatorname{OOSASM}}
\newcommand{\sgn}{\operatorname{sgn}}
\renewcommand{\c}{\operatorname{C}}

\newcommand{\one}{\raisebox{-1.5mm}{\scalebox{0.2}{\includegraphics{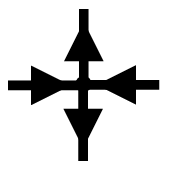}}}}
\newcommand{\mone}{\raisebox{-1.5mm}{\scalebox{0.2}{\includegraphics{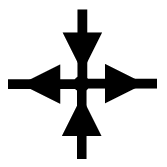}}}}
\newcommand{\nw}{\raisebox{-1.5mm}{\scalebox{0.2}{\includegraphics{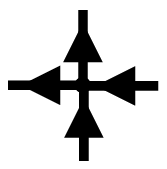}}}}
\renewcommand{\ne}{\raisebox{-1.5mm}{\scalebox{0.2}{\includegraphics{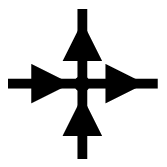}}}}
\newcommand{\sw}{\raisebox{-1.5mm}{\scalebox{0.2}{\includegraphics{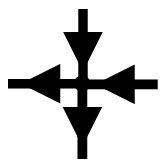}}}}
\newcommand{\se}{\raisebox{-1.5mm}{\scalebox{0.2}{\includegraphics{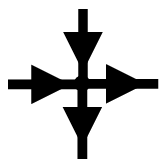}}}}
\newcommand{\lone}{\raisebox{0mm}{\scalebox{0.2}{\includegraphics{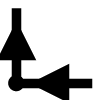}}}}
\newcommand{\rone}{\raisebox{0mm}{\scalebox{0.2}{\includegraphics{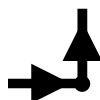}}}}
\newcommand{\lmone}{\raisebox{0mm}{\scalebox{0.2}{\includegraphics{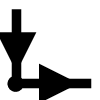}}}}
\newcommand{\rmone}{\raisebox{0mm}{\scalebox{0.2}{\includegraphics{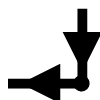}}}}
\newcommand{\lin}{\raisebox{0mm}{\scalebox{0.2}{\includegraphics{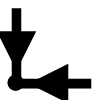}}}}
\newcommand{\lout}{\raisebox{0mm}{\scalebox{0.2}{\includegraphics{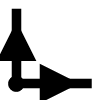}}}}
\newcommand{\rin}{\raisebox{0mm}{\scalebox{0.2}{\includegraphics{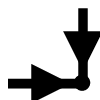}}}}
\newcommand{\rout}{\raisebox{0mm}{\scalebox{0.2}{\includegraphics{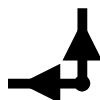}}}}
\newcommand{\bone}{\raisebox{0mm}{\scalebox{0.2}{\includegraphics{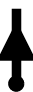}}}}
\newcommand{\bmone}{\raisebox{0mm}{\scalebox{0.2}{\includegraphics{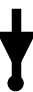}}}}
\newcommand{\wnes}{\raisebox{-3mm}{\scalebox{0.28}{
\psfrag{n}{\Huge $o_2$}
\psfrag{w}{\Huge $o_1$}
\psfrag{o}{\Huge $o_3$}
\psfrag{s}{\Huge $o_4$}
\includegraphics{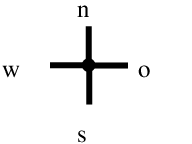}}}}

\usepackage{tikz}
\usepackage{ifthen}
\usetikzlibrary{arrows,decorations.markings}

\usepackage[utf8]{inputenc}
\newcommand{\x}{0.24}
\newcommand{\s}{0.9}
\newcommand{\wbulk}[4]{W(\raisebox{-0.35em}{\begin{tikzpicture}[
    decoration={
      markings,
      mark=at position 1 with {\arrow[scale=\s,black]{latex}};    }]
    \ifthenelse{\equal{#1}{1}}{\draw [postaction={decorate}]  (0,0) -- (\x,0);} {\draw [postaction={decorate}](\x,0) -- (0,0);}
    \ifthenelse{\equal{#2}{1}}{\draw [postaction={decorate}]  (\x,-\x) -- (\x,0);} {\draw [postaction={decorate}](\x,0) -- (\x,-\x);}
    \ifthenelse{\equal{#3}{1}}{\draw [postaction={decorate}]  (\x,0) -- (2*\x,0);} {\draw [postaction={decorate}](2*\x,0) -- (\x,0);}
    \ifthenelse{\equal{#4}{1}}{\draw [postaction={decorate}]  (\x,0) -- (\x,\x);} {\draw [postaction={decorate}](\x,\x) -- (\x,0);}
  \end{tikzpicture}}}
\newcommand{\wleft}[2]{W({\,\begin{tikzpicture}[
    decoration={
      markings,
      mark=at position 1 with {\arrow[scale=\s,black]{latex}};    }]
    \ifthenelse{\equal{#1}{1}}{\draw [postaction={decorate}]  (\x,0) -- (\x,\x);} {\draw [postaction={decorate}](\x,\x) -- (\x,0);}
    \ifthenelse{\equal{#2}{1}}{\draw [postaction={decorate}]  (\x,0) -- (2*\x,0);} {\draw [postaction={decorate}](2*\x,0) -- (\x,0);}
  \end{tikzpicture}}}
\newcommand{\wright}[2]{W({\begin{tikzpicture}[
    decoration={
      markings,
      mark=at position 1 with {\arrow[scale=\s,black]{latex}};    }]
    \ifthenelse{\equal{#1}{1}}{\draw [postaction={decorate}]  (0,0) -- (\x,0);} {\draw [postaction={decorate}](\x,0) -- (0,0);}
    \ifthenelse{\equal{#2}{1}}{\draw [postaction={decorate}]  (\x,0) -- (\x,\x);} {\draw [postaction={decorate}](\x,\x) -- (\x,0);}
  \end{tikzpicture}\,}}
\newcommand{\harrow}[1]{\begin{tikzpicture}[
    decoration={
      markings,
      mark=at position 1 with {\arrow[scale=\s,black]{latex}};    }]
    \ifthenelse{\equal{#1}{1}}{\draw [postaction={decorate}]  (0,0) -- (\x,0);} {\draw [postaction={decorate}](\x,0) -- (0,0);}
  \end{tikzpicture}}

\renewcommand{\w}{\operatorname{W}}
\newcommand{\wone}[1]{\w(\one,#1)}
\newcommand{\wmone}[1]{\w(\mone , #1)}
\newcommand{\wlone}[1]{\w(\lone, #1)}
\newcommand{\wlmone}[1]{\w(\lmone , #1)}
\newcommand{\wrone}[1]{\w(\rone , #1)}
\newcommand{\wrmone}[1]{\w(\rmone, #1)}
\newcommand{\wne}[1]{\w(\ne, #1)}
\newcommand{\wsw}[1]{\w(\sw, #1)}
\newcommand{\wlin}[1]{\w(\lin , #1)}
\newcommand{\wlout}[1]{\w(\lout, #1)}
\newcommand{\wnw}[1]{\w(\nw, #1)}
\newcommand{\wse}[1]{\w(\se, #1)}
\newcommand{\wrout}[1]{\w(\rout, #1)}
\newcommand{\wrin}[1]{\w(\rin, #1)}
\newcommand{\wwnes}[1]{\w(\wnes, #1)}

\numberwithin{equation}{section}

\newcommand{\ilse}[1]{{\color{black} #1}}

\begin{document}

\title[Extreme $\dasasm$s of odd order]{Extreme diagonally and antidiagonally symmetric alternating sign matrices of odd order}

\author[Arvind Ayyer]{Arvind Ayyer}
\address{Arvind Ayyer, Department of Mathematics, Indian Institute of Science, Bangalore - 560012, India}
\email{arvind@math.iisc.ernet.in}
\author[Roger E.~Behrend]{Roger E.~Behrend}
\address{Roger E.~Behrend, School of Mathematics, Cardiff University, Cardiff, CF24 4AG, UK}
\email{behrendr@cardiff.ac.uk}
\author[Ilse Fischer]{Ilse Fischer}
\address{Ilse Fischer, Fakult\"{a}t f\"{u}r Mathematik, Universit\"{a}t Wien, Oskar-Morgenstern-Platz 1, 1090 Wien, Austria}
\email{ilse.fischer@univie.ac.at}



\begin{abstract}
For each $\alpha \in \{0,1,-1 \}$, we count diagonally and antidiagonally symmetric alternating sign matrices ($\dasasm$s) of fixed odd order with a maximal number of $\alpha$'s along the diagonal and the antidiagonal, as well as $\dasasm$s of fixed odd order with a minimal number of $0$'s along the diagonal and the antidiagonal. In these enumerations, we encounter product formulas that have previously appeared in plane partition or alternating sign matrix counting, namely for the number of all alternating sign matrices, the number of cyclically symmetric plane partitions in a given box, and the number of vertically and horizontally symmetric $\asm$s. We also prove several refinements. For instance, in the case of $\dasasm$s with a maximal number of $-1$'s along the diagonal and the antidiagonal, these considerations lead naturally to the definition of alternating sign triangles. These are new objects that are equinumerous with $\asm$s, and we are able to prove a two parameter refinement of this fact, involving the number of $-1$'s and the inversion number on the $\asm$ side. To prove our results, we extend techniques to deal with triangular six-vertex configurations that have recently successfully been applied to settle Robbins' conjecture on the number of all $\dasasm$s of odd order. Importantly, we use a general solution of the reflection equation to prove the symmetry of the partition function in the spectral parameters. In all of our cases, we derive determinant or Pfaffian formulas for the partition functions, which we then specialize in order to obtain the product formulas for the various classes of extreme odd $\dasasm$s under consideration.
\end{abstract}

\maketitle

\section{Introduction}

An \emph{alternating sign matrix} ($\asm$) is a square matrix with entries $0$, $1$ or $-1$ such that along each row and each column the non-zero entries alternate and add up to $1$. An example is given next.
\begin{equation}
\label{dasasm-ex}
\left( \begin{matrix}
\color{red} 0 & \color{red} 0 & \color{red} 1 & \color{red} 0 & \color{red} 0 & \color{red} 0 & \color{red} 0 \\
0 & \color{red} 1 & \color{red} -1 & \color{red} 0 & \color{red} 1 & \color{red} 0 & 0 \\
1 & -1 & \color{red} 0 & \color{red} 1 & \color{red} -1 & 1 & 0 \\
0 & 0 & 1 & \color{red} -1 & 1 & 0 & 0 \\
0 & 1 & -1 & 1 & 0 & -1 & 1 \\
0 & 0 & 1 & 0 & -1 & 1 & 0 \\
0 & 0 & 0 & 0 & 1 & 0 & 0
\end{matrix} \right)
\end{equation}
(The coloring of the entries is explained shortly.)
The story of $\asm$s began in the \ilse{early} 1980's when Mills, Robbins and Rumsey \cite{MilRobRum82,DPPMRR} defined them in the course of generalizing the determinant and conjectured that the number of $n \times n$ $\asm$s is given by the following simple product formula.
\begin{equation}
\label{prod}
\prod_{i=0}^{n-1} \frac{(3i+1)!}{(n+i)!}
\end{equation}
It was more than ten years later when Zeilberger \cite{ZeilbergerASMProof} finally succeeded in providing the first proof of this formula in an $84$ page paper. Kuperberg \cite{KuperbergASMProof} then used six-vertex model techniques to give a shorter proof.

As early as the 1980's, as discussed by Robbins \cite[p.18, p. 2]{Rob91,RobbinsSymmClasses}, Stanley suggested systematically studying \emph{symmetry classes of $\asm$s}, which led Robbins to numerous conjectures, see \cite{RobbinsSymmClasses}. In particular, several symmetry classes of $\asm$s were conjectured to be enumerated by beautiful product formulas similar to \eqref{prod}. The program of proving these product formulas was recently completed in \cite{DASASM}, in which $\asm$s of odd order that are invariant under the reflections in the diagonal and in the antidiagonal, usually referred to as \emph{diagonally and antidiagonally symmetric $\asm$s} ($\dasasm$s), were enumerated. The example in \eqref{dasasm-ex} belongs to this symmetry class.  About half of Robbins' other conjectured product formulas for symmetry classes of $\asm$s were proven by Kuperberg in \cite{KuperbergRoof}, namely those for \emph{vertically symmetric $\asm$s}, \emph{half-turn symmetric $\asm$s} of even order and \emph{quarter-turn symmetric $\asm$s} of even order. Razumov and Stroganov proved the odd order cases for half-turn symmetric $\asm$s \cite{RazStroHTSymm} and for quarter-turn symmetric $\asm$s \cite{RazStr06b}, and Okada enumerated \emph{vertically and horizontally symmetric $\asm$s} \cite{OkadaCharacters}.

In addition to symmetry classes of $\asm$s, various closely-related classes of ASMs have also been studied.  Of relevance for this paper are
\emph{off-diagonally symmetric $\asm$s} (OSASMs), and \emph{off-diagonally and off-antidiagonally symmetric $\asm$s} (OOSASMs), as
introduced by Kuperberg~\cite{KuperbergRoof}.
OSASMs are even order diagonally symmetric ASMs in which each entry on the diagonal is~$0$,
while OOSASMs are DASASMs of order $4n$ in which each entry on the diagonal and antidiagonal is~$0$.
A product formula for the number of OSASMs (which is identical to that for odd order vertically symmetric ASMs)
was obtained by Kuperberg~\cite{KuperbergRoof}, but no formula for the enumeration of OOSASMs is currently known. For further information regarding symmetry classes and related classes of ASMs, see, for example,~\cite[Secs.~1.2--1.3]{DASASM},
and references therein.

The focus of the current paper is the study of odd order $\dasasm$s with a certain extreme behavior along the union of the diagonal and the antidiagonal.
Observe that a $\dasasm$ $(a_{i,j})_{1 \le i, j \le 2n+1 }$ of order $2n+1$ is determined by its entries in the \emph{fundamental triangle} $\{(i,j) | 1 \le i \le n+1, i \le j \le 2n+2-i\}$---in the example \eqref{dasasm-ex}  marked with red.
For a given odd-order $\dasasm$ $A$ and $\alpha \in \{0,1,-1\}$, let
$$\n_{\alpha}(A) = \#  \text{\, of $\alpha$'s along the portions of the diagonals of $A$ that lie in the fundamental triangle},$$
where here and in the following ``diagonals'' refers to the union of the (main) diagonal and the (main) antidiagonal. In the example \eqref{dasasm-ex}, $\n_{-1}(A)=2$, $\n_{1}(A)=1$ and $\n_{0}(A)=4$.
We have the following bounds for these statistics. (The proof of the proposition is provided in Subsection~\ref{prop}.)

\begin{prop}
\label{bounds}
For any $(2n+1) \times (2n+1)$ $\dasasm$ $A$, the statistics $\n_{\alpha}(A)$ lie in the following intervals.
\begin{enumerate}
\item $0 \le \n_{-1}(A) \le n$
\item $0 \le \n_{1}(A) \le n+1$
\item $n \le \n_{0}(A) \le 2n$
\end{enumerate}
All inequalities are sharp.
\end{prop}

The research presented in this paper started out with numerical data providing evidence that for four out of the six inequalities in the proposition we have the following phenomenon: The number of $\dasasm$s where equality is attained is round and in fact equal to numbers that have previously appeared in plane partition or alternating sign matrix counting. \ilse{(For the two other inequalities, the numbers are not even round.)} It is the primary objective of this paper to prove all these empirical observations. A number of generalizations including determinant or Pfaffian formulas for certain generating functions that are known as \emph{partition functions} in a physics
context are also provided. Next we state the main results.

\subsection{Case $\n_{-1}(A)=n$:} Order $2n+1$ $\dasasm$s $A$ with $\n_{-1}(A)=n$ are proven to be equinumerous with $n \times n$ $\asm$s. It is remarkable that this now establishes a new class of objects with this property.
The \ilse{two} other currently-known classes are \emph{totally symmetric self-complementary plane partitions} in an $2n \times 2n \times 2n$ box (which were introduced by Stanley \cite{Sta86a} and enumerated by Andrews \cite{And94}, and for which the equinumeracy with $\asm$s was first conjectured by Mills, Robbins and Rumsey
\cite{MilRobRum86}) and \emph{descending plane partitions} ($\dpp$s) with parts no greater than $n$ (which were introduced and enumerated by Andrews \cite{AndrewsMacdonald}, and for which equinumeracy with $\asm$s was first conjectured by Mills, Robbins and Rumsey \cite{MilRobRum82,DPPMRR}).

We are also able to identify statistics that have the same distribution:
For an $\asm$ $A$, let
$$\mu(A)= \# \text{\, of $-1$'s in $A$},$$
and, for an order $2n+1$ $\dasasm$
$A$, let
$$\mu_{\da}(A) = (\#  \text{\, of $-1$'s in the fundamental triangle of $A$})-n.
$$
\ilse{We will consider $\mu_{\da}(A)$ only for $A \in \dasasm(2n+1)$ with
$\n_{-1}(A)=n$, in which case it is just the number of $-1$'s in the interior of the fundamental triangle.} Throughout the paper, the sets of order $n$ $\asm$s and
order $n$ $\dasasm$s are denoted by $\asm(n)$ and $\dasasm(n)$, respectively.

\begin{theo}
\label{maxmone}
The distribution of the statistic $\mu$  on the set $\asm(n)$ is equal to the distribution of the statistic $\mu_{\da}$ on the set of $A \in \dasasm(2n+1)$ with $\n_{-1}(A)=n$, i.e., for all non-negative integers $m,n$  we have
$$
|\{A \in \asm(n) \mid \mu(A)=m\}| = |\{A \in \dasasm(2n+1) \mid \n_{-1}(A)=n, \, \mu_{\da}(A)=m \}|.
$$
\end{theo}

A corresponding result---with  the set of $A \in \dasasm(2n+1)$ such that $\n_{-1}(A)=n$ replaced by the set of $\dpp$s with parts no greater than $n$ and $\mu_{\da}$ replaced by the number of special parts in the $\dpp$---was conjectured by
Mills, Robbins and Rumsey \cite{DPPMRR} and proven by Behrend, Di Francesco and Zinn-Justin \cite{DPP1Behrend}. In fact, they have proven a refinement (also conjectured by Mills, Robbins and Rumsey) that involves two additional statistics. On the $\asm$ side, these are the {\it inversion number} and the {\it position of the unique $1$ in the top row}. (This was further generalized in \cite{DPP2Behrend}, where they also included the position of the unique $1$ in the bottom row.) In Section~\ref{ast}, we define an inversion number $\inv_{\da}$ on the set of $A \in  \dasasm(2n+1)$ with $\n_{-1}(A)=n$ such that the joint distribution of $\mu$ and $\inv$ on $\asm(n)$ is equal to the joint distribution of $\mu_{\da}$ and $\inv_{\da}$ on this subset of $\dasasm(2n+1)$ (Theorem~\ref{maxmone-gen}), thus establishing a refinement of Theorem~\ref{maxmone}. Other generalizations of Theorem~\ref{maxmone} in terms of partition functions for certain triangular six-vertex configurations are provided in Theorems~\ref{astpartfunct} and \ref{astschurtheo}, and in Corollary~\ref{astpartfunct1}.

\subsection{Case $\n_{1}(A)=n+1$:} In the second theorem, {\it cyclically symmetric plane partitions} ($\cspp$s) make an appearance. They were enumerated by Andrews \cite{AndrewsMacdonald}.

\begin{theo}
\label{maxone}
The number of $A \in \dasasm(2n+1)$ with $\n_1(A)=n+1$ is equal to the number of $\cspp$s in an $n \times n \times n$ box.
\end{theo}

The number of $\cspp$s has previously appeared in the program of enumerating symmetry classes of $\asm$s: Robbins conjectured \cite{RobbinsSymmClasses} and Kuperberg~\cite[Theorem 2]{KuperbergRoof} proved that the number of $2n \times 2n$ half-turn symmetric alternating sign matrices is the product of the total number of $n \times n$ $\asm$s and the number of $\cspp$s in an $n \times n \times n$ box.

Generalizations of Theorem~\ref{maxone} in terms of partition functions  are
provided in Theorems~\ref{dastpartfunct} and \ref{dastschurtheo}, and in Corollary~\ref{dastpartfunct1}.

\subsection{Case $\n_{0}(A)=n$:} For the lower bound of $\n_0$, we again have the total number of $\asm$s turning up. This is one of the few exceptional cases in this field that can easily be proven by establishing a bijection with other objects that are known to be enumerated by these numbers, in this particular case  \ilse{with $A \in \dasasm(2n+3)$ satisfying $\n_{-1}(A)=n+1$ (which also appear in Theorem~\ref{maxmone} and are shown to be equinumerous with $\asm$s of order $n+1$).
This bijection is provided in Subsection~\ref{imply}.}
As it is always the situation in such a case in this area so far,
the bijection is almost trivial, which is the reason why the following result can be viewed as a corollary of Theorem~\ref{maxmone}.

\begin{cor}
\label{minzero}
The number of $A \in \dasasm(2n+1)$ with $\n_0(A)=n$ is equal to the number of
$\asm$s of order $n+1$.
\end{cor}

\subsection{Case $\n_{0}(A)=2n$:}
In the case of the upper bound of $\n_0$, the number of {\it vertically and horizontally symmetric alternating sign matrices} ($\vhsasm$s) appears.

\begin{theo}
\label{maxzero}
The number of $A \in \dasasm(2n+1)$ with $\n_0(A)=2n$ is equal to the number
of order $2n+3$ $\vhsasm$s.
\end{theo}

The $\dasasm$s $A$ of order $2n+1$ with $\n_0(A)=2n$ can be regarded as odd order versions of $\oosasm$s, and so we denote this subset of
$\dasasm(2n+1)$ by $\qoosasm(2n+1)$. Indeed, as the central entry of an odd order $\dasasm$ is always non-zero, all other entries on the diagonals of such a $\dasasm$ are zero and thus this is for odd order as close as one can get to Kuperberg's original $\qoosasm$s.\footnote{Note that the central
entry $c$ of an $A \in \qoosasm(2n+1)$ is $(-1)^n$: The sum of entries in $A$ is certainly $2n+1$, however it is also $4 s + c$, where $s$ is the sum of entries in the fundamental domain of $A$ without $c$.} Interestingly, vertically symmetric
$\oosasm$s of odd order have been enumerated by Okada \cite[(B2) and (B3) of Theorem~1.3]{OkadaCharacters}. (There they are referred to as $\operatorname{VOSASM}$s, since the vertical symmetry and the diagonal symmetry implies the antidiagonal symmetry.)

Generalizations of Theorem~\ref{maxzero} in terms of partition functions are
provided in Theorems~\ref{qoosasmpartfunct} and \ref{spq}.

\subsection{Outline of the paper} In Section~\ref{basics}, we provide several basic observations: We prove Proposition~\ref{bounds}, and, already for this purpose, it is useful to translate our problems into the counting of certain orientations of triangular regions of the square grid (\emph{triangular six-vertex configurations}). In this section, we also present a simple bijection between the order $n$ \ilse{$\dasasm$} objects of Theorem~\ref{maxmone} and the order $n-1$ \ilse{$\dasasm$} objects of Corollary~\ref{minzero} that consists merely of manipulations close to the diagonal and the antidiagonal. In Section~\ref{weight-sec}, we introduce the vertex weights and some of their properties (\emph{Yang--Baxter equation}, \emph{reflection equations}), and use these weights to define the \emph{partition function}, i.e., a multiparameter generating function of the objects we want to count. There we also introduce several specializations of this partition function that are used to prove our theorems. The partition function is a Laurent polynomial in the so-called \emph{spectral parameters}, and, in Section~\ref{char-part}, we provide characterizations of this partition function that are used later on. In Sections~\ref{ast} -- \ref{qoosasm}, we then employ all these preparations to prove Theorems~\ref{maxmone}, \ref{maxone} and \ref{maxzero}, respectively. Finally, in Appendix~A we provide an ad-hoc counting of alternating sign triangles with \ilse{a single} $-1$.

\section{Basics: Proof of Proposition~\ref{bounds}, characterization of extreme configurations and Theorem~\ref{maxmone} implies Corollary~\ref{minzero}}
\label{basics}

\subsection{Triangular six-vertex configurations} Let $A=(a_{i,j})_{1 \le i, j \le 2n+1 }$ be
a  $\dasasm$ of order $2n+1$. Its restriction
to the fundamental triangle $(a_{i,j})_{1 \le i \le n+1, i \le j \le 2n+2-i}$ is
said to be an \emph{odd $\dasasm$-triangle} of order $n$.
An example of an odd $\dasasm$-triangle of order $6$ is given next.
\begin{equation}
\label{ex-dasasm-triangle}
\begin{array}{ccccccccccccccc}
 & 0  & 0 & 0 & \color{red} 1 & 0 & 0 & 0 & 0 & 0 & \color{red} 0 & 0 & 0 & 0 &  \\
 &   & 1 & 0 & \color{red} -1 & 0 & 0 & 0 & 1 & 0 & \color{red} 0 & 0 & 0 &  &  \\
 &   &  & 0 & \color{red} 1 & 0 & 0 & 0 & -1 & 0 & \color{red} 1 & 0 &  &  &  \\
 &   &  &  & \color{red} -1 & \color{red} 1 & \color{red} 0 & \color{red} 0 & \color{red} 0 & \color{red} 0 & \color{red} -1 &  &  &  &  \\
  &   &  &  &  & -1 & 1 & 0 & 0 & 0 &  &  &  &  &  \\
  &   &  &  &  &  & -1 & 0 & 1 &  &  &  &  &  &  \\
  &   &  &  &  &  &  & 1 &  &  &  &  &  &  &
\end{array}
\end{equation}
In fact, a triangular array
$(a_{i,j})_{1 \le i \le n+1, i \le j \le 2n+2-i}$ of this form, in which each entry is $0$,
$1$ or $-1$, is an odd $\dasasm$-triangle of order $n$ if and only if, for each
$j \in \{1,2,\ldots,n+1\}$, the
non-zero entries in the following sequence alternate---read from top left to the top right---and add up to $1$.
\begin{equation}
\label{path}
\begin{array}{ccccccc}
& a_{1,j} & & & & a_{1,2n+2-j} & \\
&a_{2,j} & & & & a_{2,2n+2-j} & \\
\downarrow &\vdots &  & & & \vdots & \uparrow \\
&a_{j-1,j} & & & & a_{j-1,2n+2-j} & \\
&a_{j,j} & a_{j,j+1} & \ldots & a_{j,2n+1-j} & a_{j,2n+2-j} & \\
&  & & \rightarrow & &  &
\end{array}
\end{equation}
(In the example \eqref{ex-dasasm-triangle}, this sequence is indicated in red for $j=4$.)
Let us clarify that in the special case $j=n+1$, we require that the sequence
$$a_{1,n+1},a_{2,n+1},\ldots,a_{n,n+1},a_{n+1,n+1},a_{n,n+1},\ldots,a_{1,n+1}$$
has this property, and this is satisfied if and only if the non-zero entries of $a_{1,n+1},a_{2,n+1},\ldots,a_{n+1,n+1}$ alternate, the first non-zero entry in this sequence is $1$, and $a_{n+1,n+1} \not= 0$.
For an odd $\dasasm$-triangle $A$, we define $\n_{\alpha}(A) = \n_{\alpha}(A')$, where $A'$ is the $\dasasm$ corresponding to $A$.

In the {\it six-vertex model}, these triangular arrays correspond to orientations of a triangular region
of the square grid with $n+2$ centered rows as indicated in Figure~\ref{DASASMgraph}, where
\begin{itemize}
\item the degree $4$ vertices (\emph{bulk vertices}) have two incoming and two outgoing edges, and
\item the top vertical edges point up.
\end{itemize}
\ilse{Rows $2,3,\ldots,n+2$ in the grid will correspond to rows $1,2,\ldots,n+1$ of the $\dasasm$-triangle, respectively.}
The term \emph{six-vertex} is derived from the fact that there are $\binom{4}{2}=6$ possible local configurations around a bulk vertex.
The underlying undirected graph is denoted by $\mathcal{T}_n$ in the following.

The triangular six-vertex configuration of an odd $\dasasm$-triangle can be obtained by restricting the standard six-vertex configuration on a square of the respective $\dasasm$ to the fundamental triangle.
More concretely, each vertex of $\mathcal{T}_n$, except for the degree $1$ vertices at the top, corresponds to an entry of the associated odd $\dasasm$-triangle of order $n$. As usual, a bulk vertex
whose two incident vertical edges are outgoing ($\one$), respectively incoming
($\mone$),
corresponds to a $1$, respectively $-1$, in the odd $\dasasm$-triangle, while
all other bulk vertices correspond to zeros.
\emph{Left boundary vertices} and \emph{right boundary vertices} are the degree $2$ vertices on the left and right boundary, and such a vertex corresponds to a
 $1$ (resp.\ $-1$) if and only if the local configuration around the vertex is the restriction of the local bulk configuration corresponding to a $1$ (resp.\ $-1$). This also applies to the bottom vertex. That is,
$$
\one, \lone, \rone, \bone \leftrightarrow 1,    \mone, \lmone, \rmone, \bmone \leftrightarrow -1 \quad \text{and} \quad
\sw, \ne, \se, \nw, \lin, \lout, \rin, \rout \leftrightarrow 0.$$
The configuration in Figure~\ref{DASASMgraph} is the six-vertex configuration of the odd $\dasasm$-triangle in \eqref{ex-dasasm-triangle}. Note that when restricting the six-vertex configuration of an odd $\dasasm$ to the fundamental triangle, the local configuration $\lone$ on the left-boundary must originate from $\one$, and thus corresponds to $1$, because the only other local configuration with this restriction, i.e., $\nw$, cannot appear on the diagonal of a six-vertex configuration of a $\dasasm$. We have a similar situation for $\lmone, \rone, \rmone$.

\begin{figure}
\scalebox{0.4}{\includegraphics{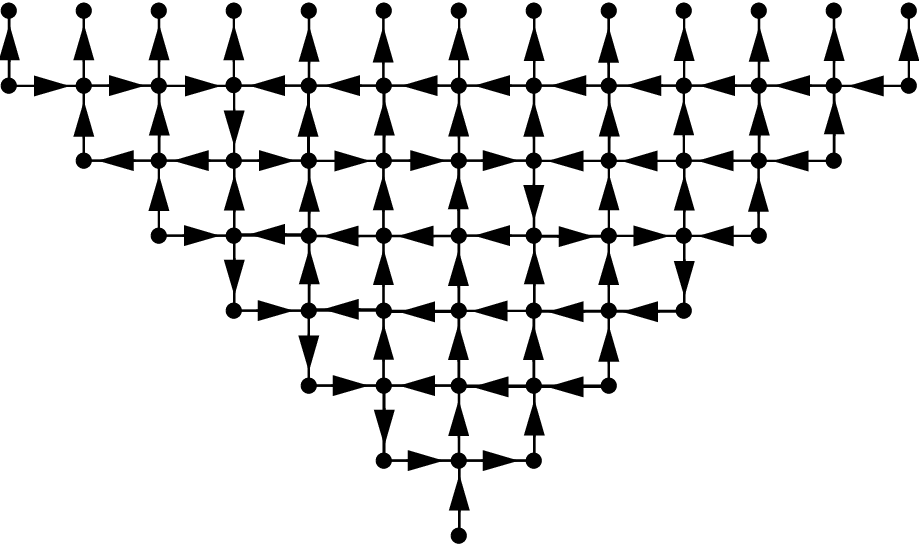}}
\caption{\label{DASASMgraph} Triangular six-vertex configuration of order $6$}
\end{figure}

\subsection{Proof of Proposition~\ref{bounds}}
\label{prop}

In order to derive the crucial identity \eqref{boundeq}, we employ the fact that, in a directed graph, the sum of all outdegrees is equal to the sum of all indegrees. The bulk vertices of a triangular six-vertex configuration as well as the vertices on the left and right boundary that correspond to $\pm 1$'s are balanced in the sense that the outdegree is equal to  the indegree, and so these vertices do not contribute to this identity.

For an odd $\dasasm$-triangle $A$, we denote the number of boundary zeros with indegree $2$ in the corresponding triangular six-vertex configuration by $\n_{0,\text{in}}(A)$, and number of boundary zeros with outdegree $2$ by $\n_{0,\text{out}}(A)$. In the example in Figure~\ref{DASASMgraph}, we have $\n_{0,\text{in}}(A)=0$ and $\n_{0,\text{out}}(A)=6$.  Now, since all $2n+1$ top vertices have indegree $1$,
the sum of all outdegrees is $2 \n_{0,\text{out}}(A) + [a_{n+1,n+1}=1]$ (where we use the Iverson bracket, i.e., $[\text{statement}]=1$ if the statement is true, and $0$ otherwise), while the sum of all indegrees is $2 \n_{0,\text{in}}(A) + 2n+1 + [a_{n+1,n+1}=-1]$, and, since $[a_{n+1,n+1}=1]-[a_{n+1,n+1}=-1]=a_{n+1,n+1}$,
we can conclude that
\begin{equation}
\label{boundeq}
2 \n_{0,\text{out}}(A) + a_{n+1,n+1} - (2 \n_{0,\text{in}}(A) + 2n+1)=0.
\end{equation}

The lower bounds for $\n_{\pm 1}(A)$ are trivial and the upper bound for $\n_{0}(A)$ follows because the central entry of a $\dasasm$ is always non-zero.

An example of a $\dasasm$ in which the upper bound of $\n_0$ is attained is the matrix where every other entry of the restriction of the superdiagonal (i.e., the secondary diagonal immediately above the main diagonal) to the fundamental triangle is $1$, the central entry is
$(-1)^n$ and all other entries of the fundamental triangle are $0$. \ilse{The corresponding odd $\dasasm$-triangles of order $4$ and $5$ are the following:}
$$
\begin{array}{ccccccccccccccc}
 &   &  &  & 0 & 1 & 0 & 0 & 0 & 0 & 0 & 0 & 0 &  &  \\
 &   &  &  &  & 0 & 0 & 0 & 0 & 0 & 0 & 0 &  &  &  \\
 &   &  &  &  &  & 0 & 1 & 0 & 0 & 0 &  &  &  &  \\
 &   &  &  &  &  &  & 0 & 0 & 0 &  &  &  &  &  \\
  &   &  &  &  &  & &  & 1 &  &  &  &  &  &
\end{array}
\begin{array}{ccccccccccccccc}
 &   &  &  & 0 & 1 & 0 & 0 & 0 & 0 & 0 & 0 & 0 & 0 & 0  \\
 &   &  &  &  & 0 & 0 & 0 & 0 & 0 & 0 & 0 & 0  &  0 &  \\
 &   &  &  &  &  & 0 & 1 & 0 & 0 & 0 & 0 & 0  &  &  \\
 &   &  &  &  &  &  & 0 & 0 & 0 & 0 & 0  &  &  &  \\
  &   &  &  &  &  & &  & 0 & 1 & 0 &  &  &  &  \\
   &   &  &  &  &  & &  &  & -1 &  &  &  &  &
\end{array}
$$
As we \ilse{will} see below that
$\n_{-1}(A)=n$ implies $\n_{1}(A)=0$, and $\n_{1}(A)=n+1$ implies $\n_{-1}(A)=0$, the fact that the two trivial inequalities are sharp follows from the fact that the upper bounds of $\n_{\pm 1}$ are sharp. The latter facts are shown below.

\emph{\ilse{Inequality} $\n_{-1}(A) \le n$:} Using
$\n_1(A) + \n_{-1}(A) + \n_{0,\out}(A) + \n_{0,\inw}(A) = 2n+1$ and \eqref{boundeq}, we deduce
\begin{align*}
\n_{-1}(A) &= 2n+1-\n_{1}(A) - \n_{0,\text{out}}(A) - \n_{0,\text{in}}(A) \\
&=
n + \frac{1}{2} - \n_{1}(A) - 2 \n_{0,\text{in}}(A) + \frac{a_{n+1,n+1}}{2} \le n.
\end{align*}
We have equality iff $\n_{0,\text{in}}(A)=\n_{1}(A)=0$ and $a_{n+1,n+1}=-1$, or else $\n_{0,\text{in}}(A)=0$ and $\n_{1}(A)=a_{n+1,n+1}=1$. However, the latter possibility cannot occur since
then the bottom bulk vertex in the central column would have three incoming edges---the left (resp. right) boundary vertex adjacent to it is of type $\lmone$ or $\lout$ (resp. $\rmone$ or $\rout$)---and this is impossible.

An order $2n+1$ $\dasasm$ with $\n_{-1}(A)=n$ is
$$
A=\left( \begin{matrix}
0 & 1 & 0 & \cdots & \cdots & \cdots & \cdots & 0 \\
1 & -1 & 1 & 0 &   &  &  & \vdots \\
0 & 1 & -1 & 1 & 0  &   &  & \vdots \\
0 & 0 & 1 & -1 & 1 & 0  &  & \vdots \\
\vdots &  & \ddots  &  \ddots & \ddots & \ddots  & \ddots & \vdots   \\
\vdots &  & & 0 & 1 & -1 & 1 & 0 \\
\vdots &  &  &  & 0 & 1 & -1  & 1 \\
0 & \cdots & \cdots & \cdots & \cdots & 0 & 1 & 0
\end{matrix} \right).
$$

\emph{\ilse{Inequality} $\n_{1}(A) \le n+1$:} Here we have
$$
\n_{1}(A) = n + \frac{1}{2} - \n_{-1}(A) - 2 \n_{0,\text{in}}(A) + \frac{a_{n+1,n+1}}{2} \le  n+1,
$$
and equality is attained iff $\n_{0,\text{in}}(A)=0$, $\n_{-1}(A)=0$ and $a_{n+1,n+1}=1$. The identity matrix is a matrix where equality is attained.

\emph{\ilse{Inequality} $\n_{0}(A) \ge n$:} From \eqref{boundeq}, it follows
$$
\n_{0}(A) =  \n_{0,\text{out}}(A) + \n_{0,\text{in}}(A)  \ge \n_{0,\text{out}}(A) - \n_{0,\text{in}}(A) = n + \frac{1}{2} -
\frac{a_{n+1,n+1}}{2} \ge n,
$$
and so we have equality iff $\n_{0,\text{in}}(A)=0$ and $a_{n+1,n+1}=1$. Since $\n_{1}(A)=n+1$ implies $\n_{0}(A)=n$, this inequality is sharp too.

\subsection{Characterization of extreme configurations}

The proof of Proposition~\ref{bounds} implies immediately the following characterization.

\begin{cor}
\label{char}
Let $A$ be an odd $\dasasm$-triangle of order $n$.
\begin{enumerate}
\item\label{char-1} We have $\n_{-1}(A)=n$ if and only if, for each row except the bottom row, the sum of entries is $1$ when disregarding the left and right boundary entries. In the six-vertex configuration, this is fulfilled if and only if, in each row, the leftmost horizontal edge points to the right and the rightmost horizontal edge points to the left.
\item\label{char-2} We have $\n_{1}(A)=n+1$ if and only if the sum of entries in each row is $1$. In the six-vertex configuration, this is fulfilled if and only if, in each row, the leftmost vertical edge and the rightmost vertical edge point upwards.
\item\label{char-3} We have $\n_{0}(A)=n$ if and only if, for each row, one of the following is true.
\begin{enumerate}
\item The sum of entries is $1$.
\item The sum of entries with the left or right boundary entry
excluded from the sum is $1$.
\item The sum of entries with both boundary entries excluded from the sum is $1$.
\end{enumerate}
\end{enumerate}
\end{cor}

\begin{rem}
\begin{enumerate}
\item Odd $\dasasm$-triangles $A$ of order
$n$
with $\n_{-1}(A)=n$ are precisely the odd $\dasasm$-triangles of order $n$ in which the sum of entries is minimal. Indeed, the characterization in Corollary~\ref{char}
(\ref{char-1}) implies the following different characterization: An odd $\dasasm$-triangle $A$ of order $n$ satisfies $\n_{-1}(A)=n$ if and only if each column sum is $0$. In order to see this, also recall that the bottom entry of such an $A$ is $-1$. Now the sum of entries in an odd $\dasasm$-triangle is at least $0$ (since each column sum is at least $0$), and the minimum is attained if and only if each column sum is $0$.
\item The characterization in Corollary~\ref{char} (\ref{char-2}) implies that
odd $\dasasm$-triangles $A$ of order $n$
with $\n_{1}(A)=n+1$ are precisely the odd $\dasasm$-triangles of order $n$ in which the sum of entries is maximal. This is because the sum of entries in an odd $\dasasm$-triangle of order $n$ is at most $n+1$, since each of the $n+1$ rows has row sum at most $1$, and the maximum is attained if and only if each row has sum $1$.
\end{enumerate}
\end{rem}

\subsection{Theorem~\ref{maxmone} implies Corollary~\ref{minzero}}
\label{imply}
This is best understood in terms of the six-vertex model. As noted in the proof of Proposition~\ref{bounds}, the triangular six-vertex configurations equivalent to the objects from Corollary~\ref{minzero} are characterized by the two facts that there is no left or right boundary vertex with indegree $2$ and the bottom vertical edge points upwards.

To transform an order $n$ configuration of Corollary~\ref{minzero} into an order $n+1$ configuration of Theorem~\ref{maxmone}, add vertices left of each left boundary vertex and connect the new vertices to their right neighbors by an edge that is directed to the right, see Figure~\ref{implies} for an example. Add similar vertices and edges on the right boundary, where the new horizontal edges are directed to the left.

Also add a new vertex below the bottom central vertex and introduce vertical edges that connect the $2n+3$ added vertices to their top neighbors (for the two new top boundary vertices add vertices above them and two vertical edges that are directed upwards). Since the indegree of each former boundary vertex was either $0$ or $1$ (and is currently either $1$ or $2$) and the former bottom vertex had indegree $0$ (and has currently indegree $2$), there is a unique way to orient the new vertical edges such that the former boundary and bottom vertices (now bulk vertices) are balanced. By Corollary~\ref{char}, this produces an object with the desired properties.

To reverse the transformation, simply delete left and right boundary vertices of an order $n+1$ configuration of Theorem~\ref{maxmone} as well as the bottom vertex, and all edges incident with these vertices. None of the new left or right boundary vertices can have indegree $2$, since, before the deletion of the vertices and edges, the leftmost vertical edge in each row was pointing to the right, while the rightmost vertical edge in each row was pointing to the left. This also implies that the bottom bulk vertex in the central column was of type $\one$ before the deletion, and thus the new bottom vertex points upwards.

\begin{figure}
\scalebox{0.4}{\includegraphics{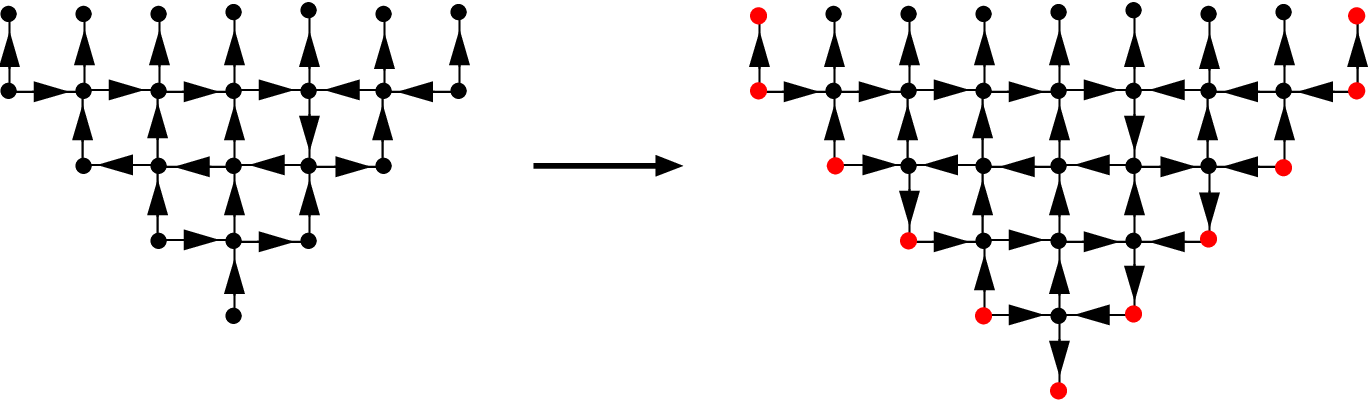}}
\caption{\label{implies} Illustration of the proof that Theorem~\ref{maxmone} implies Corollary~\ref{minzero}}
\end{figure}

\subsection{Alternating sign triangles}
\label{astintroduce}

If we delete the diagonals of an odd $\dasasm$-triangle $A$ of order $n$ with $\n_{-1}(A)=n$, then, by Corollary~\ref{char}, we obtain an object of the following type.

\begin{defi}
An \emph{alternating sign triangle} ($\ast$) of order $n$ is a triangular array $(a_{i,j})_{1 \le i \le n, i \le j \le 2n-i}$ in which each entry is $0$, $1$ or $-1$ and the following conditions are fulfilled.
\begin{enumerate}
\item\label{AST-1} The non-zero entries alternate in each row and each column.
\item\label{AST-2} All row sums are $1$.
\item\label{AST-3} The topmost non-zero entry of each column is $1$ (if it exists).
\end{enumerate}
The set of order $n$ $\ast$s is denoted by $\ast(n)$.
\end{defi}

Here is a list of all $\ast$s of order $3$.
$$
\begin{aligned}
\begin{array}{ccccc}
1 & 0 & 0 & 0 & 0  \\
& 1 & 0 & 0  & \\
& & 1 & &
\end{array} & \left|
\begin{array}{ccccc}
0 & 0 & 0 & 1 & 0  \\
& 1 & 0 & 0  & \\
& & 1 & &
\end{array} \right|  & \hspace{-3mm}
\begin{array}{ccccc}
0 & 0 & 0 & 0 & 1  \\
& 1 & 0 & 0  & \\
& & 1 & &
\end{array} \hspace{3mm} &   \left|
\begin{array}{ccccc}
1 & 0 & 0 & 0 & 0  \\
& 0 & 0 & 1  & \\
& & 1 & &
\end{array} \right.  \\
\begin{array}{ccccc}
0 & 1 & 0 & 0 & 0  \\
& 0 & 0 & 1  & \\
& & 1 & &
\end{array}  & \left|
\begin{array}{ccccc}
0 & 0 & 0 & 0 & 1  \\
& 0 & 0 & 1  & \\
& & 1 & &
\end{array} \right| & \hspace{-3mm}
\begin{array}{ccccc}
0 & 0 & 1 & 0 & 0  \\
& 1 & -1 & 1  & \\
& & 1 & &
\end{array}
\end{aligned}
$$

In fact, order $n$ $\ast$s are in bijection with \ilse{order $2n+1$ $\dasasm$s} with $\n_{-1}(A)=n$ as the diagonals can be reconstructed as follows:
$$
\text{Place a $-1$ below each column of the $\ast$ whose entries add up to $1$, and place a $0$ otherwise.}
$$
The resulting triangle is surely an odd $\dasasm$-triangle $A'$ of order $n$, as the non-zero entries of all sequences as given in $\eqref{path}$ are alternating and add up to $1$; furthermore, it is the unique way of adding entries in these positions to achieve that. Moreover, we have $\n_{-1}(A')=n$ as an $\ast$ has precisely $n$ columns that add up to $1$. This is because
the total sum of all entries in an order $n$ $\ast$ is $n$ (by property~\eqref{AST-2} in the definition) and each of the $2n-1$ column sums is either $0$ or $1$
(by properties~\eqref{AST-1} and \eqref{AST-3}).

Theorem~\ref{maxmone} states that, for all non-negative integers $n,m$, the number of $n \times n$ $\asm$s with $m$ occurrences of $-1$ is equal to the number of $\ast$s of order $n$ with $m$ occurrences of $-1$.

For $m=0$, this is easy to see: $\asm$s without $-1$ are permutation matrices. There are also $n!$ $\ast$s of order $n$: Each row contains precisely one $1$ and we build up the $\ast$ by placing in each row a $1$, starting with the bottom row. For the $1$ in the bottom row, there is one choice, for the $1$ in the penultimate row there are in \ilse{principle} $3$ possible columns, but one is already taken by the $1$ in the bottom row and thus there are $2$ actual choices. In general, in the $i$-th row counted from the bottom, there are $2i-1$ columns, but $i-1$ are already occupied by $1$'s that are situated in rows below. This leaves us with $i$ possibilities. In total, there are $1 \cdot 2 \cdot 3 \cdots  n$ $\ast$s with $n$ rows and no $-1$.

In order to give an indication as to why it is probably not easy to construct a bijection between $\asm$s and $\ast$s, we elaborate on the $m=1$ case in Appendix~\ref{k1}.

As a further comment that is also related to the previous subsection, observe that in order to transform an odd $\dasasm$-triangle $A$ of order $n$ with $\n_0(A)=n$ into the corresponding $\ast$ of order $n+1$, one simply has to replace all $-1$'s along the left and right boundary by $0$'s.

\subsection{Dual alternating sign triangles and quasi alternating sign triangles}
\label{section_dast}

If we delete the diagonals of an odd $\dasasm$-triangle of order $n$ with $\n_{1}(A)=n+1$, then, by Corollary~\ref{char}, we obtain triangular arrays of the following type.

\begin{defi}
A \emph{dual alternating sign triangle} ($\dast$) of order $n$ is a triangular array \\
$(a_{i,j})_{1 \le i \le n, i \le j \le 2n-i}$ in which each entry is $0$, $1$ or $-1$ and the following conditions are fulfilled.
\begin{enumerate}
\item\label{DAST-1} The non-zero entries alternate in each row and each column.
\item\label{DAST-2} All column sums are $0$.
\item\label{DAST-3} The topmost non-zero entry of each column is $1$ (if it exists).
\end{enumerate}
The set of order $n$ $\dast$s is denoted by $\dast(n)$.
\end{defi}

Next we display all order $3$ $\dast$s.
$$
\begin{aligned} \left.
\begin{array}{ccccc}
0 & 0 & 0 & 0 & 0  \\
& 0 & 0 & 0  & \\
& & 0 & &
\end{array} \right| & \left.
\begin{array}{ccccc}
0 & 1 & 0 & 0 & 0  \\
& -1 & 0 & 0  & \\
& & 0 & &
\end{array} \right| & \left. \hspace{-3mm}
\begin{array}{ccccc}
0 & 0 & 0 & 1 & 0  \\
& 0 & 0 & -1  & \\
& & 0 & &
\end{array} \hspace{3mm}  \right| &
\begin{array}{ccccc}
0 & 0 & 1 & 0 & 0  \\
& 0 & -1 & 0  & \\
& & 0 &  &
\end{array}   \\ \left.
\begin{array}{ccccc}
0 & 0 & 1 & 0 & 0  \\
& 0 & 0 & 0  & \\
& & -1 &  &
\end{array} \right| & \left.
\begin{array}{ccccc}
0 & 0 & 0 & 0 & 0  \\
& 0 & 1 & 0  & \\
& & -1 & &
\end{array} \right| & \left. \hspace{-3mm}
\begin{array}{ccccc}
0 & 1 & 0 & 0 & 0  \\
& -1 & 1 & 0  & \\
& & -1 & &
\end{array} \right| &
\begin{array}{ccccc}
0 & 0 & 0 & 1 & 0  \\
& 0 & 1 & -1  & \\
& & -1 & &
\end{array}
\end{aligned}
$$

It is possible to reconstruct the deleted diagonal entries, except for the rows that contain only zeros \ilse{(referred to as $0$-rows)}, in the following \ilse{way}:
\begin{itemize}
\item Add a $1$ below the bottom entry.
\item If the leftmost non-zero entry of a row is $-1$ (resp.\ 1), place a $1$ (resp.\ 0) left of the leftmost entry of that row.
\item If the rightmost non-zero entry of a row is $-1$ (resp.\ 1), place a $1$ (resp.\ 0) right of the rightmost entry of that row.
\item In the case of $0$-rows, there are two choices of placing a $1$ on one end and a $0$ on the other.
\end{itemize}
Therefore,
\begin{equation}
\label{2enum}
|\{A \in \dasasm(2n+1) \mid \n_{1}(A)=n+1\}| = \sum_{A \in \dast(n)} 2^{\#
\text{of $0$-rows of A}}.
\end{equation}
The $\dast$s listed above correspond to $8,2,2,2,2,2,1,1$ $\dasasm$s of order $7$ with $\n_{1}(A)=4$, respectively. This is in accordance with Theorem~\ref{maxone} as there are $20$ $\cspp$s in a $3 \times 3 \times 3$ box.

We \ilse{now} introduce another set of triangular arrays that is, on the one hand, equinumerous with the set of $A \in \dasasm(2n+1)$ such that $\n_{1}(A)=n+1$, and, on the other, contains all $\ast$s of order $n$.

\begin{defi}
A \emph{quasi alternating sign triangle} ($\qast$) of order $n$ is a triangular array \\
$(a_{i,j})_{1 \le i \le n, i \le j \le 2n-i}$ in which each entry is $0$, $1$ or $-1$ and the following conditions are fulfilled.
\begin{enumerate}
\item\label{QAST-1} The non-zero entries alternate in each row and column.
\item\label{QAST-2} The row sums are $1$ for rows $1,2,\ldots,n-1$, and $0$ or $1$ for row $n$.
\item\label{QAST-3} The topmost non-zero entry in each column is $1$ (if it exists).
\end{enumerate}
The set of order $n$ $\qast$s is denoted by $\qast(n)$.
\end{defi}

To construct a bijection between $\qast(n)$ and $\{A \in \dasasm(2n+1) \mid \n_{1}(A)=n+1\}$, recall from Corollary~\ref{char} (\ref{char-2}) that the objects in the latter set correspond to triangular six-vertex configurations on ${\mathcal T}_n$ such that, in each row, the leftmost and the rightmost vertical edge point upwards. Now we perform the following operations on such configurations:
\begin{enumerate}
\item For the second vertex of each row (excluding the first and the last row, which contain only degree $1$ vertices), we interchange the
orientation of the bottom vertical edge with the orientation of the left horizontal edge incident with this vertex.
\item For the penultimate vertex of each row (excluding the first and the last row), we interchange the orientation of the bottom vertical edge with the orientation of the right horizontal edge incident with this vertex.
\end{enumerate}
This leads to triangular six-vertex configurations of order $n$ such that, in each row, the leftmost horizontal edge points to the right and the rightmost horizontal edge points to the left with the exception of the rightmost horizontal edge in the penultimate row. Compare this to the triangular six-vertex configurations in Corollary~\ref{char} (\ref{char-1}) which were shown to correspond to $\ast$s of order $n$. Similarly, it can be shown that the present configurations correspond to $\qast$s of order $n$.

If we perform only the first operation above, we obtain triangular six-vertex configurations that are equivalent to the following triangular arrays. These arrays can be seen as a mixture of $\ast$s and $\dast$s.

\begin{defi}
A \emph{mixed alternating sign triangle} ($\mast$) of order $n$ is a triangular array \\
$(a_{i,j})_{1 \le i \le n, i \le j \le 2n-i}$ in which each entry is  $0$, $1$ or $-1$ and the following conditions are fulfilled.
\begin{enumerate}
\item\label{MAST-1} The non-zero entries alternate in each row and column.
\item\label{MAST-2} The column sums are $0$ for columns $n+1,\ldots,2n-1$.
\item\label{MAST-3} The first non-zero entry in each row and each column is $1$ (if it exists).
\end{enumerate}
\end{defi}

Similarly, it can be seen that $\ast(n)$ is in bijection with the set of $\mast$s of order $n$ whose bottom entry is $1$, and also in bijection with the set of $\mast$s of order $n$ whose $n$-th column has sum $0$.

Here are two related remarks.
\begin{itemize}
\item As $\ast(n) \subseteq \qast(n)$, and $\ast(n)$ is equinumerous with
the set of $n \times n$ $\asm$s, while $\qast(n)$ is equinumerous with the set of $\cspp$s in an $n \times n \times n$ box, there should exist a natural subset of the set of $\cspp$s in an $n \times n \times n$ box that has the same cardinality as the set of order $n$  $\asm$s.
\item The set of order $n$ $\qast$s can be partitioned according to the entry at the bottom and the central column sum: (i) the bottom entry is $1$, (ii) the bottom entry is $0$ and the central column sum is $0$, (iii) the bottom entry is $0$ and the central column sum is $1$. Each of the first two sets is actually equinumerous with $\ast(n)$. The set of order $n$ $\qast$s that belong to the first or third class are equinumerous with odd $\dasasm$-triangles of order $n-1$ with the property that the entries of each row---except for possibly the bottom row---sum to $1$, where possibly either or both of the boundary entries are disregarded.  Indeed, such a $\qast$ is bijectively transformed into such an odd $\dasasm$-triangle by replacing each $0$ on the diagonals by $-1$ if it is contained in a column with sum $1$.  Compare this to Corollary~\ref{char} \eqref{char-3} to see that the set of odd $\dasasm$-triangles $A$ of order $n-1$ with $\n_{0}(A)=n-1$ is a subset of the set of these $\dasasm$-triangles.
\end{itemize}

\section{Weights, local equations and the partition function}
\label{weight-sec}

We have seen that the various classes of extreme $\dasasm$s under consideration in
Theorems~\ref{maxmone}--\ref{maxzero} have an alternative interpretation in terms of certain orientations of triangular regions of the square grid.  The differences between the classes are due to the various conditions along the left, right and bottom boundary.
In this section, we define a universal generating function---in the language of physics a universal \emph{partition function}---for all of these orientations of $\mathcal{T}_n$ as well as two refinements according to the orientation of the bottom vertical edge.  These generating functions involve \emph{boundary parameters} which are sufficiently general such that the generating function of each class of extreme $\dasasm$s is a specialization of the universal generating function or one of its refinements.

\subsection{Weights} We assign a weight $\w(C)$ to each triangular six-vertex configuration $C$. This weight is the product of the vertex weights\ilse{,} to be defined next. The weight of an individual vertex depends---besides a global variable $q$---on the \emph{orientations of the surrounding edges}, on a \emph{label} which is assigned to the vertex, and, for bulk vertices, on the \emph{position of the label}.

We define
$$\bar x = x^{-1}  \quad \text{and} \quad
\sigma(x)=x-\bar x,$$
and we use this notation throughout the rest of this paper.

The vertex weights are provided in Table~\ref{weights}, with the exception of the degree $1$ vertices, which always have weight $1$.  The label of the vertex is denoted by $u$, and we assume in this table that for bulk vertices, the label is situated in the south-west corner, as follows.
$$
\scalebox{1.5}{
\psfrag{u}{\tiny $u$}
\includegraphics{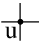}}
$$
Also elsewhere, if the position of the label of a degree $4$ vertex is not indicated, it is assumed to be in the south-west corner.
If the label appears in some other corner, one has to rotate the diagram until the label is in the south-west corner, and then determine the weight.
 The left boundary weights depend on the \emph{left boundary constants} $\alpha_L, \beta_L, \gamma_L, \delta_L$, and the right boundary weights depend on the \emph{right boundary constants} $\alpha_R, \beta_R, \gamma_R, \delta_R$. \ilse{For boundary vertices the position of the label does not matter.}

\begin{table}[ht]
$$
\begin{array}{|l|l|l|c|} \hline\rule[-0.8ex]{0ex}{3.5ex}
\text{Bulk weights} & \text{Left boundary weights} & \text{Right boundary weights} & \text{Matrix entry}  \\ \hline
 \hline\rule[-0.6ex]{-0.4ex}{3.5ex}
 \wone{u}=1 & \wlone{u}= \frac{\beta_L \, q u +  \gamma_L \, \bar q \bar u}{\sigma(q^2)}&  \wrone{u} = \frac{\beta_R \, q \bar u + \gamma_R \, \bar q u}{\sigma(q^2)} & 1 \\[1.5mm]
 \wmone{u}=1 &  \wlmone{u} = \frac{\gamma_L \, q u +  \beta_L \, \bar q \bar u}{\sigma(q^2)} &   \wrmone{u} = \frac{\gamma_R \, q \bar u + \beta_R \, \bar q u}{\sigma(q^2)} &-1 \\[1.5mm]
 \wsw{u} = \frac{\sigma(q^2 u)}{\sigma(q^4)}  & \wlin{u} = \alpha_L \, \frac{\sigma(q^2u^2)}{\sigma(q^2)} & & 0 \\[1.5mm]
  \wne{u}= \frac{\sigma(q^2 u)}{\sigma(q^4)}  & \wlout{u} = \delta_L \, \frac{\sigma(q^2u^2)} {\sigma(q^2)} & & 0  \\[1.5mm]
   \wse{u}   = \frac{\sigma(q^2 \bar u)}{\sigma(q^4)} & & \wrin{u} = \alpha_R \, \frac{\sigma(q^2 \bar u^2)}{\sigma(q^2)} & 0  \\[1.5mm]
 \wnw{u} = \frac{\sigma(q^2 \bar u)}{\sigma(q^4)} & & \wrout{u} = \delta_R \, \frac{\sigma(q^2 \bar u^2)}{\sigma(q^2)} & 0 \\[1.5mm]
    \hline
\end{array}
$$
\caption{\label{weights} Vertex weights}
\end{table}
The following observations are crucial.
\begin{itemize}
\item
For bulk vertices, the weight is unchanged when reversing the orientation of all edges.
\item For bulk vertices, rotation of the configuration or the label by $90^\circ$ is equivalent to replacing the label $u$ by $\bar u$.
\item Reflection of a local configuration in the vertical symmetry axis is equivalent to replacing
$u$ by $\bar u$ and, in addition for boundary weights, interchanging $L$ and $R$ in the boundary constants.
\end{itemize}

\begin{figure}
\scalebox{0.8}{
\psfrag{1}{\large $u_1$}
\psfrag{2}{\large$u_2$}
\psfrag{3}{\large$u_3$}
\psfrag{4}{\large$u_4$}
\psfrag{5}{\large$u_5$}
\psfrag{6}{\large$u_6$}
\psfrag{12}{\large $u_1 u_2$}
\psfrag{13}{\large $u_1 u_3$}
\psfrag{14}{\large $u_1 u_4$}
\psfrag{15}{\large $u_1 u_5$}
\psfrag{16}{\large $u_1 u_6$}
\psfrag{23}{\large $u_2 u_3$}
\psfrag{24}{\large $u_2 u_4$}
\psfrag{25}{\large $u_2 u_5$}
\psfrag{26}{\large $u_2 u_6$}
\psfrag{34}{\large $u_3 u_4$}
\psfrag{35}{\large $u_3 u_5$}
\psfrag{36}{\large $u_3 u_6$}
\psfrag{45}{\large $u_4 u_5$}
\psfrag{46}{\large $u_4 u_6$}
\psfrag{56}{\large $u_5 u_6$}
\includegraphics{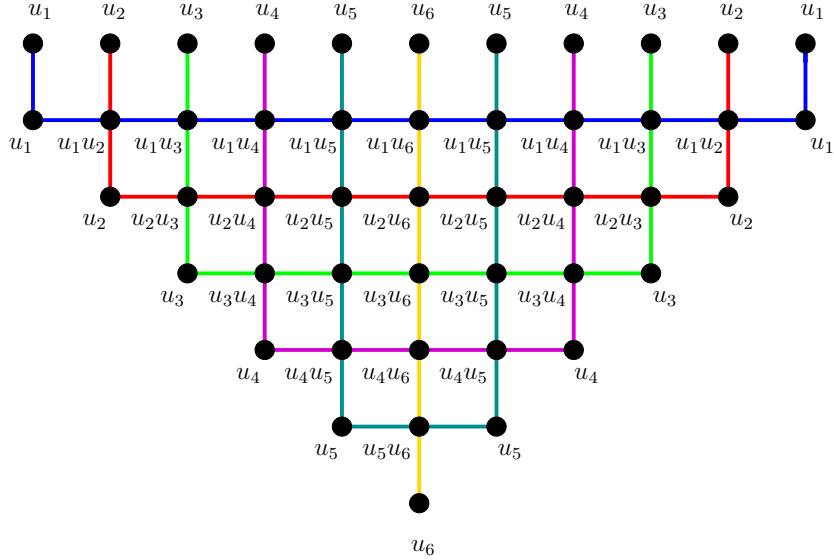}}\caption{\label{ODASASMfig} Spectral parameters.}\end{figure}

The \emph{label}  of a vertex in $\mathcal{T}_n$ is determined as follows: The path in $\mathcal{T}_n$ that is induced by the sequence of vertices corresponding to the
entries in \eqref{path} is associated with the  \emph{spectral parameter} $u_j$, see also Figure~\ref{ODASASMfig}.
The label of a vertex is then the product of the spectral parameters associated with the paths which contain that vertex. In the case of bulk vertices, there are two such paths, while the path is unique for boundary vertices.

\emph{Example.} The left configuration in Figure~\ref{implies} has the weight
\ilse{
\begin{multline*}
\delta_L \frac{\sigma(q^2 u_1^2)}{\sigma(q^2)} \cdot \frac{\sigma(q^2 u_1 u_2)}{\sigma(q^4)} \cdot \frac{\sigma(q^2 u_1 u_3)}{\sigma(q^4)} \cdot \frac{\sigma(q^2 u_1 u_4)}{\sigma(q^4)}  \cdot 1 \cdot \frac{\sigma(q^2 \bar u_1 \bar u_2)}{\sigma(q^4)}  \cdot
\delta_R \frac{\sigma(q^2 \bar u_1^2)}{\sigma(q^2)} \\ \times \frac{\beta_L q u_2 + \gamma_L \bar q \bar u_2}{\sigma(q^2)} \cdot \frac{\sigma(q^2 \bar u_2 \bar u_3)}{\sigma(q^4)} \cdot  \frac{\sigma(q^2 \bar u_2 \bar u_4)}{\sigma(q^4)} \cdot 1 \cdot
\frac{\beta_R q \bar u_2 + \gamma_R \bar q u_2}{\sigma(q^2)} \\ \times
\delta_L \frac{\sigma(q^2 u_3^2)}{\sigma(q^2)} \cdot \frac{\sigma(q^2 u_3 u_4)}{\sigma(q^4)} \cdot \frac{\beta_R q \bar u_3 + \gamma_R \bar q u_3}{\sigma(q^2)},
\end{multline*}
where rows $2,3,4$ of vertices in the graph correspond to the lines $1,2,3$ above, respectively, and the weights are read from left to right.}

\subsection{Local equations}

In this subsection, we demonstrate that the weights introduced in the previous subsection satisfy certain local equations. The most prominent local equation is the \emph{Yang--Baxter equation}, which is stated first. The schematic diagrams in this subsection have to be interpreted as follows, see for instance \eqref{ybe}. Each diagram is a graph which involves a number of external edges, for which only one endpoint is indicated with a bullet. Close to the other endpoint of such an edge, we place an $o_i$ which indicates its orientation, i.e.,``$\inw$'' or ``$\out$'', relative to the indicated endpoint of the external edge. A diagram represents the generating function of all orientations of the graph such that the external edges have the prescribed orientation, degree $4$ vertices are balanced, and the vertex weights are as given in Table~\ref{weights}, where the parameter near a vertex indicates its label.

\begin{theo}[Yang--Baxter equation, \cite{Bax82}]
\label{yang}
If $x y z = q^2$ and $o_1,o_2,\ldots,o_6 \in \{ \inw, \out \}$, then
\begin{equation}
\label{ybe}
\scalebox{0.8}{
\psfrag{x}{$x$}
\psfrag{y}{$y$}
\psfrag{z}{$z$}
\psfrag{1}{$o_6$}
\psfrag{2}{$o_5$}
\psfrag{3}{$o_4$}
\psfrag{4}{$o_3$}
\psfrag{5}{$o_2$}
\psfrag{6}{$o_1$}
\includegraphics{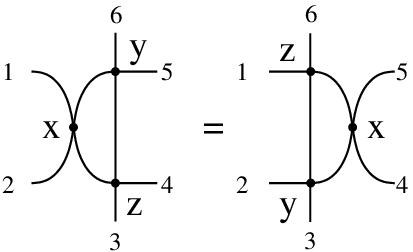}} \qquad.
\end{equation}
\end{theo}

The theorem can be proven by verifying that, for each choice of orientations of the external edges, \eqref{ybe} is either trivial or reduces to a simple identity satisfied by the bulk weights in Table~\ref{weights}.

To deal with triangular six-vertex configurations, we also need left and right
forms of the reflection equation. The reflection equation was introduced (and applied to six-vertex model bulk weights) by Cherednik~\cite[Eq.~(10)]{Che84}, with important further results being obtained by Sklyanin~\cite{Skl88}.

\begin{theo}[Left and right reflection equations]
\label{re}
Suppose $x=q^2 \bar u v, y= u  v$ and the bulk weights are as given in Table~\ref{weights}. Then
\begin{equation}
\label{lre}
\scalebox{0.8}{
\psfrag{x}{$x$}
\psfrag{y}{$y$}
\psfrag{u}{$u$}
\psfrag{v}{$v$}
\psfrag{1}{$o_1$}
\psfrag{2}{$o_2$}
\psfrag{3}{$o_3$}
\psfrag{4}{$o_4$}
\includegraphics{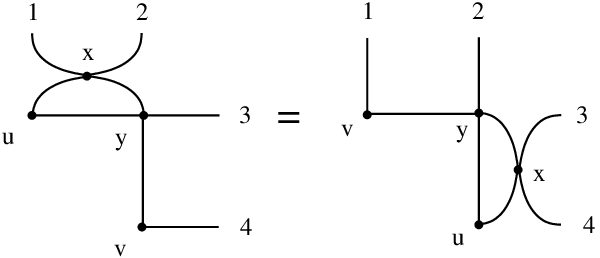}} \qquad
\end{equation}
for all $o_1,o_2,o_3,o_4 \in \{\inw,\out\}$ if and only if there exist parameters
$\alpha_L,\beta_L,\gamma_L,\delta_L$ independent of $u$ and a function $f(u)$ such that
$$
\wlone{u}= \frac{\beta_L \, q u +  \gamma_L \, \bar q \bar u}{\sigma(q^2)} f(u), \qquad
\wlmone{u}= \frac{\gamma_L \, q u +  \beta_L \, \bar q \bar u}{\sigma(q^2)} f(u),
$$
and
$$
\wlin{u} = \alpha_L \, \frac{\sigma(q^2u^2)}{\sigma(q^2)} f(u), \qquad
\wlout{u} = \delta_L \, \frac{\sigma(q^2u^2)} {\sigma(q^2)} f(u),
$$
\ilse{i.e., up to a multiplicative factor, the left boundary weights are chosen as in Table~\ref{weights}.}
Similarly, if $x=q^2 \bar u v, y= \bar u  \bar v$ and the bulk weights are chosen as in Table~\ref{weights}, then
\begin{equation}
\label{rre}
\scalebox{0.8}{
\psfrag{x}{$x$}
\psfrag{y}{$y$}
\psfrag{u}{$u$}
\psfrag{v}{$v$}
\psfrag{1}{$o_1$}
\psfrag{2}{$o_2$}
\psfrag{3}{$o_3$}
\psfrag{4}{$o_4$}
\includegraphics{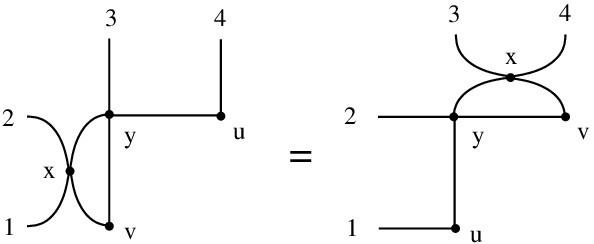}}
\end{equation}
for all $o_1,o_2,o_3,o_4 \in \{\inw,\out\}$ if and only if $\wrone{u},\wrmone{u},\wrin{u}, \wrout{u}$ are chosen---up to multiplication by a function $g(u)$---as in Table~\ref{weights}.
\end{theo}

The result was obtained, independently, by de Vega and
Gonz\'{a}lez-Ruiz~\cite[Eq.~(15)]{DevGon93}, and Ghoshal and Zamolodchikov~\cite[Eq.~(5.12)]{GhoZam94}. We also provide a proof below.

\begin{proof}
First we show that every solution of \eqref{lre} is of the form as given in the statement of the theorem.
Set $o_1=o_4=\inw$, $o_2=o_3=\out$ in \eqref{lre} and obtain (after subtraction of an identical term from each side)
\begin{multline}
\label{error}
\wnw{x} \wlone{u} \wlone{v} + \wne{y} \wlmone{u} \wlone{v}  \\
= \wse{x} \wlmone{u} \wlmone{v}
+ \wne{y} \wlone{u} \wlmone{v}.
\end{multline}
The equation shows in particular that $\wlone{u} \equiv 0$ if and only if
$\wlmone{u} \equiv 0$. We first assume $\wlone{u} \not\equiv 0$ and
$\wlmone{u} \not\equiv 0$. We divide \eqref{error} by $\wlmone{u} \wlmone{v}$ and deduce
$$
\frac{\wlone{u}}{\wlmone{u}} \left( \wnw{x} \frac{\wlone{v}}{\wlmone{v}} - \wne{y} \right)=
- \wne{y} \frac{\wlone{v}}{\wlmone{v}} + \wse{x}.
$$
After plugging in the definitions for the bulk weights and setting
$x=q^2 \bar u v$, $y= u v$, we can deduce that
\begin{align*}
- \wne{y} \frac{\wlone{v}}{\wlmone{v}} + \wse{x} &= \beta_L q u + \gamma_L \bar q \bar u,  \\
\wnw{x} \frac{\wlone{v}}{\wlmone{v}} - \wne{y} &= \gamma_L q u + \beta_L \bar q \bar u,
\end{align*}
with $\beta_L = \frac{\bar q \bar v}{\sigma(q^4)} - \frac{q v}{\sigma(q^4)} \frac{\wlone{v}}{\wlmone{v}}$ and $\gamma_L = \frac{\bar q \bar v}{\sigma(q^4)} \frac{\wlone{v}}{\wlmone{v}} - \frac{q v}{\sigma(q^4)}$.
Since these are both parameters independent of $u$ and at least one of them is non-zero, we can deduce that $\wlone{u}$ and $\wlmone{u}$ are of the form stated in the theorem.

In \eqref{lre}, we now choose $o_1=o_2=o_3=\inw$, $o_4=\out$ and obtain
\begin{multline*}
\wne{x} \wsw{y} \wlin{u} \wlmone{v}  \\
= \wnw{x} \wse{y}  \wlin{u} \wlmone{v} +
\wnw{x} \wlone{u} \wlin{v}  + \wsw{y} \wlmone{u} \wlin{v}.
\end{multline*}
In this identity, every quantity is known (in the case of $\wlone{u}$ and $\wlmone{u}$ up to $\beta_L, \gamma_L, f(u)$),
except for $\wlin{u}$ and $\wlin{v}$.
The identity thereby becomes
$$(\beta_L \bar v \bar q + \gamma_L q v) \left( \wlin{u} \sigma(q^2 v^2) f(v) - \wlin{v} \sigma(q^2 v^2) f(u) \right)=0$$
from which we see that $\wlin{u}$ is also of the form given in the theorem. In order to obtain the expression for $\wlout{u}$, choose
$o_1=o_2=\inw$ and $o_3=o_4=\out$ in \eqref{lre}, which gives
\begin{equation}
\label{quotient}
\wlin{u} \wlout{v} = \wlout{u} \wlin{v}.
\end{equation}

If, on the other hand, $\wlone{u} \equiv 0$ and $\wlmone{u} \equiv 0$, we choose $\beta_L=\gamma_L=0$ and the assertion also follows simply from \eqref{quotient}.

Conversely, it is straightforward to check that the solution stated in the theorem satisfies \eqref{lre}
for all $o_1,o_2,o_3,o_4 \in \{\inw,\out\}$.

The right reflection equation follows from the left reflection equation
by substituting $u \mapsto \bar v, v \mapsto \bar u$. \end{proof}

There are two additional simple rules that are necessary. For $o_1,o_2,o_3,o_4 \in
\{\inw,\out\}$, we have
\begin{equation}
\label{trivial}
\wwnes{q^2} = \delta_{o_1,\widetilde o_2} \delta_{o_3, \widetilde o_4},
\end{equation}
and
\begin{equation}
\label{rue}
\raisebox{-1.2cm}{\scalebox{0.8}{
\psfrag{1}{\large $o_1$}
\psfrag{2}{\large $o_2$}
\psfrag{3}{\large $q u$}
\psfrag{4}{\large $q \bar u$}
\includegraphics{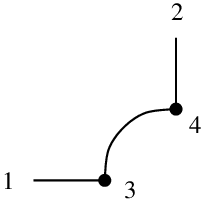}}} = \left(\wrone{q u} \wrone{q \bar u} + \wrin{q u} \wrout{q \bar u} \right) \delta_{o_1,\widetilde o_2},
\end{equation}
where $\widetilde \inw = \out$, $\widetilde \out = \inw$ and $\delta$ is the Kronecker delta.

\subsection{The (universal) partition function and its specializations}

The universal partition function is the generating function of all permissible orientations of $\mathcal{T}_n$ with respect to the weight function we have defined, that is
$$
 \sum_{C} \w(C) =:Z_n(u_1,\ldots,u_n;u_{n+1}),
$$
where the sum is over all orientations $C$ of $\mathcal{T}_n$ in which each bulk vertex has two incoming and two outgoing edges and the top edges point up.

Moreover, $Z^{\uparrow}_n(u_1,\ldots,u_n; u_{n+1})$ and $Z^{\downarrow}_n(u_1,\ldots,u_n; u_{n+1})$ denote the restrictions of this sum to configurations in which the bottom vertical edge points up or down, respectively. In \cite[Eq. (24)]{DASASM}, it was shown that the two refinements are expressible in terms of the full partition function as follows.

\begin{lem}
\label{updown}
The universal partition function and its refinements satisfy
\begin{align*}
Z^{\uparrow}_n(u_1,\ldots,u_n;u_{n+1}) = \frac{1}{2} \left( Z_n(u_1,\ldots,u_n;u_{n+1}) +
(-1)^{n} Z_n(u_1,\ldots,u_{n};-u_{n+1}) \right), \\
Z^{\downarrow}_n(u_1,\ldots,u_n;u_{n+1}) = \frac{1}{2} \left( Z_n(u_1,\ldots,u_n;u_{n+1}) -
(-1)^{n} Z_n(u_1,\ldots,u_{n};-u_{n+1}) \right).
\end{align*}
\end{lem}

In that paper we assumed more special boundary weights, but the proof of the lemma is independent of the form of the boundary weights, since the argument involves only vertex weights that depend on $u_{n+1}$, and this parameter appears in none of the boundary weights.

Next, we introduce three specializations of the spectral parameters $u_1,\ldots,u_{n+1}$, the global parameter $q$ and the boundary parameters $\alpha_L, \beta_L, \gamma_L, \delta_L, \alpha_R, \beta_R, \gamma_R, \delta_R$, for which $Z_n(u_1,\ldots,u_n;u_{n+1})$ counts extreme $\dasasm$s in Theorems \ref{maxmone}, \ref{maxone} and \ref{maxzero}, respectively. In Sections~\ref{ast}--\ref{qoosasm}, we then derive determinant or Pfaffian expressions for these specializations of the partition function.\footnote{In a forthcoming paper \cite{DSASM}, we will study the universal partition function of diagonally symmetric $\asm$s using the weights in Table~\ref{weights}. In contrast to $Z_{n}(u_1,\ldots,u_n;u_{n+1})$, we are able to derive an expression for the universal partition function in this case.} \ilse{We then introduce a fourth specialization which applies to the case of all $\dasasm$s, as studied in \cite{DASASM}.}

First observe that if we specialize $u_j=1$ for all $j$ and $q=e^{\frac{i \pi}{6}}=\frac{\sqrt{3}+i}{2}$, then all bulk weights are equal to $1$.

\subsubsection{The $\ast$ specialization}
\label{ast-spec}
By Corollary~\ref{char} \eqref{char-1}, we need to have
\begin{gather*}
\wlone{1}=\wlin{1}=\wrone{1}=\wrin{1}=0, \\
\wlmone{1}=\wlout{1}=\wrmone{1}=\wrout{1}=1,
\end{gather*}
if $q=e^{\frac{i \pi}{6}}$. However, \ilse{in order to prove a generalization of
Theorem~\ref{maxmone}, namely Theorem~\ref{maxmone-gen}, it will be crucial to introduce an additional global parameter $p$ such that the  following is satisfied (setting $p=1$ then leads to the above conditions):}
\begin{equation}
\label{ast1g}
\wlone{p}=\wlin{p}=\wrone{p}=\wrin{p}=0,
\end{equation}
\begin{equation}
\label{ast2g}
\wlmone{p}=\wlout{p}=\frac{\sigma(p^2 q^2)}{\sigma(q^2)} \quad \text{and} \quad \wrmone{p}=\wrout{p}= \frac{\sigma(\bar p^2 q^2)}{\sigma(q^2)}.
\end{equation}
This is satisfied if we assign the following values to the boundary parameters. Note that this is true even if we do not specialize $q$. This is also the case for the two other  specializations.
$$
\begin{array}{|c|c|c|c|c|c|c|c|c|} \hline
\multirow{2}{*}{$\ast$} & \alpha_L & \beta_L & \gamma_L & \delta_L &
\alpha_R & \beta_R & \gamma_R & \delta_R \\ \cline{2-9}
& 0 & - \bar p \bar q & p q  & 1 & 0 & - p \bar q & \bar p q  & 1 \\ \hline
\end{array}
$$
The explicit boundary weights are therefore as follows.
\begin{align*}
\wlone{u}&= \frac{\sigma(p \bar u)}{\sigma(q^2)}, &\wlmone{u}&= \frac{\sigma(p q^2 u)}{\sigma(q^2)}, &\wlin{u}&= 0, &\wlout{u}&= \frac{\sigma(q^2 u^2)}{\sigma(q^2)},\\
\wrone{u}&= \frac{\sigma(\bar p u)}{\sigma(q^2)}, &\wrmone{u}&= \frac{\sigma(\bar p q^2 \bar u)}{\sigma(q^2)}, &\wrin{u}&= 0, &\wrout{u}&= \frac{\sigma(q^2 \bar u^2)}{\sigma(q^2)}
\end{align*}
We denote this specialization of the universal partition function as $Z_{n}(u_1,\ldots,u_n;u_{n+1})_{\ast}$
and record that
$$
\left. Z_{n}(1,\ldots, 1;1)_{\ast} \right|_{p=1,q=e^{\frac{i \pi}{6}}} = \left|\hspace{-0mm}\ast(n)\right|.
$$

\subsubsection{The $\qast$ specialization}
\label{dast-spec}

By Corollary~\ref{char} \eqref{char-2}, we require
\begin{gather*}
\wlmone{1}=\wlin{1}=\wrmone{1}=\wrin{1}=0, \\
\wlone{1}=\wlout{1}=\wrone{1}=\wrout{1}=1,
\end{gather*}
and the bottom vertical edge needs to point up, that is we need to work with
$Z^{\uparrow}_{n}(u_1,\ldots,u_n;u_{n+1})$. Again, we introduce a parameter $p$ that is set to
$1$ for proving Theorem~\ref{maxone}, and require
$$\wlmone{p}=\wlin{p}=\wrmone{p}=\wrin{p}=0,$$ as well as
$$\wlone{p}=\wlout{p}= \frac{\sigma(p^2 q^2)}{\sigma(q^2)} \quad \text{and} \quad
\wrone{p}=\wrout{p}=\frac{\sigma(\bar p^2 q^2)}{\sigma(q^2)}.$$
We choose the boundary parameters as follows.
$$
\begin{array}{|c|c|c|c|c|c|c|c|c|} \hline
\multirow{2}{*}{$\qast$} & \alpha_L & \beta_L & \gamma_L & \delta_L &
\alpha_R & \beta_R & \gamma_R & \delta_R \\ \cline{2-9}
& 0 & p q  & - \bar p \bar q  & 1 & 0 &  \bar p q & - p \bar q  & 1 \\ \hline
\end{array}
$$
In this case, the explicit boundary weights are
\begin{align*}
\wlone{u}&= \frac{\sigma(p q^2 u)}{\sigma(q^2)}, &\wlmone{u}&= \frac{\sigma(p \bar u)}{\sigma(q^2)}, &\wlin{u}& = 0, &\wlout{u}& = \frac{\sigma(q^2 u^2)}{\sigma(q^2)}, \\
\wrone{u}&= \frac{\sigma(\bar p q^2 \bar u)}{\sigma(q^2)}, &\wrmone{u}&= \frac{\sigma(\bar p u)}{\sigma(q^2)}, &\wrin{u}& = 0, &\wrout{u}& = \frac{\sigma(q^2 \bar u^2)}{\sigma(q^2)}.
\end{align*}
We denote this specialization of the universal partition function as $Z^{\uparrow}_{n}(u_1,\ldots,u_n;u_{n+1})_{\qast}$ and obtain that
$$
\left. Z^{\uparrow}_{n}(1,\ldots,1;1)_{\qast} \right|_{p=1,q=e^{\frac{i \pi}{6}}} = \left| \qast(n) \right|.
$$

\subsubsection{The $\qoosasm$ specialization}

\ilse{Recalling that $A \in \dasasm(2n+1)$ satisfies $\n_0(A)=2n$ if and only if all entries on the diagonals are zero except for the central entry of $A$ (see the proof of Proposition~\ref{bounds} in Section~\ref{bounds})}, we need to have
\begin{gather*}
\wlone{1}=\wlmone{1}=\wrone{1}=\wrmone{1}=0, \\
\wlin{1}=\wlout{1}=\wrin{1}=\wrout{1}=1,
\end{gather*}
and so it is natural to choose the following.
$$
\begin{array}{|c|c|c|c|c|} \hline
\multirow{2}{*}{$\qoosasm$} & \alpha_L=\alpha_R & \beta_L=\beta_R & \gamma_L=\gamma_R & \delta_L=\delta_R \\ \cline{2-5}
& 1 & 0 & 0 & 1 \\ \hline
\end{array}
$$
Here, the left boundary weights are
\begin{gather*}
\wlone{u}=\wlmone{u}=0, \quad
\wlin{u}=\wlout{u}= \frac{\sigma(q^2 u^2)}{\sigma(q^2)}, \\
\wrone{u}=\wrmone{u}=0,   \quad
 \wrin{u}=\wrout{u}=\frac{\sigma(q^2 \bar u^2)}{\sigma(q^2)}.
\end{gather*}
In this case, the left boundary weight in row $i$ of a configuration with non-zero weight is $\frac{\sigma(q^2 u_i^2)}{\sigma(q^2)}$, while the right boundary weight in row $i$ is $\frac{\sigma(q^2 \bar u_i^2)}{\sigma(q^2)}$. This implies that the universal partition function has the factor $\prod\limits_{i=1}^{n} \frac{\sigma(q^2 u_i^2) \sigma(q^2 \bar u_i^2)}{\sigma(q^2)^2}$ in this specialization. We cancel this factor and denote the reduced specialization as $Z_{n}(u_1,\ldots,u_n;u_{n+1})_{\qoosasm}$.
Therefore,
$$
\left. Z_{n}(1,\ldots,1;1)_{\qoosasm} \right|_{q=e^{\frac{i \pi}{6}}} = | \qoosasm(2n+1) |.
$$
Note that this normalization of the universal partition function could alternatively be achieved by taking the functions $f(u)$ and $g(u)$ in
Theorem~\ref{re} to be $f(u)=\sigma(q^2)/\sigma(q^2 u^2)$ and $g(u)=\sigma(q^2)/\sigma(q^2 \bar u^2)$.
Kuperberg used in \cite{KuperbergRoof} analogous specializations
to study $\operatorname{OSASM}$s and $\oosasm$s of order $4n$.

\subsubsection{The $\dasasm$ specialization}
\label{dasasm-spec}

If we require
$$\wlone{1}=\wlmone{1}=\wlin{1}=\wlout{1}=\wrone{1}=\wrmone{1}= \wrin{1}=\wrout{1}=1
$$
\ilse{in order to} count all odd-order $\dasasm$s as was done in \cite{DASASM}, then we may choose $\alpha_L=\alpha_R=\delta_L=\delta_R=1$ and
$\beta_L=\beta_R=\gamma_L=\gamma_R=\sigma(q)$. Here, the left boundary weight in row $i$ is divisible by $q u_i + \bar q \bar u_i$, while the right boundary weight is divisible by $q \bar u_i + \bar q u_i$. Thus we may divide this specialization by
$$
\prod_{i=1}^{n} \frac{\sigma(q)^2 (q u_i + \bar q \bar u_i)(q \bar u_i + \bar q u_i)}{\sigma(q^2)^2}
$$
and still obtain a Laurent polynomial in $u_1,u_2,\ldots,u_{n+1}$. (Alternatively, this normalization could be achieved by using certain choices for the functions $f(u)$ and $g(u)$ in Theorem~\ref{re}.)
This normalization was used in \cite{DASASM}.

\section{Characterization of the partition function}
\label{char-part}

\subsection{Some properties of the partition function}

\ilse{One of the main reasons for using the vertex weights as given in Table~\ref{weights} is that this choice ensures the symmetry of $Z_n(u_1,\ldots,u_n;u_{n+1})$ in $u_1,\ldots,u_n$ as discussed next. Here we rely heavily on previous work that is presented in \cite{DASASM}. At the end of the previous section (see Subsection~\ref{dasasm-spec}) we have shown that the weights in Table~\ref{weights} are (up to an irrelevant factor) generalizations of the weights used in \cite{DASASM}. In \cite[Proposition~12]{DASASM} it was demonstrated how these more special weights imply the symmetry of the partition function. However, this proof of symmetry relies only on the properties of the weights given in Theorems~\ref{yang} and \ref{re}, and since these properties also hold for the more general weights in Table~\ref{weights}, the symmetry follows also in this general setting.}
By Lemma~\ref{updown}, we \ilse{then} have symmetry also for the two refined partition functions.

\begin{theo}[Symmetry]
\label{sym}
The partition function $Z_{n}(u_1,\ldots,u_n;u_{n+1})$, and its two refinements
$Z_n^{\uparrow}, Z_n^{\downarrow}$
are symmetric in $u_1,u_2,\ldots,u_n$.
\end{theo}

Moreover, we have the following invariance.

\begin{prop}
\label{inv}
The partition function $Z_{n}(u_1,\ldots,u_n;u_{n+1})$ and its two refinements $Z_n^{\uparrow}, Z_n^{\downarrow}$ are
invariant when simultaneously replacing
$(u_1,\ldots,u_{n+1})$ with $(\bar u_1, \ldots, \bar u_{n+1})$
and interchanging left and right boundary constants, i.e.,
$\alpha_L \leftrightarrow \alpha_R, \beta_L \leftrightarrow \beta_R, \gamma_L \leftrightarrow \gamma_R, \delta_L \leftrightarrow \delta_R$.
\end{prop}

\begin{proof}
This follows from reflecting the configuration in the vertical symmetry axis as the reflection causes for each individual vertex the replacement of the label by its reciprocal and, for boundary weights, the interchanging of respective left and right boundary constants \ilse{(which is precisely the fact displayed at the third bullet point after Table~\ref{weights})}.
\end{proof}

Finally, $Z_n, Z^{\uparrow}_n, Z^{\downarrow}_n$ are Laurent polynomials in $u_1,u_2,\ldots,u_{n+1}$ that are even in each $u_i$, $i=1,\ldots,n$.
We need the following notion: Suppose $\sum_{i=o}^{d} a_i u^i$ is a Laurent polynomial in $u$, and $a_o \not=0, a_d \not=0$. Then $d$ is the \emph{degree} of the Laurent polynomial and $o$ is its \emph{order}.

\begin{prop}
\label{even}
\begin{enumerate}
\item\label{even-1}
For each $i=1,2,\ldots,n$, the partition function $Z_{n}(u_1,\ldots,u_n;u_{n+1})$ and its two refinements $Z_n^{\uparrow}, Z_n^{\downarrow}$ are even in $u_i$.
\item\label{even-2}
  The partition function $Z_{n}(u_1,\ldots,u_n;u_{n+1})$ and its two refinements $Z_n^{\uparrow}, Z_n^{\downarrow}$  are Laurent polynomials in $u_1,\ldots,u_{n+1}$ of degree at most $n$ and order at least $-n$ in $u_{n+1}$, and of degree at most $2n+2$ and order at least $-2n-2$ in $u_{i}$, for each $i=1,2,\ldots,n$.
\end{enumerate}
\end{prop}

\begin{proof}
\emph{\eqref{even-1}:}
By Theorem~\ref{sym}, it suffices to show the assertion for $i=1$.

We show that the contribution of each configuration to the partition function is even in $u_1$. Distinguish between the cases whether or not
the unique $1$ in the top row of the $\dasasm$ is on the (left or right) boundary. Clearly, only the weights of the top bulk vertices and the two top boundary vertices can involve the spectral parameter $u_1$, and if the $1$ is on the left boundary, then the contribution of the first row to the weight  is$$
\frac{\beta_L \, q u_1 +  \gamma_L \, \bar q \bar u_1}{\sigma(q^2)}
\frac{\delta_R \sigma(q^2 \bar u_1^2)}{\sigma(q^2)}
\frac{\sigma(q^2 \bar u_1 \bar u_{n+1})}{\sigma(q^4)}\prod_{i=2}^{n} \frac{\sigma(q^2 \bar u_1 \bar u_i)^2}{\sigma(q^4)^2},
$$
while the contribution is
$$
\frac{\delta_L \sigma(q^2 u_1^2)}{\sigma(q^2)}
\frac{\beta_R q \bar u_1 + \gamma_R \bar q u_1}{\sigma(q^2)}
\frac{\sigma(q^2 u_1  u_{n+1})}{\sigma(q^4)}\prod_{i=2}^{n} \frac{\sigma(q^2 u_1 u_i)^2}{\sigma(q^4)^2},
$$
if the $1$ is on the right boundary. Both expressions are obviously even in $u_1$.
If the unique $1$ in the top row is not on the boundary, then the contribution of the top bulk vertices and the two top boundary vertices is the product of $\frac{\delta_L \sigma(q^2 u_1^2)}{\sigma(q^2)}  \frac{\delta_R \sigma(q^2 \bar u_1^2)}{\sigma(q^2)} $ and an even number ($=2n-2$) of $\frac{\sigma(q^2 u_1 u_i)}{\sigma(q^4)}$'s and $\frac{\sigma(q^2 \bar u_1 \bar u_i)}{\sigma(q^4)}$'s, $i=2,3,\ldots,n+1$,
and this product is also even in $u_1$.

\emph{\eqref{even-2}:} The bounds for the degree and order in $u_i$, $i=1,2,\ldots,n$, follow from the discussion in the proof of \eqref{even-1}.  As for the degree and order in $u_{n+1}$, note that there are precisely $n$ vertices of $\mathcal{T}_n$ whose weight can involve $u_{n+1}$ (i.e., the bulk vertices of the central column). The degree of the weight in $u_{n+1}$ of such a vertex is at most $1$ and the order is at least $-1$.
\end{proof}

\subsection{Characterization through evaluations in $u_{n+1}$}

In the following lemma we show that the evaluation of the universal partition function
$Z_n(u_1,\ldots,u_n; u_{n+1})$ at $u_{n+1} = q^2 \bar u_1$ is expressible in terms of the order $n-1$ universal partition function.

\begin{lem}
\label{eval}
For any $n=1,2,3,\ldots$,
\begin{align}
\nonumber
Z_n(u_1,& \ldots,u_n; q^2 \bar u_1) \\
\nonumber
&= \frac{1}{2}
\bigg\{ \left[ \wrone{u_1}+ \wrout{u_1}+\wrin{u_1} +  \wrmone{u_1} \right]
Z_{n-1}(u_2,\ldots,u_n; u_1)  \\
\label{full}
 &  \,\,\, +  (-1)^{n+1} \left[ \wrone{u_1}+ \wrout{u_1}- \wrin{u_1} -  \wrmone{u_1} \right]  Z_{n-1}(u_2,\ldots,u_n; - u_1) \bigg\} \Omega_n,
\end{align}
where
$$
\Omega_n = \wlout{u_1}
\prod_{i=2}^{n} \wne{u_1 u_i} \wne{q^2 \bar u_1 u_i}.
$$
Concerning the two refinements, we have
\begin{align}
\label{up}
Z^{\uparrow}_{n}(u_1,\ldots,u_n; q^2 \bar u_1)
&=
 \left(  \wrone{u_1} Z^{\uparrow}_{n-1}(u_2,\ldots,u_n; u_1) + \wrin{u_1} Z^{\downarrow}_{n-1}(u_2,\ldots,u_n;  u_1) \right) \Omega_n, \\
\label{down}
Z^{\downarrow}_{n}(u_1,\ldots,u_n; q^2 \bar u_1)
 & =
\left( \wrout{u_1} Z^{\uparrow}_{n-1}(u_2,\ldots,u_n; u_1)
  +  \wrmone{u_1} Z^{\downarrow}_{n-1}(u_2,\ldots,u_n;  u_1) \right) \Omega_n.
\end{align}

\end{lem}

\begin{proof} We prove \eqref{up} and illustrate the proof with the help of the case $n=4$. By definition, the left-hand side is the generating function of the following graph where $u_{n+1} = q^2 \bar u_1$.
\begin{center}
\scalebox{0.4}{
\psfrag{1}{\huge $u_1$}
\psfrag{2}{\huge $u_2$}
\psfrag{3}{\huge $u_3$}
\psfrag{4}{\huge $u_4$}
\psfrag{12}{\huge $u_1 u_2$}
\psfrag{13}{\huge $u_1 u_3$}
\psfrag{23}{\huge $u_2 u_3$}
\psfrag{14}{\huge $u_1 u_4$}
\psfrag{24}{\huge $u_2 u_4$}
\psfrag{34}{\huge $u_3 u_4$}
\psfrag{15}{\huge $u_1 u_5$}
\psfrag{25}{\huge $u_2 u_5$}
\psfrag{35}{\huge $u_3 u_5$}
\psfrag{45}{\huge $u_4 u_5$}
\includegraphics{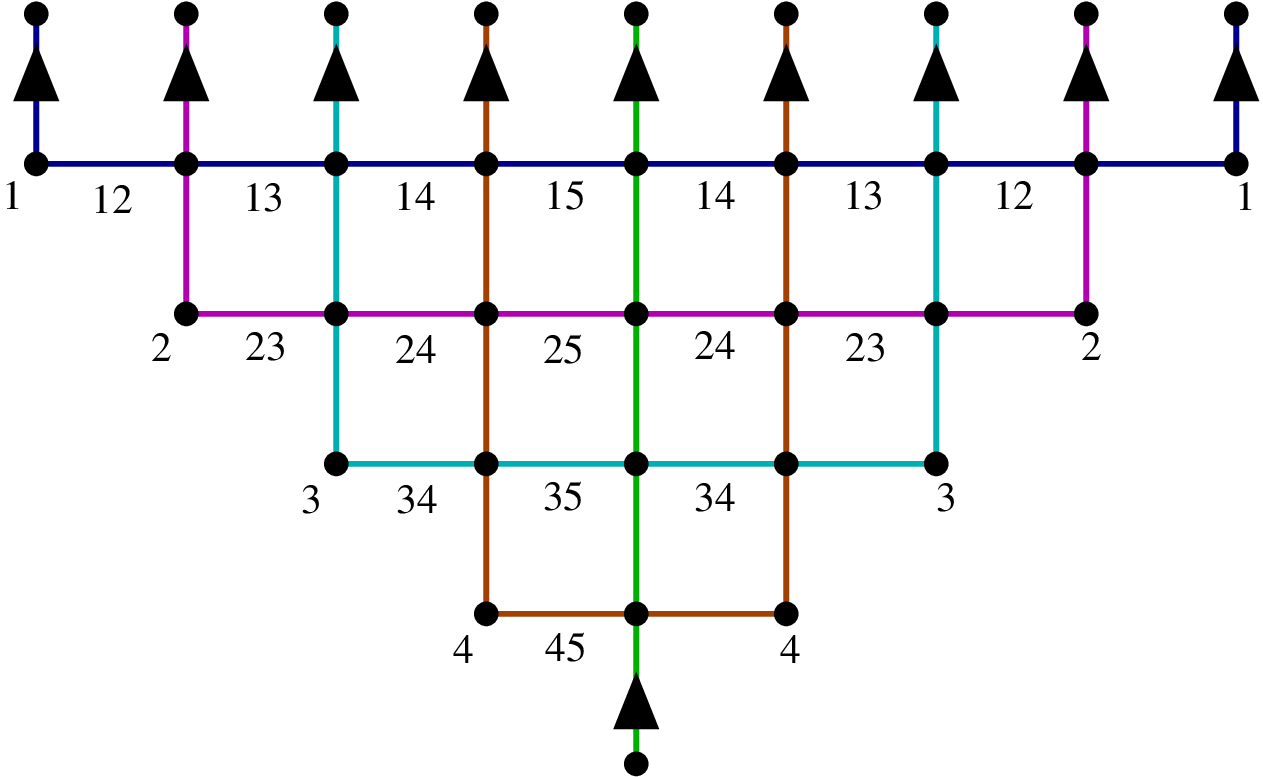}}
\end{center}
The label of the top bulk vertex in the central column is $u_1 u_{n+1} = q^2$. Using \eqref{trivial} \ilse{at this vertex} and observing that the local configurations in the left half of the top row are now forced, we see that this is equal to the generating function of the following graph.
\begin{center}
\scalebox{0.4}{
\psfrag{1}{\huge $u_1$}
\psfrag{2}{\huge $u_2$}
\psfrag{3}{\huge $u_3$}
\psfrag{4}{\huge $u_4$}
\psfrag{12}{\huge $u_1 u_2$}
\psfrag{13}{\huge $u_1 u_3$}
\psfrag{23}{\huge $u_2 u_3$}
\psfrag{14}{\huge $u_1 u_4$}
\psfrag{24}{\huge $u_2 u_4$}
\psfrag{34}{\huge $u_3 u_4$}
\psfrag{15}{\huge $u_1 u_5$}
\psfrag{25}{\huge $u_2 u_5$}
\psfrag{35}{\huge $u_3 u_5$}
\psfrag{45}{\huge $u_4 u_5$}
\includegraphics{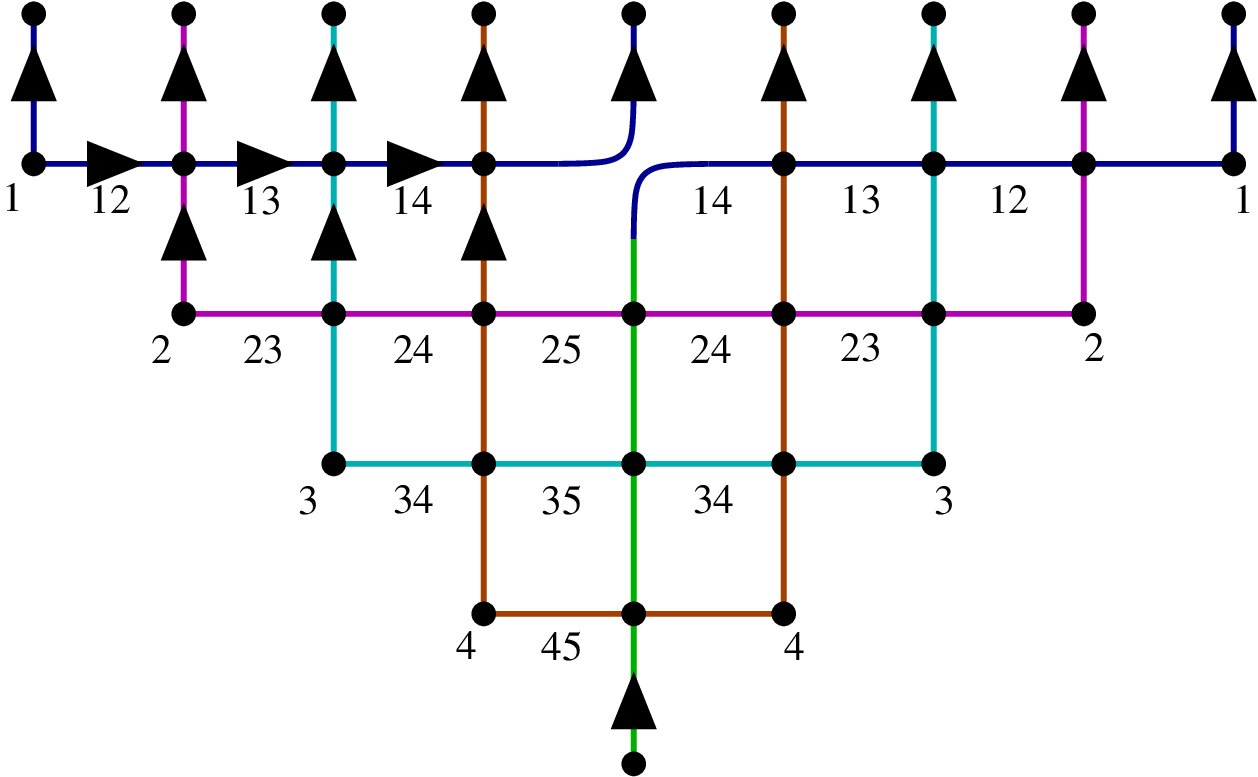}}
\end{center}
The $n$ vertices with fixed local configuration in the left half of the top row contribute
$$
\wlout{u_1} \prod_{i=2}^{n} \wne{u_1 u_i}
$$
to the weight. We delete these vertices and move the right half of the graph down.
\begin{center}
\scalebox{0.4}{
\psfrag{1}{\huge $u_1$}
\psfrag{2}{\huge $u_2$}
\psfrag{3}{\huge $u_3$}
\psfrag{4}{\huge $u_4$}
\psfrag{12}{\huge $u_1 u_2$}
\psfrag{13}{\huge $u_1 u_3$}
\psfrag{23}{\huge $u_2 u_3$}
\psfrag{14}{\huge $u_1 u_4$}
\psfrag{24}{\huge $u_2 u_4$}
\psfrag{34}{\huge $u_3 u_4$}
\psfrag{15}{\huge $u_1 u_5$}
\psfrag{25}{\huge $u_2 u_5$}
\psfrag{35}{\huge $u_3 u_5$}
\psfrag{45}{\huge $u_4 u_5$}
\includegraphics{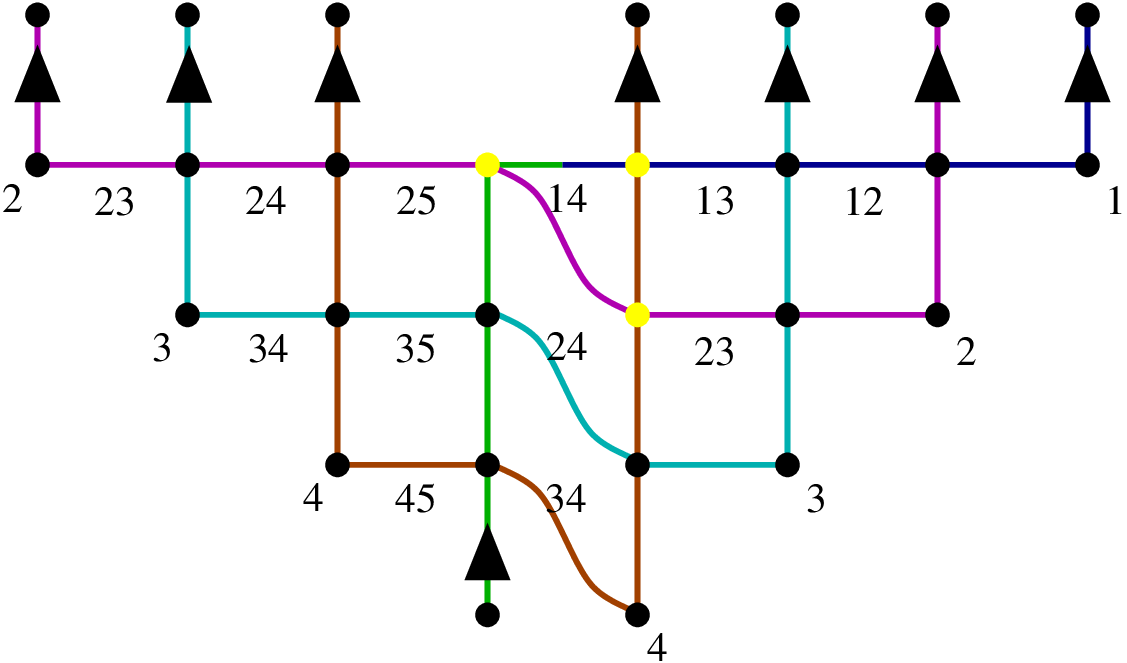}}
\end{center}
Now we apply the Yang--Baxter equation \eqref{ybe} to the vertices with label $u_2 u_{n+1}$, $u_1 u_n$ and to the right vertex with label $u_2 u_n$ (indicated with yellow in our example). This is permissible since $u_{n+1} = q^2 \bar u_1$. We repeatedly apply the Yang--Baxter equation until we reach the right boundary, and then apply the right reflection equation \eqref{rre}. We then untangle the crossing and gain the
factor $\wsw{u_2 u_{n+1}} = \wne{u_2 u_{n+1}}=\w_2$. In our example, the procedure is as follows.
\begin{center}
\scalebox{0.35}{
\psfrag{1}{\huge $u_1$}
\psfrag{2}{\huge $u_2$}
\psfrag{3}{\huge $u_3$}
\psfrag{12}{\huge $u_1 u_2$}
\psfrag{13}{\huge $u_1 u_3$}
\psfrag{23}{\huge $u_2 u_3$}
\psfrag{14}{\huge $u_1 u_4$}
\psfrag{24}{\huge $u_2 u_4$}
\psfrag{34}{\huge $u_3 u_4$}
\psfrag{4}{\huge $u_4$}
\psfrag{15}{\huge $u_1 u_5$}
\psfrag{25}{\huge $u_2 u_5$}
\psfrag{35}{\huge $u_3 u_5$}
\psfrag{45}{\huge $u_4 u_5$}
\psfrag{=}{\Huge $=$}
\psfrag{= W}{\Huge $=\w_2$}
\includegraphics{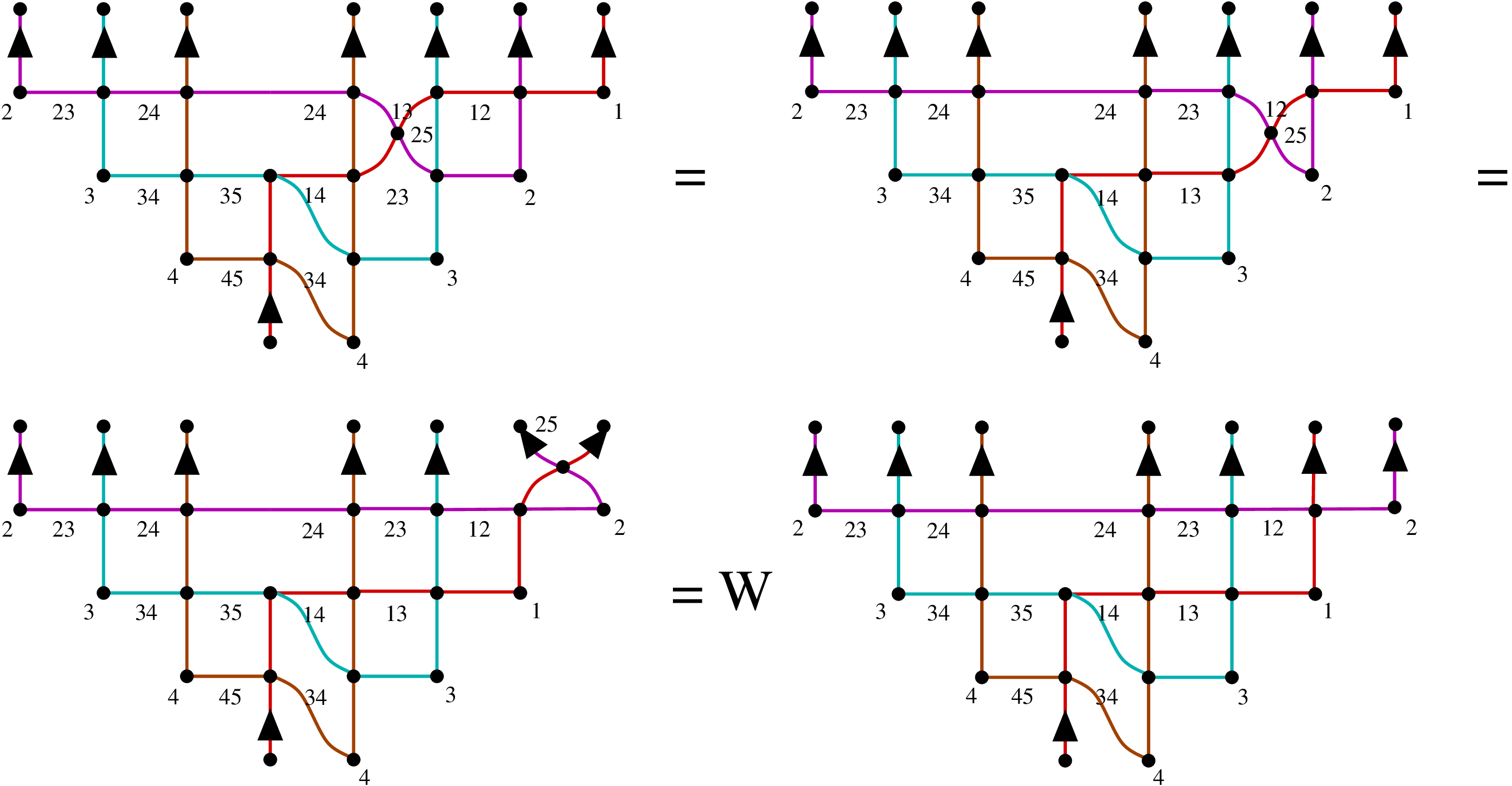}}
\end{center}
Next we apply the Yang--Baxter equation \eqref{ybe} to the vertices with label $u_3 u_{n+1}$, $u_1 u_n$ and to the right vertex with label $u_3 u_n$ and, with repeated applications of \eqref{ybe}, move to the right as much as possible---unless this triangle is already at the right boundary. We then use the right reflection equation \eqref{rre} and afterwards the Yang--Baxter equation until we reach the top. We untangle and gain the factor $\wsw{u_3 u_{n+1}} = \wne{u_3 u_{n+1}}=\w_3$. In our example, we perform the following.
\begin{center}
\scalebox{0.35}{
\psfrag{1}{\huge $u_1$}
\psfrag{2}{\huge $u_2$}
\psfrag{3}{\huge $u_3$}
\psfrag{12}{\huge $u_1 u_2$}
\psfrag{13}{\huge $u_1 u_3$}
\psfrag{23}{\huge $u_2 u_3$}
\psfrag{14}{\huge $u_1 u_4$}
\psfrag{24}{\huge $u_2 u_4$}
\psfrag{34}{\huge $u_3 u_4$}
\psfrag{4}{\huge $u_4$}
\psfrag{15}{\huge $u_1 u_5$}
\psfrag{25}{\huge $u_2 u_5$}
\psfrag{35}{\huge $u_3 u_5$}
\psfrag{45}{\huge $u_4 u_5$}
\psfrag{=}{\Huge $=$}
\psfrag{=W}{\Huge $=\w_3$}
\includegraphics{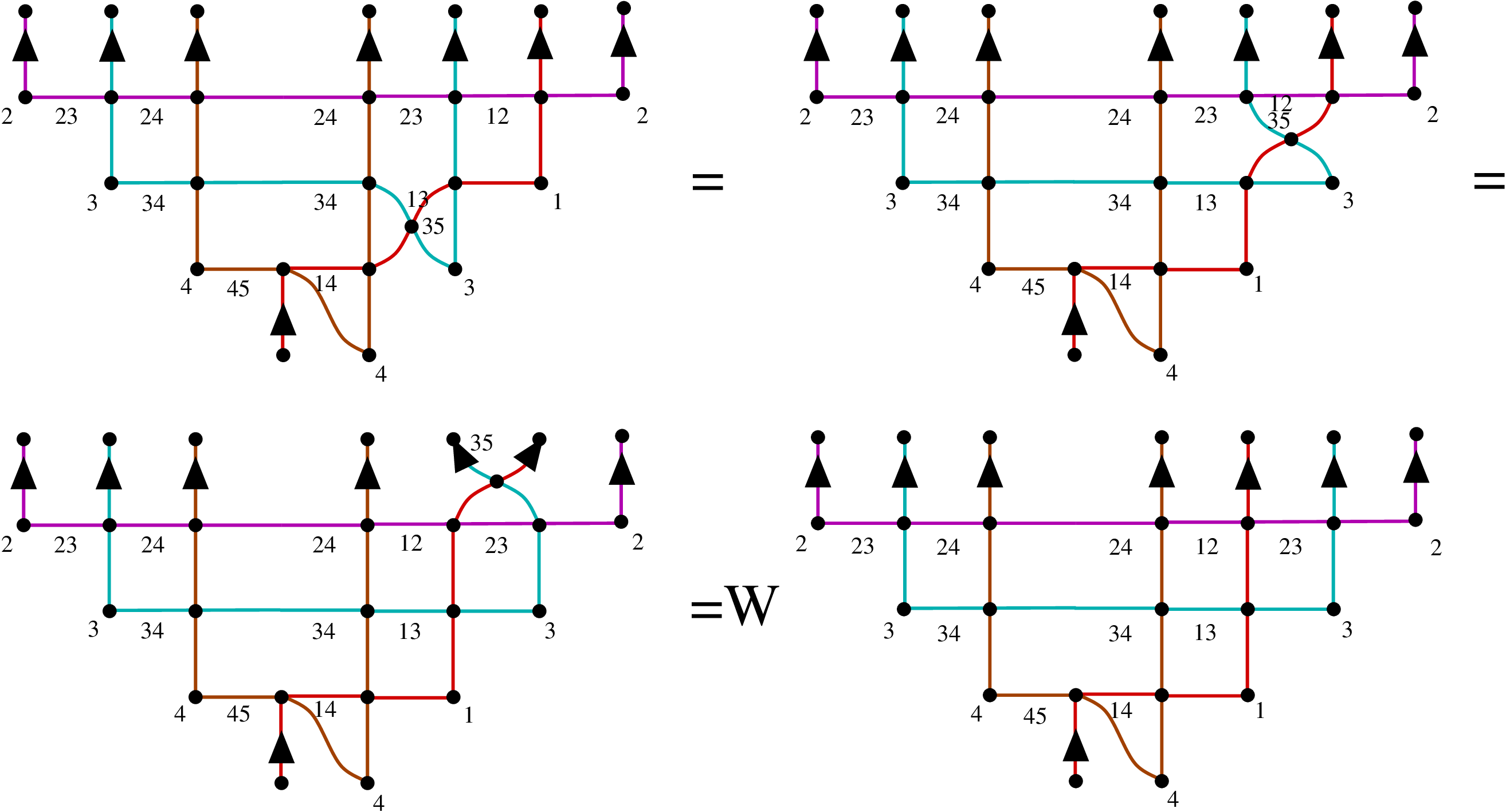}}
\end{center}
We continue in this manner and obtain in total the factor $\prod_{i=2}^{n} \wne{u_i u_{n+1}}$. \ilse{In our example, we still need to apply the following steps:
\begin{center}
\scalebox{0.35}{
\psfrag{1}{\huge $u_1$}
\psfrag{2}{\huge $u_2$}
\psfrag{3}{\huge $u_3$}
\psfrag{12}{\huge $u_1 u_2$}
\psfrag{13}{\huge $u_1 u_3$}
\psfrag{23}{\huge $u_2 u_3$}
\psfrag{14}{\huge $u_1 u_4$}
\psfrag{24}{\huge $u_2 u_4$}
\psfrag{34}{\huge $u_3 u_4$}
\psfrag{4}{\huge $u_4$}
\psfrag{15}{\huge $u_1 u_5$}
\psfrag{25}{\huge $u_2 u_5$}
\psfrag{35}{\huge $u_3 u_5$}
\psfrag{45}{\huge $u_4 u_5$}
\psfrag{=}{\Huge $=$}
\psfrag{=W}{\Huge $=\w_4$}
\includegraphics{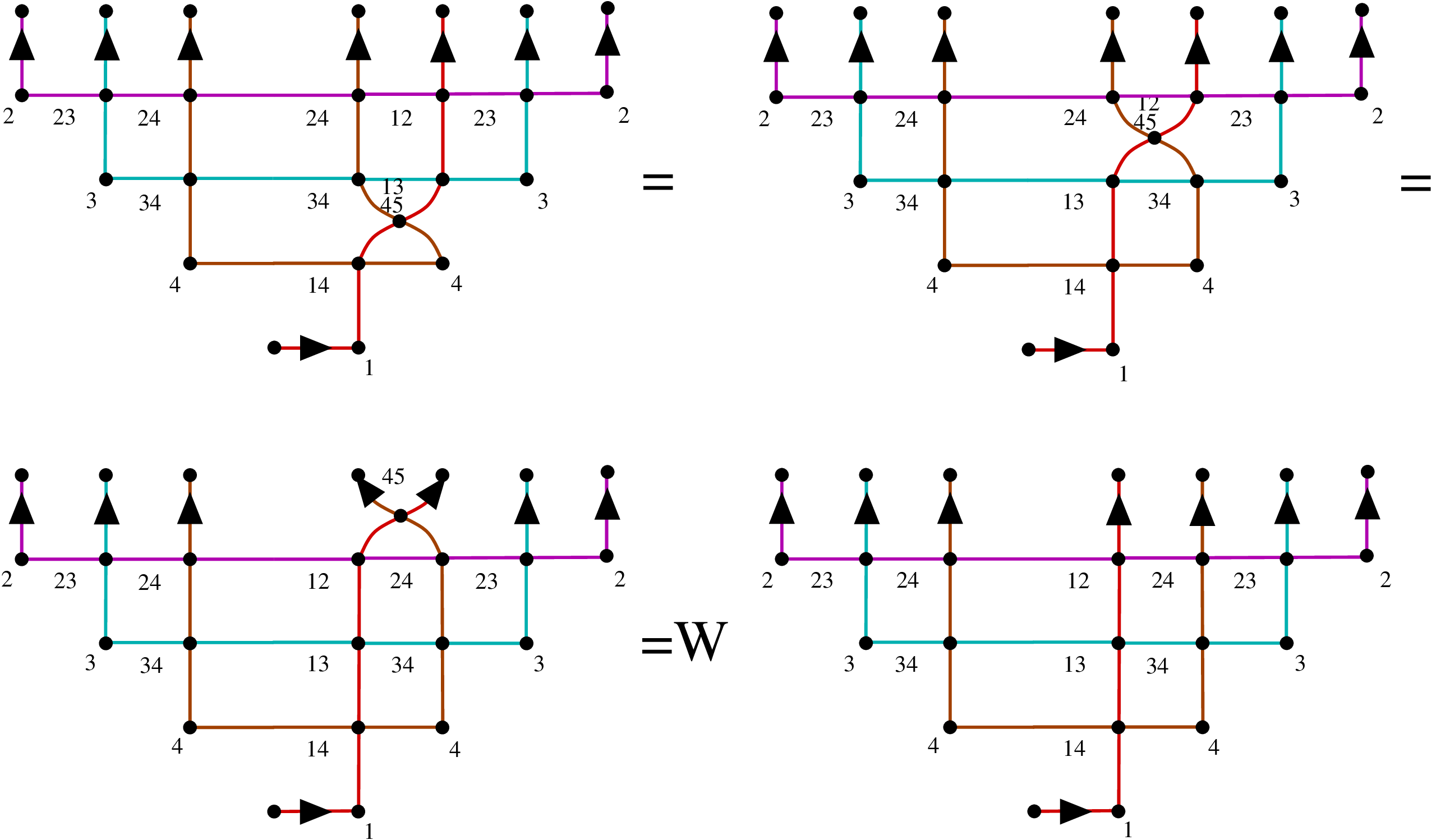}},
\end{center}
where $\wsw{u_4 u_{n+1}} = \wne{u_4 u_{n+1}}=\w_4$.}
The generating function of the resulting graph can be expressed as
$$\wrone{u_1} Z^{\uparrow}_{n-1}(u_2,\ldots,u_n; u_1) + \wrin{u_1} Z^{\downarrow}_{n-1}(u_2,\ldots,u_n;  u_1).$$

The proof of \eqref{down} is analogous, and \eqref{full} then follows from \eqref{up} and \eqref{down} using
Lemma~\ref{updown}.
\end{proof}

\begin{rem}
In \cite[Proposition~15]{DASASM}, we have shown the following identity for the
$\dasasm$ specialization of the order $n$ partition function.
\begin{equation}
\label{complicated}
\begin{aligned}
Z_n(u_1,\ldots,u_n; q^2 \bar u_1) & =
\left( (\wrone{u_1}+ \wrout{u_1})
Z^{\uparrow}_{n-1}(u_2,\ldots,u_n;  u_1) \right. \\
&
\qquad \qquad \qquad \left. + (\wrin{u_1} +  \wrmone{u_1})  Z^{\downarrow}_{n-1}(u_2,\ldots,u_n; u_1) \right) \Omega_n
\end{aligned}
\end{equation}
The proof is based on \eqref{trivial}, \eqref{ybe} and \eqref{rre}, and can thus be generalized to our more general weights. The identity is equivalent to Lemma~\ref{eval}, due to Lemma~\ref{updown}.
\end{rem}

\begin{theo}
\label{un} \ilse{
A sequence of Laurent polynomials $\widehat{Z}_n(u_1,\ldots,u_n;u_{n+1})$, $n \ge 1$,
in $u_{n+1}$ over the field
$\mathbb{C}(q,u_1,\ldots,u_n,\alpha_L,\beta_L,\gamma_L,\delta_L,\alpha_R,\beta_R,\gamma_R,\delta_R)$ is equal to the sequence of partition functions $Z_{n}(u_1,\ldots,\allowbreak u_n;u_{n+1})$ if and only if
$\widehat{Z}_1(u_1;u_2) = Z_1(u_1;u_2)$, and the
following conditions are satisfied for $n>1$.
\begin{enumerate}
\item
\label{1degree} The degree of $\widehat{Z}_n(u_1,\ldots,u_n;u_{n+1})$ in $u_{n+1}$ is no greater than $n$ and the order is no smaller than $-n$.
\item
\label{1sym} $\widehat{Z}_n(u_1,\ldots,u_n;u_{n+1})$ is symmetric in $u_1,\ldots,u_n$.
\item
\label{1inv} $\widehat{Z}_n(u_1,\ldots,u_n;u_{n+1})$ is invariant when simultaneously replacing $(u_1,\ldots,u_{n+1})$ with $(\bar u_1,\ldots,\allowbreak \bar u_{n+1})$ and interchanging left and right boundary constants.
\item
\label{1even} $\widehat{Z}_n(u_1,\ldots,u_n;u_{n+1})$ is even in $u_i$, for all $i=1,2,\ldots,n$.
\item
\label{1eval} $\widehat{Z}_n(u_1,\ldots,u_n; u_{n+1})$ satisfies the identity obtained from \eqref{full} by replacing  $Z_n$ with $\widehat{Z}_n$ on the left-hand side, and $Z_{n-1}$ with $\widehat{Z}_{n-1}$ on the right-hand side.
\end{enumerate}
Analogous characterizations for $Z^{\uparrow}_n(u_1,\ldots,u_n;u_{n+1})$ and  $Z^{\downarrow}_n(u_1,\ldots,u_n;u_{n+1})$ are obtained by replacing \eqref{full} in
\eqref{1eval} with \eqref{up} and \eqref{down}, respectively.}
\end{theo}

\begin{proof} \ilse{By Theorem~\ref{sym}, Proposition~\ref{inv},  Proposition~\ref{even} and Lemma~\ref{eval}, the partition function $Z_n(u_1,\ldots,u_n;u_{n+1})$ has all the
properties listed for $\widehat{Z}_n(u_1,\ldots,u_n;u_{n+1})$ in the statement. We show by induction with respect to $n$ that $\widehat{Z}_n(u_1,\ldots,u_n;u_{n+1})$ is uniquely determined by these properties. For $n=1$ this is true by assumption, and we assume $n>1$ in the following.

We consider  $\widehat{Z}_n(u_1,\ldots,u_n;u_{n+1})$ as Laurent polynomials in $u_{n+1}$, and, by \eqref{1degree}, it suffices to find $2n+1$ evaluations.
Now one evaluation (at $u_{n+1}= q^2 \bar u_1$) is given by \eqref{1eval}. Using  \eqref{1even} and \eqref{1eval}, as well as the facts that $\wrone{u_1}, \wrmone{u_1}$ are odd in $u_1$, while $\wrin{u_1}, \wrout{u_1}, \Omega_n$ are even, we obtain an evaluation at $u_{n+1}= - q^2 \bar u_1$ as follows.
\begin{equation}
\label{negative}
\begin{aligned} \widehat{Z}_{n}(u_1, u_2, &\ldots,u_{n}; - q^2 \bar u_1) \\
& \stackrel{\eqref{1even}}{=}
\widehat{Z}_n(- u_1,  u_2 \ldots, u_n;  - q^2 \bar u_1) \\
& \stackrel{\eqref{1eval}}{=}  \frac{1}{2}
\left\{ \left[ - \wrone{u_1} +  \wrout{u_1} + \wrin{u_1}  - \wrmone{u_1} \right]
\widehat{Z}_{n-1}(u_2,\ldots,u_n; - u_1) \right.  \\
& \quad  \left.  +  (-1)^{n+1} \left[ - \wrone{u_1}+ \wrout{u_1} -  \wrin{u_1} + \wrmone{u_1} \right]  \widehat{Z}_{n-1}(u_2,\ldots,u_n;  u_1) \right\}   \Omega_n.
\end{aligned}
\end{equation}
Using \eqref{1inv}, we also obtain evaluations at $u_{n+1} = \pm \bar q^2 \bar u_1$ as follows.
\begin{equation*}
\begin{aligned} \widehat{Z}_{n}(u_1, &\ldots,u_{n}; \pm \bar q^2 \bar u_1) \\
& \stackrel{\eqref{1inv}}{=}
 \widehat{Z}_n(\bar u_1, \ldots, \bar u_n;  \pm  q^2 u_1)_{L \leftrightarrow R} \\
& \stackrel{\eqref{1eval}, \eqref{negative},\eqref{1inv}}{=}
 \frac{1}{2}
\left\{  \left[ \pm \wlone{u_1} +  \wlout{u_1} + \wlin{u_1}  \pm \wlmone{u_1} \right]
\widehat{Z}_{n-1}(u_2,\ldots, u_n;  \pm u_1) \right.  \\
& \qquad  \left.  +  (-1)^{n+1} \left[ \pm \wlone{u_1}+ \wlout{ u_1} -  \wlin{ u_1} \mp \wlmone{u_1} \right]  \widehat{Z}_{n-1}(u_2,\ldots,u_n;  \mp u_1) \right\}  \\
& \qquad  \qquad \times  \wrout{u_1}
\prod_{i=2}^{n} \wnw{u_1 u_i} \wnw{\bar q^2 \bar u_1 u_i}
\end{aligned}
\end{equation*}
By the symmetry in $u_1,\ldots,u_{n}$, $u_1$ can be replaced by any $u_i$, $i=2,3,\ldots,n$, and so we have in total $4n$ evaluations in $u_{n+1}$, namely at $u_{n+1} = \pm q^{\pm 2} \bar u_i$, $i=1,2,\ldots,n$. All these evaluations are uniquely determined as $\widehat{Z}_{n-1}(u_2,\ldots, u_n; u_1)$ is uniquely determined by the induction hypothesis.}

The proofs for $Z_{n}^{\uparrow}$ and $Z_n^{\downarrow}$ are similar.
\end{proof}

\subsection{Characterization through evaluations in $u_1$}

This section provides preparations for alternative proofs of some results  that avoid Lemma~\ref{eval}, and it can therefore also be omitted when reading the article.

\begin{lem}
\label{lem_eval1}
Let $n$ be a positive integer.  For $u_1 = q, i q$,
\begin{equation}
\label{eval1}
Z_{n}(u_1,\ldots,u_n;u_{n+1}) =
\wlout{u_1} \wrone{u_1} \wne{u_1 u_{n+1}} \prod_{j=2}^{n} \wne{u_1 u_j}^2
Z_{n-1}(u_2,\ldots,u_n;u_{n+1}).
\end{equation}
The corresponding identities also hold for the refined partition functions
$Z^{\uparrow}_n,Z^{\downarrow}_n$.
\end{lem}

\begin{proof}
The possible local configurations for the top right boundary vertex  are $\rone$ and $\rout$. However, as $\wrout{q} = \wrout{ i q} = 0$, this local configuration must be $\rone$, which causes a fixing of the top row. More precisely, the local configurations of the top bulk vertices are $\ne$, and $\lout$ for the top left boundary vertex. \end{proof}

\begin{lem}
\label{lem_eval2} Let $n > 1$.
For $u_1= q^2 \bar u_2$,
\begin{multline}
\label{eval2}
Z_n(u_1,\ldots,u_n;u_{n+1}) = (\wrout{u_1} \wrin{u_2} +
\wrone{u_1} \wrone{u_2}) \wlout{u_1} \wlout{u_2} \\
\wne{u_1 u_{n+1}} \wne{u_2 u_{n+1}} \prod_{j=3}^{n} \wne{u_1 u_j}^2
\wne{u_2 u_j}^2 Z_{n-2}(u_3,\ldots,u_n;u_{n+1}).
\end{multline}
The corresponding identities also hold for the refined partition functions
$Z^{\uparrow}_n, Z^{\downarrow}_n$.
\end{lem}

\begin{proof}
This identity appeared before for the
$\dasasm$ specialization of the universal partition function in  \cite[Proposition~14]{DASASM}. The proof is based on \eqref{trivial} and \eqref{rue}, and since these rules are also valid for our more general weights, \eqref{eval2} is true in general.
\end{proof}

\begin{theo}
\label{u1} Let $p,q$ be indeterminates. Suppose we consider a specialization of $Z_n(u_1,\ldots,u_n;u_{n+1})$ where $u_{n+1}=p$ and
$\alpha_L,\beta_L,\gamma_L,\delta_L,\alpha_R,\beta_R,\gamma_R,\delta_R \in \mathbb{C}(p,q)$ such that
\begin{itemize}
\item $Z_n(u_1,\ldots,u_n;u_{n+1})$ has a zero at $u_1 = \bar p q^2$, and
\item left boundary constants are transformed into corresponding right boundary constants when replacing $p$ by $\bar p$.
\end{itemize}
A sequence of Laurent polynomials $\widehat{Z}_n(u_1,\ldots,u_n)$, $n \ge 1$,
in the variables $u_1,\ldots,u_{n}$ with coefficients in $\mathbb{C}(p,q)$ is equal to the above mentioned sequence of specializations of partition functions $Z_{n}(u_1,\ldots,u_n;u_{n+1})$ if and only if the sequences agree for $n=1,2$ and the following conditions are satisfied for $n>2$.
\begin{enumerate}
\item
\label{2degree}
 The degree of $\widehat{Z}_n(u_1,\ldots,u_n)$ in $u_{1}$ is no greater than $2n+2$ and the order is no smaller than $-2n-2$.
\item
\label{2sym}
$\widehat{Z}_n(u_1,\ldots,u_n)$ is symmetric in $u_1,\ldots,u_n$.
\item
\label{2inv}
$\widehat{Z}_n(u_1,\ldots,u_n)$ is invariant under the transformation
$(u_1,\ldots,u_{n},p) \mapsto (\bar u_1,\ldots, \bar u_{n},\bar p)$.
\item
\label{2even}
$\widehat{Z}_n(u_1,\ldots,u_n)$ is even in $u_i$, for all $i=1,2,\ldots,n$.
\item
\label{2eval}
$\widehat{Z}_n(u_1,\ldots,u_n)$ satisfies the identities obtained from \eqref{eval1} and \eqref{eval2} by replacing on the left-hand sides $Z_n(u_1,\ldots,u_n;u_{n+1})$ by
$\widehat{Z}_n(u_1,\ldots,u_n)$, replacing on the right-hand side of \eqref{eval1} $Z_{n-1}(u_2, \ldots, u_n; u_{n+1})$ by $\widehat{Z}_{n-1}(u_2, \ldots, u_n)$, replacing on the right-hand side of \eqref{eval2} $Z_{n-2}(u_3,\allowbreak \ldots, u_n; u_{n+1})$ by $\widehat{Z}_{n-2}(u_3, \ldots, u_n)$  and considering the same specializations of the
boundary constants and $u_{n+1}$ on the right-hand sides.
\item
\label{2extra}
$\widehat{Z}_n(u_1,\ldots,u_n)$ has a zero at $u_{1}=\bar p q^2$.
\end{enumerate}
Analogous characterizations for $Z^{\uparrow}_n(u_1,\ldots,u_n;u_{n+1})$ and  $Z^{\downarrow}_n(u_1,\ldots,u_n;u_{n+1})$ are obtained by replacing \eqref{eval1} and \eqref{eval2} with the respective identities.
\end{theo}

\begin{proof}
The partition function $Z_n(u_1,\ldots,u_n;u_{n+1})$ satisfies properties \eqref{2degree}--\eqref{2eval} by Proposition~\ref{even}, Theorem~\ref{sym}, Proposition~\ref{inv}, and Lemmas~\ref{lem_eval1} and
\ref{lem_eval2}, and property~\eqref{2extra} by assumption.

An even Laurent polynomial of degree at most $2n+2$ and order at least $-2n-2$ is uniquely determined by $2n+3$ evaluations where we need to
guarantee that there is no pair of evaluations which differ only in sign.

We have $2n-2$ evaluations (at $u_1=q^{\pm 2} \bar u_j$, $j=2,\ldots,n$) if we combine the evaluation obtained from \eqref{eval2} with \eqref{2sym} and \eqref{2inv}, and $4$
evaluations (at $u_1=q^{\pm 1}, i q^{\pm 1}$) if we combine the two evaluations obtained from \eqref{eval1}
with \eqref{2inv}. The additional evaluation is provided by \eqref{2extra}.
\end{proof}

\section{Maximal number of $-1$'s: Alternating sign triangles}
\label{ast}

The main purpose of this section is to provide the proof of Theorem~\ref{maxmone} and generalizations.

\subsection{The $\ast$ partition function}
\begin{theo}
\label{astpartfunct}
We have the following determinant formula for the partition function $Z_n(u_1,\ldots,u_n;\allowbreak u_{n+1})_{\ast}$.
\begin{multline}
\label{astpartfunctexpr}
\frac{\prod_{j=1}^n\sigma(\bar p u_j)\,\prod_{i=1}^n\prod_{j=1}^{n+1}\sigma(q^2u_iu_j)\sigma(q^2 \bar u_i \bar u_j)}
{\sigma(q^2)^{2n}\,\sigma(q^4)^{n(n-1)}\,\prod_{i=1}^n\sigma(u_i \bar u_{n+1})\,\prod_{1\le i<j\le n}\sigma(u_i \bar u_j)^2} \\
\times
\det_{1\le i,j\le n+1}\left( \Large \begin{cases}  \frac{1}{\sigma(q^2u_iu_j)\,\sigma(q^2 \bar u_i \bar u_j)},&i\le n\\[1.5mm]
1 - \frac{\sigma(\bar u_{n+1} u_j)}{\sigma(\bar p u_j)},&i=n+1\end{cases}\right)
\end{multline}
\end{theo}

\begin{proof} We use Theorem~\ref{un}.
The case $n=1$ is easy to verify.
To show that \eqref{astpartfunctexpr} is indeed a Laurent polynomial in $u_{n+1}$, observe that
$$
\frac{\prod_{j=1}^n\sigma(\bar p u_j)\,\prod_{i=1}^n\prod_{j=1}^{n+1}\sigma(q^2u_iu_j)\sigma(q^2 \bar u_i \bar u_j)}
{\sigma(q^2)^{2n}\,\sigma(q^4)^{n(n-1)}\,\prod_{1\le i<j\le n}\sigma(u_i \bar u_j)^2}
\det_{1\le i,j\le n+1}\left( \Large \begin{cases}\frac{1}{\sigma(q^2u_iu_j)\,\sigma(q^2 \bar u_i \bar u_j)},&i\le n\\[1.5mm]
1 - \frac{\sigma(\bar u_{n+1} u_j)}{\sigma(\bar p u_j)},&i=n+1\end{cases}\right)
$$
is a Laurent polynomial in $u_{n+1}$.  It is divisible by $\prod_{i=1}^{n} \sigma(u_i \bar u_{n+1})$, as the determinant vanishes if we set $u_{n+1}= \pm u_j$, for $j=1,\ldots,n$, since then the $j$-th column coincides with the $(n+1)$-st column of the matrix.\footnote{It is also not difficult to see directly that \eqref{astpartfunctexpr} is in fact a Laurent polynomial in $u_1,\ldots,u_{n+1}$: Observe that
\begin{multline*}
\frac{\prod_{j=1}^{n+1}\sigma(\bar p v_j)\,\prod_{i=1}^n\prod_{j=1}^{n+1}\sigma(q^2u_iv_j)\sigma(q^2 \bar u_i \bar v_j)}
{\sigma(q^2)^{2n}\,\sigma(q^4)^{n(n-1)}}
\det_{1\le i,j\le n+1}\left(\Large \begin{cases}\frac{1}{\sigma(q^2u_iv_j)\,\sigma(q^2 \bar u_i \bar v_j)},&i\le n\\[1.5mm]
1 - \frac{\sigma(\bar u_{n+1} v_j)}{\sigma(\bar p v_j)},&i=n+1\end{cases}\right)
\end{multline*}
is a Laurent polynomial in the variables $u_1,\ldots,u_{n+1},v_1,\ldots,v_{n+1}$ that is antisymmetric in $u_1,\ldots,u_n$ and
in $v_1,\ldots,v_{n+1}$. It is also even in each $u_i$, $i=1,\ldots,n$, and odd in each $v_i$, $i=1,2,\ldots,n+1$, and thus
divisible by $\prod_{1 \le i < j \le n} \sigma(u_i \bar u_j) \prod_{1 \le i < j \le n+1} \sigma(v_i \bar v_j)$. After setting $v_{n+1} = u_{n+1}$, it is  also
divisible by $\sigma(\bar p u_{n+1})$, since then the bottom right corner entry of the matrix is $1$.}

We check that \eqref{astpartfunctexpr} has the desired properties from
Theorem~\ref{un}.

\emph{\eqref{1degree}:} To deduce the bounds on the degree and order of \eqref{astpartfunctexpr} as a Laurent polynomial in
$u_{n+1}$, use the Leibniz formula for the determinant and observe that, after multiplying each summand with the prefactor and cancelling as much as possible, we obtain a sum of rational functions. Each numerator has degree $2n$ or $2n-1$ (depending on whether or not the entry in the bottom right corner of the matrix is involved) and order $-2n$ or $-2n+1$. The denominators have degree $n$ and order $-n$, and since we already know that \eqref{astpartfunctexpr} is a Laurent polynomial in $u_{n+1}$, the bounds follow.

\emph{\eqref{1sym}, \eqref{1inv} and \eqref{1even}:} The expression is \ilse{readily checked to be} symmetric in $u_1,\ldots,u_{n}$, and also even in $u_i$, for each $i \in \{1,2,\ldots,n\}$.
Moreover, it is invariant under the transformation $(u_1,\ldots,u_{n+1},p) \mapsto (\bar u_1,\ldots, \bar u_{n+1}, \bar p)$. Note that interchanging respective left and right boundary constants in the $\ast$-specialization (see Subsection~\ref{ast-spec}) is
equivalent to the replacement $p \mapsto \bar p$.

\emph{\eqref{1eval}:} Let $X_n(u_1,\ldots,u_n;u_{n+1})$ denote \eqref{astpartfunctexpr}. Then we need to show the following.
\begin{multline}
\label{toshow}
X_n(u_1,\ldots,u_n;  q^2 \bar u_1) = \frac{(q -  p \bar q  u_1)\sigma(q^2  u_1^2)}
{2 \sigma(q^2)^2} \prod_{i=2}^{n} \frac{\sigma(q^2  u_1  u_i) \sigma(q^4 \bar u_1  u_i)}{\sigma(q^4)^2} \\
\times ( (\bar u_1 + \bar p) (\bar u_1 q +  u_1 \bar q)
X_{n-1}(u_2,\ldots,u_n; u_1)  \\ +
(-1)^{n+1} (\bar u_1 - \bar p) \sigma(\bar u_1 q)
X_{n-1}(u_2,\ldots,u_n;-u_1) )
\end{multline}
For this purpose, multiply the $(n+1)$-st column of the matrix in $X_n(u_1,\ldots,u_n;  q^2 \bar u_1)$ with
the prefactor $\sigma(q^2 u_1 u_{n+1}) \sigma(q^2 \bar u_1 \bar u_{n+1})$. After taking $u_{n+1} = q^2 \bar u_1$, the $(n+1)$-st column is
$(1,0,\ldots,0)^t$. Now expand the determinant with respect to this column to reduce the computation to the determinant of an $n \times n$ matrix. After moving the first column to the right, note that the first $n-1$ rows of this matrix match those of the matrices in $X_{n-1}(u_2,\ldots,u_n;u_1)$ and $X_{n-1}(u_2,\ldots,u_n;-u_1)$. Take the linear combination of the last rows in the right-hand side
according to the prefactors and compare it with the last row and the prefactor in
the left-hand side \ilse{to see that they are equal}. \end{proof}

\subsection{Specialization of the $\ast$ partition function at $u_{n+1}=p$}
Taking $u_{n+1} \to p$ in \eqref{astpartfunctexpr}, the last row of the determinant becomes
$(0,0,\ldots,0,1)$, and we obtain the following simplified formula.

\begin{cor}
\label{astpartfunct1}
We have the following formula for $Z_n(u_1,\ldots,u_n;p)_{\ast}$.
\begin{equation}
\label{astpartfunct1expr}
\frac{\prod_{i=1}^n\sigma(q^2 p u_i)\sigma(q^2 \bar p \bar u_i)\,
\prod_{i,j=1}^n\sigma(q^2u_iu_j)\sigma(q^2 \bar u_i \bar u_j)}
{\sigma(q^2)^{2n}\,\sigma(q^4)^{n(n-1)}\,\prod_{1\le i<j\le n}\sigma(u_i \bar u_j)^2}
\det_{1\le i,j\le n}\biggl(\frac{1}{\sigma(q^2u_iu_j)\,\sigma(q^2 \bar u_i \bar u_j)}\biggr)
\end{equation}
\end{cor}

Next we provide a sketch of an alternative proof of Corollary~\ref{astpartfunct1} which avoids Theorem~\ref{astpartfunct} and Lemma~\ref{eval}, the latter being probably the most complicated ingredient in the proof of Theorem~\ref{astpartfunct}. This proof is also more in line with Kuperberg's approach in \cite{KuperbergRoof} as it uses a characterization of the partition function in terms of evaluations in $u_{1}$.

\begin{proof}[Sketch of an alternative proof using Theorem~\ref{u1}]
First we need to verify that
$$Z_n(\bar p q^2,u_2,\ldots,u_n;p)_{\ast}=0.$$
By the symmetry of Theorem~\ref{sym}, it suffices to show that $Z_n(u_1,\ldots,u_{n-1},\bar p q^2;p)_{\ast}=0$.

Consider the bottom bulk vertex $v$ in the central column and assume $u_n=\bar p q^2, u_{n+1}=p$.
Its label is $u_{n} u_{n+1} = q^2$, and using \eqref{trivial}, there are only four possible local configurations around $v$ that contribute a non-zero weight (which is in fact $1$ in all cases). Moreover, as
$\wrin{u_{n}}=\wrmone{u_{n}}=0$, there are only two possibilities (namely $\rone, \rout$) for the right boundary vertex $v_R$ adjacent to $v$.

It turns out that the weights of the remaining configurations cancel in pairs:
Fix a configuration and obtain another configuration by reversing the orientation of the bottom vertical edge incident with $v$ and the right horizontal edge incident with $v$. Then only the weight of $v_R$ has changed: It is $\frac{\sigma(\bar p^2 q^2)}{\sigma(q^2)}$ if the horizontal edge points to the right, and it is $-\frac{\sigma(\bar p^2 q^2)}{\sigma(q^2)}$ if the horizontal edge points to the left.

Now we may use Theorem~\ref{u1}. The cases $n=1,2$ are easy to verify. To show that \eqref{astpartfunct1expr} is a Laurent polynomial in $u_1,\ldots,u_{n}$, proceed as in the footnote of the proof of  Theorem~\ref{astpartfunct}.

\emph{\eqref{2degree}:} We use the Leibniz rule for the determinant in \eqref{astpartfunct1expr}. The degree and order of the numerator in $u_1$ of each summand---after multiplication with the prefactor and cancelling---are $4n-2$ and $-4n+2$, respectively, while the degree and order of the denominator are $2n-2$ and $-2n+2$, respectively.

\emph{\eqref{2sym},\eqref{2inv},\eqref{2even}, \eqref{2eval} and \eqref{2extra}:} The expression is symmetric in $u_1,\ldots,u_n$, invariant under the transformation
$(u_1,\ldots,u_n,p) \mapsto (\bar u_1,\ldots,\bar u_n, \bar p)$, and even in $u_i$, $i=1,2,\ldots,n$.
The identities obtained from Lemmas~\ref{eval1} and \ref{eval2} can be verified straightforwardly. It is also evident that the expression in \eqref{astpartfunct1expr} has a zero at
$u_1 = \bar p q^2$.
\end{proof}

A comparison with the partition function of all $\asm$s shows a close relation
to the $\ast$ partition function at $u_{n+1}=p$: Let $Y_n(u_1,\ldots,u_n;v_1,\ldots,v_n)$ denote the
partition function of $n \times n$ $\asm$s ($u_1,\ldots,u_n$ and $v_1,\ldots,v_n$ are the horizontal and vertical spectral parameters, respectively, and we use the bulk weights given in Table~\ref{weights}). Then
$$
Y_n(u_1,\ldots,u_n;v_1,\ldots,v_n) = \frac{\prod_{i,j=1}^{n} \sigma(q^2 u_i \bar v_j) \sigma(q^2 \bar u_i v_j)}{\sigma(q^4)^{n(n-1)} \prod_{1 \le i < j \le n} \sigma(u_i \bar u_j) \sigma(\bar v_i v_j)}
\det_{1 \le i, j \le n} \left( \frac{1}{\sigma(q^2 u_i \bar v_j) \sigma(q^2 \bar u_i v_j)} \right).
$$
To see this, consult for instance \cite[Theorem~10]{KuperbergRoof}: Set
$a= q^2$ in Kuperberg's formula and divide by $\sigma(q^4)^{n^2}$ to take the difference in the choice of vertex weights into account.
A determinant formula for the partition function of the six-vertex model on an $n \times n$ grid with domain wall boundary conditions that is equivalent to the one above was first derived by
Izergin \cite[Eq. (5)]{Ize87}, using results of Korepin \cite{Ize87,Kor82}. Corollary~\ref{astpartfunct1} now implies that
\begin{equation}
\label{ast-asm}
Z_n(u_1,\ldots,u_n;p)_{\ast} =
\prod_{i=1}^n \frac{\sigma(q^2p u_i)\sigma(q^2 \bar p \bar u_i)}{\sigma(q^2)^2}
Y_n(u_1,\ldots,u_n; \bar u_1, \ldots, \bar u_n).
\end{equation}

In particular, \eqref{ast-asm} now implies that there is the same number of order $n$ $\asm$s as there is of order $n$ $\ast$s, since $Y_n(1,\ldots,1;1,\ldots,1)$ is the number of order $n$ $\asm$s.

\subsection{Specialization of the $\ast$ partition function at $u_{n+1}=p$ and $q=e^{\frac{i \pi}{6}}$}

Next we specialize $q=e^{\frac{i \pi}{6}}$ and obtain an expression involving \emph{Schur functions} $\schur_{\lambda}(x_1,\ldots,x_n)$.
We use the well-known determinant formula
$$
\schur_{\lambda}(x_1,\ldots,x_n) = \frac{\det_{1 \le i , j \le n} (x_i^{\lambda_j+n-j})}{\prod_{1 \le i < j \le n} (x_i - x_j)},
$$
where $\lambda=(\lambda_1,\ldots,\lambda_n)$,
and if the partition has less than $n$ parts, we add the appropriate number of zero parts to the partition.

\begin{theo}
\label{astschurtheo}
We have the following Schur function expression for $Z_n(u_1,\ldots,u_n;p)_{\ast}$ at $q=e^{\frac{i \pi}{6}}$.
\begin{equation}
\label{astschur}
3^{-\binom{n+1}{2}} \prod_{i=1}^{n} (p ^2 u_i^2 + 1 + \bar p^2 \bar u_i^2)  \schur_{(n-1,n-1,n-2,n-2,\ldots,1,1)}(u_1^2,\bar u_1^2,\ldots,u_n^2, \bar u_n^2)
\end{equation}
\end{theo}

\begin{proof}
Use \eqref{astpartfunct1expr} and \cite[Theorem 2.4 (1), 2nd eq.]{OkadaCharacters}.
We note that the Schur function expression for the ASM partition function at $q=e^{\frac{i \pi}{6}}$ was obtained simultaneously by Okada and
by Stroganov \cite[Eq. (17)]{StroNewWay}.
\end{proof}

Let us point out that formulas for the specializations of
$Z_n(u_1,\ldots,u_n;p)_{\ast}$ at $q=e^{\frac{i \pi}{4}},e^{\frac{i \pi}{8}},e^{\frac{i \pi}{12}}$ are also provided in \cite[Theorem 2.4 (1)]{OkadaCharacters}. These specializations correspond to the $x$-enumerations of $\asm$s (or $\ast$s, see also
Subsection~\ref{xenum}) for $x=0,2,3$, respectively.

\begin{rem}
In \cite{AyyBehFis16b} it is shown that the Schur function in \eqref{astschur} factorizes as
\begin{equation}
\label{ASTSchurfact}
\begin{aligned}
\schur_{(2n-1,2n-1,2n-2,2n-2,\ldots,1,1)}&(u_1,\bar u_1,u_2,\bar u_2,\ldots,u_{2n}, \bar u_{2n})
 \\ &=\sp_{(n-1,n-1,n-2,n-2,\ldots,1,1)}(u_1,\ldots,u_{2n}) \oeven_{(n,n,n-1,n-1,\ldots,1,1)}(u_1,\ldots,u_{2n}), \\
\schur_{(2n,2n,2n-1,2n-1,\ldots,1,1)}(u_1,&\bar u_1,u_2,\bar u_2,\ldots,u_{2n+1},
\bar u_{2n+1})
 \\ &=\so_{(n,n,n-1,n-1,\ldots,1,1)}(u_1,\ldots,u_{2n+1}) \so_{(n,n,n-1,n-1,\ldots,1,1)}(-u_1,\ldots,-u_{2n+1}),
\end{aligned}
\end{equation}
where the notation $\so$ and $\oeven$ for the orthogonal characters is the same as that used in \cite{CiuKra09}. These factorizations are very similar to those obtained by Ciucu and Krattenthaler in that paper.
\end{rem}

\subsection{Inversion numbers: A generalization of Theorem~\ref{maxmone}}

The inversion number of an order $n$ $\asm$ $A=(a_{i,j})$ is defined as
$$
\inv(A) = \sum\limits_{1 \le i' < i \le n, 1 \le j' \le j \le n} a_{i' j} a_{i j'}.
$$
If $A$ is a permutation matrix, $\inv(A)$ is the number of inversions
of the permutation $\pi=(\pi_1,\ldots,\pi_n)$, where $\pi_i$ is the column for the unique
$1$ in row $i$. The inversion number of $\asm$s was first defined in \cite[Eq. (18)]{RR86}, where it was referred to as the number of positive inversions. Another generalization of the inversion number of permutations is obtained if, instead of summing over all $1 \le j' \le j \le n$, we sum over all $1 \le j' < j \le n$, and this statistic, which in fact evaluates to $\inv(A)+\mu(A)$, appeared in \cite[p. 344]{DPPMRR}.
In terms of the six-vertex model, we have
$$
\inv(A) = \# \sw \in \c(A),
$$
where $\c(A)$ denotes the six-vertex configuration of $A$. This is because
$$\inv(A) = \sum_{i,j=1}^{n} \left( \sum_{1 \le i' < i} a_{i' j} \right)
\left( \sum_{1 \le j' \le j} a_{i j'} \right),$$
i.e., $\inv(A)$ is the number of entries $a_{i,j}$ of $A$ such that the sum of entries in the same column above $a_{i,j}$ (not including $a_{i,j}$) is $1$ and the sum of entries in the same row to the left of $a_{i,j}$ (including $a_{i,j}$) is $1$.

It is not hard to see that also
$\inv(A) = \sum_{1 \le i \le i' \le n, 1 \le j < j' \le n} a_{i' j} a_{i j'}$ or, equivalently,
$\inv(A) = \# \ne \in \c(A)$. This implies
$$
\inv(A) = \frac{1}{2} \left(\# \sw \in \c(A) + \# \ne \in \c(A) \right).
$$
We generalize this version of the inversion number to $\ast$s straightforwardly: Suppose $T \in \ast(n)$. Then let
$$
\inv_{\da}(T) = \frac{1}{2} \left(\# \sw \in \c(T) + \# \ne \in \c(T) \right),
$$
where $\c(T)$ is the triangular six-vertex configuration associated with $T$.

It can be shown (where we leave the details of the derivation to the reader)
that the inversion number of $T=(t_{i,j})_{1 \le i \le n, i \le j \le 2n-i}$ can
be written directly in terms of its entries as
$$
\inv_\da(T)=\sum_{i'<i\atop j'\le j}t_{i'j}t_{ij'}=\sum_{i'\le i\atop j'<j}t_{i'j}t_{ij'}=\sum_{i'<i\atop j'<j}t_{i'j}t_{ij'}-\mu_\da(T),
$$
where the summation is over all $i',j,i,j'$ such that $t_{i' j}$ and $t_{i j'}$
are defined. It follows from these expressions that $\inv_{\da}(T)$ is an integer, while it follows from the definition that $\inv_{\da}(T)$ is non-negative.

\begin{theo}
\label{maxmone-gen}
The joint distribution of the statistics $\mu$ and $\inv$ on the set $\asm(n)$ \ilse{is the same as} the joint distribution of the statistics $\mu_{\da}$ and $\inv_{\da}$ on the set $\ast(n)$, i.e., for all integers $m$ and $i$ we have
$$
|\{A \in \asm(n) \mid \mu(A)=m,  \inv(A)=i\}| = |\{T \in \ast(n) \mid \mu_{\da}(T)=m, \inv_{\da}(T)=i \}|.
$$
\end{theo}

\begin{proof}
Define the following complementary inversion numbers for $A \in \asm(n)$ and
$T \in \ast(n)$:
$$
\inv'(A) = \frac{1}{2} \left( \# \se \in \c(A) + \# \nw \in \c(A) \right)
\qquad
\inv'_{\da}(T) = \frac{1}{2} \left( \# \se \in \c(T) + \# \nw \in \c(T) \right)
$$
As there are $n$ more $1$'s in an $\asm$ (resp. $\ast$) of order $n$ than $-1$'s, we have
$$
\mu(A) = \binom{n}{2} - \inv(A) - \inv'(A) \quad \text{and} \quad
\mu_{\da}(T) = \binom{n}{2} - \inv_{\da}(T) - \inv'_{\da}(T),
$$
and so it suffices to show
$$
\sum_{A \in \asm(n)} x^{\inv(A)} y^{\inv'(A)} =
\sum_{T \in \ast(n)} x^{\inv_{\da}(T)} y^{\inv'_{\da}(T)}.
$$
The left hand side can be studied using the partition function $Y_n$. Indeed, we have
\begin{equation}
\label{asm-side}
Y_n(p,\ldots,p;\bar p,\ldots, \bar p) = \sum_{A \in \asm(n)}
\left( \frac{\sigma(q^2 p^2)}{\sigma(q^4)} \right)^{ 2 \inv(A)} \left( \frac{\sigma(q^2 \bar p^2)}{\sigma(q^4)} \right)^{2 \inv'(A)},
\end{equation}
since the (bulk) weights for the $\asm$ partition function $Y_n(u_1,\ldots,u_n;v_1,\ldots,v_n)$
are chosen as in Table~\ref{weights} and the label of the vertex in row $i$ and column $j$ is $u_i \bar v_j$. Now it is important to note that $\frac{\sigma(q^2 p^2)}{\sigma(q^4)}$ and $\frac{\sigma(q^2 \bar p^2)}{\sigma(q^4)}$ are algebraically independent if $p,q$ are (algebraically independent) indeterminates. Indeed, let $F(x,y)=\sum_{s,t \ge 0} a_{s,t} x^s y^{t}$ be a polynomial in $x,y$ and
$d$ its total degree. Then $F\left(\frac{\sigma(q^2 p^2)}{\sigma(q^4)},\frac{\sigma(q^2 \bar p^2)}{\sigma(q^4)} \right)$ is a Laurent polynomial in
$p$ of degree $2 d$ with leading coefficient
$$
\frac{1}{\sigma(q^4)^{d}} \sum_{s,t \ge 0, s+t=d} a_{s,t} (-1)^t q^{2s-2t},
$$
and thus non-zero.

On the $\ast$ side, we consider the specialization $Z_n(p,\ldots,p;p)_{\ast}$.  By \eqref{ast1g}, the sum in the generating function $Z_n(p,\ldots,p;p)_{\ast}$ can be taken over all odd $\dasasm$-triangles $T'$ of order $n$ with $\n_{-1}(T')=n$, and by \eqref{ast2g}, all left boundary weights are $\frac{\sigma(q^2 p^2)}{\sigma(q^2)}$, while all right boundary weights are $\frac{\sigma(q^2 \bar p^2)}{\sigma(q^2)}$. Hence, \ilse{by Table~\ref{weights},}
$$
Z_n(p,\ldots,p;p)_{\ast} = \left( \frac{\sigma(q^2 p^2)
\sigma(q^2 \bar p^2)}{\sigma(q^2)^2} \right)^n \sum_{T \in \ast(n)} \left( \frac{\sigma(q^2 p^2)}{\sigma(q^4)} \right)^{ 2 \inv_{\da}(T)} \left( \frac{\sigma(q^2 \bar p^2)}{\sigma(q^4)} \right)^{2 \inv'_{\da}(T)},
$$
and the assertion now follows from \eqref{ast-asm} and \eqref{asm-side}.
\end{proof}

\subsection{Position of the $1$ in the top row of an $\ast$}

A statistic that was extensively studied for $\asm$s is the position of the $1$ in the top row, see for instance \cite{DPPMRR,ZeilbergerRefinedASMProof}. In this subsection, we study the analogous statistic on $\ast$s and derive a formula for the number of $\ast$s of order $n$ that have the $1$ in the top row in a prescribed position in terms of the doubly refined enumeration of $\asm$s with respect to the positions of the $1$'s in the top row and leftmost column. Explicit formulas for this doubly refined enumeration of $\asm$s can in turn be derived from results in \cite{StroNewWay}, see also \cite{FischerRefEnumASM,FischerLinRel,BehrendQuad}.

For an $\ast$ $T$, let
$$\kappa_{\da}(T)= \text{(position of the $1$ in the top row of $T$)}-1$$
and, for an $\asm$ $A$, let
\begin{multline*}
\kappa(A) =  \text{(position of the $1$ in the top row of $A$)} + \text{(position of the $1$ in the first column of $A$)}-2.
\end{multline*}
We define two generating functions as follows.
$$
\mathcal{K}_n^{\ast}(x) = \sum_{T \in \ast(n)} x^{\kappa_{\da}(T)}, \qquad
\mathcal{K}_n^{\asm}(x) = \sum_{A \in \asm(n)} x^{\kappa(A)}
$$
For example, $\mathcal{K}_n^{\ast}(x)=2+x+x^2+x^3+2 x^4$ and
$\mathcal{K}_n^{\asm}(x)=2+2 x^2+2 x^3 + x^4$.
By setting $u_2=\ldots=u_n=1, p=1,
q=e^{\frac{i \pi}{6}}$ in \eqref{ast-asm}, it can be shown (where we leave the details of the derivation to the reader) that
$$
(2-x) x^2 \mathcal{K}_n^{\ast}(x) = (x^2-x+1) \mathcal{K}_n^{\asm}(x)+(1-x)(x^{2n}-1) A_{n-1},
$$
where $A_n$ denotes the number of $n \times n$ ASMs.
This implies
$$
2 T_{n,k} = T_{n,k-1} + D_{n,k}-D_{n,k+1}+D_{n,k+2} + (\delta_{k,-1}-\delta_{k,-2} + \delta_{k,2n-2} - \delta_{k,2n-1}) A_{n-1},
$$
with $T_{n,k} = |\{T \in \ast(n) \mid \kappa_{\da}(T)=k\}|$ and
$D_{n,k}= |\{A \in \asm(n) \mid \kappa(A)=k\}|$. The recursion can be solved and we obtain
$$
T_{n,k} = 3 \sum_{i=0}^{k-3} 2^{i-k} D_{n,i+3} - \frac{1}{4} D_{n,k+1} + \frac{1}{2} D_{n,k+2} + (2^{-k+1}-\frac{3}{2} \delta_{k,0} - \frac{3}{4} \delta_{k,1} + \frac{1}{2} \delta_{k,2n-2} ) A_{n-1}.
$$
If $B_{n,i,j}$ denotes the number of $n \times n$ $\asm$s that have the $1$ in the top row in column $j$ and the $1$ in the leftmost column in row $i$, then
$$
D_{n,k} = \sum_{i+j=k+2} B_{n,i,j}.
$$

\subsection{Another interesting $\ast$ statistic}
\label{unique1ast}

Each left or right boundary entry of an $\ast$ $T$ is either $0$ or $1$. Let $\l(1,T)$ denote the number of
$1$'s on the left boundary, and $\r(0_1,T)$ the number of $0$'s on the right boundary that are contained in a column with sum $1$.  Define the statistic
$$
\rho(T) = \l(1,T) + \r(0_1,T)+1
$$
for $T \in \ast(n)$. We have numerical evidence that, for any integers $n$ and $r$,
$$
|\{A \in \ast(n) \mid \rho(A)=r \}| = |\{A=(a_{i,j}) \in \asm(n) \mid  a_{1,r}=1 \}|.
$$
In a forthcoming paper, we will study this statistic and derive related constant term identities.
Interestingly, data suggests that the joint distribution of $\mu_{\da}, \inv_{\da}$ and $\rho$ on $\ast(n)$ is not the same as the joint distribution of $\mu, \inv$ and the position of the unique $1$ in the top row of an $\asm$ on $\asm(n)$.

\subsection{$2$-enumeration and $3$-enumeration of $\ast$s}
\label{xenum}

The following weighted enumeration of $\asm$s is referred to as the $x$-enumeration:
$$
\sum_{A \in \asm(n)} x^{\mu(A)}
$$
If $x=0$, then we obviously obtain $n!$ since $\asm$s with no $-1$ are precisely permutation matrices. Besides $x=0,1$, the numbers are also round for
$x=2$ (see \cite{DPPMRR}) and $x=3$ (see \cite{KuperbergASMProof}). Theorem~\ref{maxmone} states that the $x$-enumeration of $\ast$s is the same as the $x$-enumeration of $\asm$s for any $x$. Hence we obtain the following.

\begin{cor}
\label{23}
The $2$-enumeration of order $n$ $\ast$s is $2^{\binom{n}{2}}$, while the
$3$-enumeration of order $2n+1$ $\ast$s is
$$
3^{n(n+1)} \prod_{i=1}^{n} \left( \frac{(3i-1)!}{(n+i)!} \right)^2
$$
and
$$
3^{n(n+2)}  \frac{(3n+2)! n!}{(2n+1)! ^2} \prod_{i=1}^{n} \left( \frac{(3i-1)!}{(n+i)!} \right)^2
$$
for order $2n+2$ $\ast$s.
\end{cor}

Curiously, as demonstrated in \cite{ProppDomino} (see also \cite[Remark 4.3]{CiucuReflect}), the $2$-enumeration of $\asm$s can be translated into the enumeration of \emph{perfect matchings} of  certain regions of the square grid (or, equivalently, into the plain enumeration of the domino tilings of the \emph{Aztec diamond}). Similarly, we may now translate our result on the $2$-enumeration of $\ast$s into a matching enumeration result.

In order to do so, define for each integer $n>0$ a graph $\mathcal{Q}_n$ as indicated in Figure~\ref{m5} (left). It consists of $n$ rows of centered sequences of tilted squares (shaded in our drawing and referred to as diamonds in the following) of lengths $2n-1, 2n-3,\ldots,1$, from top to bottom. The bottom vertices of the columns (except for the middle column) play a special role and these vertices are denoted by their column position, counted from the left. Note that $\mathcal{Q}_{n}$ is bipartite and the vertices $1,2,\ldots,n-1,n+1,\ldots,2n-1$ all lie in the same vertex class. There are $n-1$ more vertices in this class than in the other class, and hence, in order to obtain a graph that could possibly possess a perfect matching, we may delete $n-1$ vertices in $\{1,2,\ldots,n-1,n+1,\ldots, 2n-1\}$ from $\mathcal{Q}_{n}$. There is an example of such a perfect matching in Figure~\ref{m5} (right), where we delete vertices $1,2,4,9$.

\begin{figure}
\scalebox{0.3}{
\psfrag{l1}{\Huge $1$}
\psfrag{l2}{\Huge$2$}
\psfrag{l3}{\Huge $3$}
\psfrag{l4}{\Huge $4$}
\psfrag{r1}{\Huge $9$}
\psfrag{r2}{\Huge $8$}
\psfrag{r3}{\Huge $7$}
\psfrag{r4}{\Huge $6$}
\includegraphics{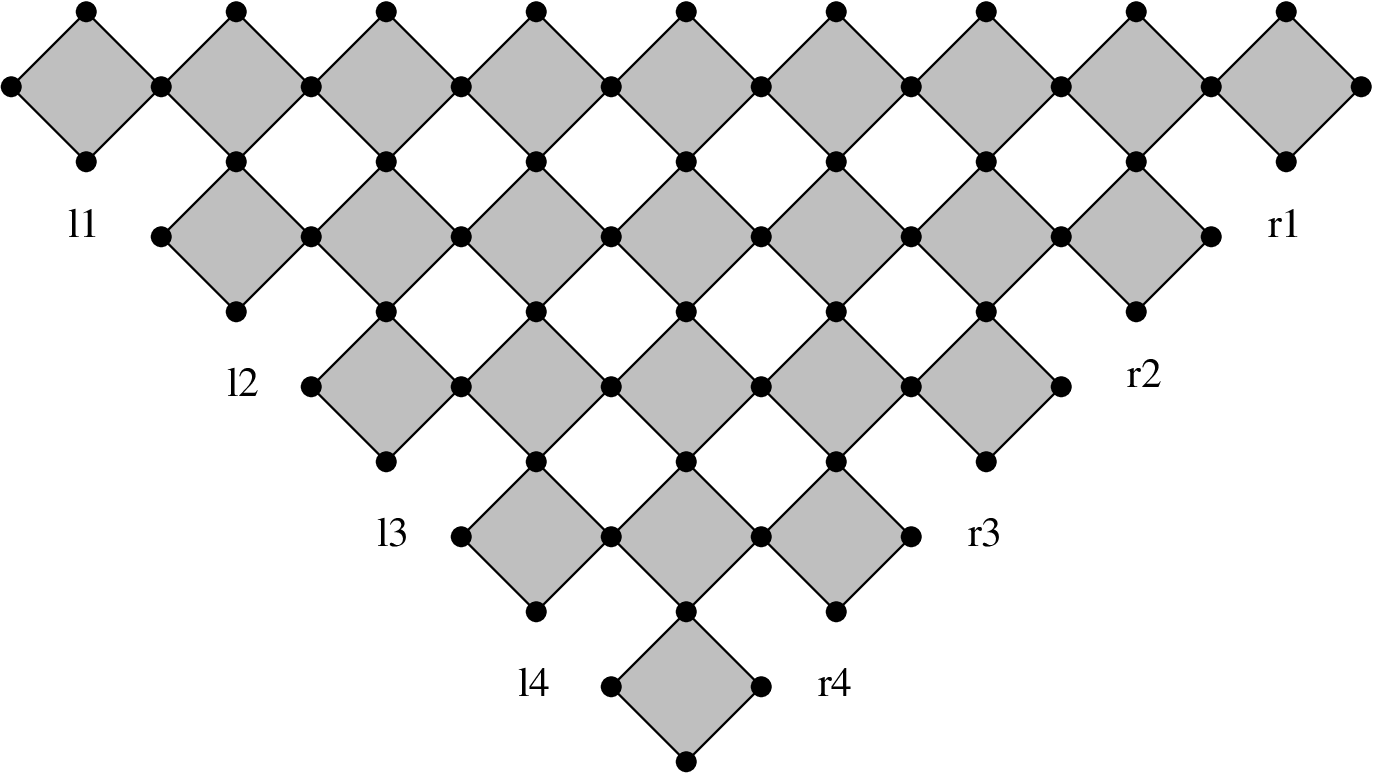}} \qquad
\scalebox{0.3}{
\psfrag{1}{\Huge $1$}
\psfrag{0}{\Huge $0$}
\psfrag{-1}{\Huge $-1$}
\includegraphics{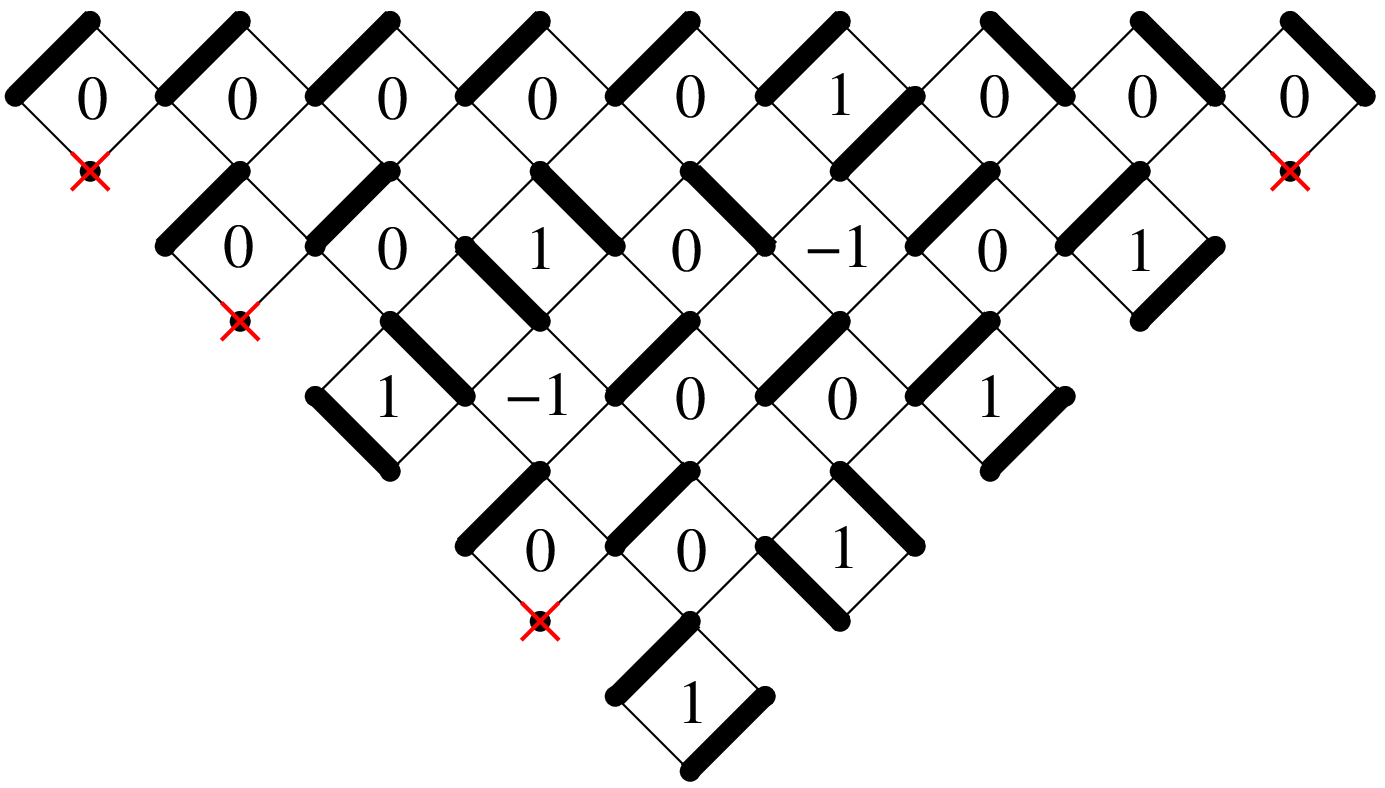}}
\caption{\label{m5} $\mathcal{Q}_5$ (left) and an example (right)}
\end{figure}

Now, on the $\ast$-side, since each of the $n$ rows of an order $n$ $\ast$ has sum $1$, and each of the $2n-1$ columns has
sum $0$ or $1$, there must be precisely $n-1$ columns that have sum $0$. Let $1 \le i_1 < i_2 < \ldots < i_{n-1} \le 2n-1$ and denote by
$\ast(n;i_1,\ldots,i_{n-1})$ the subset of $\ast(n)$, where the columns with sum $0$ are the columns $i_1,\ldots,i_{n-1}$. (Note that this set is obviously empty if $i_j=n$ for a $j$.)

In the example in Figure~\ref{m5} (right), we indicate a surjection from the set of perfect matchings of $\mathcal{Q}_n - \{i_1,\ldots,i_{n-1}\}$ onto $\ast(n;i_1,\ldots,i_{n-1})$ (see also \cite[Remark 4.3]{CiucuReflect}): Each diamond corresponds to an entry of the $\ast$, and this entry is in fact the number of matching edges on the boundary of the diamond. To compute the size of the preimage of a given $\ast$ under this map, observe that for each $1$ of the $\ast$, there are obviously two possibilities to arrange the two matching edges on the boundary of the respective diamond. In summary,
$$
\sum_{T \in \ast(n)} 2^\text{$\#$ of $1$'s in $T$} = \sum_{V \subseteq \{1,2,\ldots,n-1,n+1,\ldots,2n-1\} \atop |V|=n-1} (\text{$\#$ of perfect matchings of $\mathcal{Q}_n - V$}).
$$
As $\sum_{T \in \ast(n)} 2^\text{$\#$ of $1$'s in $T$} = 2^n \sum_{T \in \ast(n)} 2^{\mu(T)}$,
Corollary~\ref{23} now implies the following.

\begin{cor}
\label{match}
For any $n=1,2,3,\ldots$,
$$
 \sum_{V \subseteq \{1,2,\ldots,n-1,n+1,\ldots,2n-1\} \atop |V|=n-1} (\text{$\#$ of perfect matchings of $\mathcal{Q}_n - V$}) = 2^{\binom{n+1}{2}}.
$$
\end{cor}

The number of perfect matchings of a tilted square region of the square grid of order $n$ (dual to the Aztec diamond of order $n$) is also $2^{\binom{n+1}{2}}$, since the plain enumeration of the perfect matchings of this region is up to the factor $2^n$ equal to the $2$-enumeration of order $n$
$\asm$s as mentioned above. A bijection between the set of perfect matchings of this region and the set of perfect matchings of the family of graphs $(\mathcal{Q}_n - V)_{V \subseteq \{1,2,\ldots,n-1,n+1,\ldots,2n-1\}}$ might give a hint on the bijection between $\asm$s and $\ast$s.

\section{Maximal number of $1$'s: Quasi alternating sign triangles}
\label{dast}

The main purpose of this section is to provide a proof of Theorem~\ref{maxone} and generalizations.

\subsection{The $\qast$ partition function}

\begin{theo}
\label{dastpartfunct}
We have the following determinant formula for the partition function $Z^{\uparrow}_n(u_1,\ldots,u_n;\allowbreak u_{n+1})_{\qast}$.
\begin{multline}
\label{dastpartfunctexpr}
\frac{\prod_{j=1}^{n}\sigma( \bar p u_j) \prod_{i=1}^{n} \prod_{j=1}^{n+1} \sigma(q^2 u_i u_j) \sigma(q^2 \bar{u_i} \bar u_j)}{\sigma(q^2)^{2n} \sigma(q^4)^{n^2} \prod_{i=1}^{n} \sigma(u_i \bar u_{n+1}) \prod_{1 \le i < j \le n} \sigma(u_i \bar u_j)^2} \\
\times
\det_{1\le i,j\le n+1}\left( \Large \begin{cases} \frac{1}{\sigma(q^2u_iu_j)}+
\frac{1}{\sigma(q^2 \bar u_i \bar u_j)},& i \le n \\
\frac{\sigma(\bar p u_{n+1})}{\sigma( \bar p u_j)}, & i=n+1  \end{cases} \right)
\end{multline}
\end{theo}

\begin{proof} We use Theorem~\ref{un}.
The case $n=1$ is easy to verify.
To show that \eqref{dastpartfunctexpr} is indeed a Laurent polynomial in $u_{n+1}$, observe that
$$
\frac{\prod_{j=1}^{n}\sigma( \bar p u_j) \prod_{i=1}^{n} \prod_{j=1}^{n+1} \sigma(q^2 u_i u_j) \sigma(q^2 \bar{u_i} \bar u_j)}{\sigma(q^2)^{2n} \sigma(q^4)^{n^2} \prod_{1 \le i < j \le n} \sigma(u_i \bar u_j)^2}
\det_{1\le i,j\le n+1}\left( \Large \begin{cases} \frac{1}{\sigma(q^2u_iu_j)}+
\frac{1}{\sigma(q^2 \bar u_i \bar u_j)},& i \le n \\
\frac{\sigma(\bar p u_{n+1})}{\sigma( \bar p u_j)}, & i=n+1  \end{cases} \right)
$$
is a Laurent polynomial in $u_{n+1}$.  It is divisible by $\prod_{i=1}^{n} \sigma(u_i \bar u_{n+1})$, as the determinant vanishes if we set $u_{n+1}= \pm u_j$, for $j=1,\ldots,n$, since then the $j$-th column coincides with the $(n+1)$-st column of the matrix up to sign.

We check the properties from Theorem~\ref{un}.

\emph{\eqref{1degree}:} To deduce the bounds on the degree and order as a Laurent polynomial in
$u_{n+1}$, use the Leibniz formula for the determinant and observe that, after multiplying each summand with the prefactor and expanding, we obtain a sum of rational functions where each numerator has degree $2n$ and order $-2n$, and each denominator has degree $n$ and order $-n$.

\emph{\eqref{1sym}, \eqref{1inv} and \eqref{1even}:} The expression is readily checked to be symmetric in $u_1,\ldots,u_{n}$,  invariant under the transformation $(u_1,\ldots,u_{n+1},p) \mapsto (\bar u_1, \ldots, \bar u_{n+1}, \bar p)$ and even in $u_i$, $i=1,2,\ldots,n$. Note that interchanging respective left and right boundary constants in the $\qast$-specialization (see Subsection~\ref{dast-spec}) is
equivalent to the replacement $p \mapsto \bar p$.

\emph{\eqref{1eval}:} As $\wrin{u_1}=0$, $Z_n^{\downarrow}(u_1,\ldots,u_n;u_{n+1})$ disappears on the right-hand side of \eqref{up}.
Let $X_n(u_1,\ldots,u_n;u_{n+1})$ denote \eqref{dastpartfunctexpr}. Then we need to show the following.
$$
X_n(u_1,\ldots,u_n;  q^2 \bar u_1) =  \frac{\sigma( \bar p q^2 \bar u_1) \sigma(q^2  u_1^2)}
{\sigma(q^2)^2} \prod_{i=2}^{n} \frac{\sigma(q^2 u_1 u_i) \sigma(q^4 \bar u_1  u_i)}{\sigma(q^4)^2} X_{n-1}(u_2,\ldots,u_n; u_1)
$$
For this purpose, multiply the $(n+1)$-st column of the matrix of $X_n(u_1,\ldots,u_n;  q^2 \bar u_1)$ with
the prefactor $\sigma(q^2 u_1 u_{n+1}) \sigma(q^2 \bar u_1 \bar u_{n+1})$. After taking $u_{n+1} =  q^2 \bar u_1$, the $(n+1)$-st column is
$(\sigma(q^4),0,\ldots,0)^t$, and now expand the determinant with respect to this column.
\end{proof}

\subsection{Specialization of the  $\qast$ partition function at $u_{n+1}=p$}

Taking $u_{n+1} \to p$ in \eqref{dastpartfunctexpr}, the last row of the matrix becomes
$(0,0,\ldots,0,1)$, and we obtain the following simplified formula.

\begin{cor}
\label{dastpartfunct1}
We have the following formula for $Z^{\uparrow}_n(u_1,\ldots,u_n;p)_{\qast}$.
\begin{equation}
\label{dastpartfunct1expr}
\frac{\prod_{i=1}^n\sigma(q^2 p u_i)\sigma(q^2 \bar p \bar u_i)\,
\prod_{i,j=1}^n\sigma(q^2u_iu_j)\sigma(q^2 \bar u_i \bar u_j)}
{\sigma(q^2)^{2n}\,\sigma(q^4)^{n^2}\,\prod_{1\le i<j\le n}\sigma(u_i \bar u_j)^2}
\det_{1\le i,j\le n}\biggl(  \frac{1}{\sigma(q^2u_iu_j)}+
\frac{1}{\sigma(q^2 \bar u_i \bar u_j)} \biggr)
\end{equation}
\end{cor}

Here we again sketch an alternative proof of Corollary~\ref{dastpartfunct1} which avoids Theorem~\ref{dastpartfunct} and Lemma~\ref{eval}.

\begin{proof}[Sketch of an alternative proof using Theorem~\ref{u1}]
We need to verify $Z^{\uparrow}_n(\bar p q^2,u_2,\ldots,u_n;p)_{\qast}=0$. By the symmetry of Theorem~\ref{sym}, it suffices to show that
$Z^{\uparrow}_n(u_1,\ldots,u_{n-1},\bar p q^2;p)_{\qast}=0$. Assume $u_n=\bar p q^2$ and $u_{n+1}=p$, and consider the bottom bulk vertex $v$ in the central column. Its label is $u_{n} u_{n+1} = q^2$, and since $\wnw{q^2} =0$, there are only two possible local configurations around $v$ that contribute a non-zero weight (namely $\mone$ and $\ne$). However, this implies that there are only two
possible configurations for the right boundary vertex $v_R$ adjacent to $v$ (the vertical edge connecting $v$ and $v_R$ is oriented from $v$ to $v_R$). The assertion follows as $\wrin{u_n} = \wrone{u_n}=0$.

The cases $n=1,2$ are easy to verify. To show that \eqref{dastpartfunct1expr} is a Laurent polynomial in $u_1,\ldots,u_{n}$, proceed as in the footnote of the proof of  Theorem~\ref{astpartfunct}.

\emph{\eqref{2degree}:} We use the Leibniz rule for the determinant. The expression is a sum of rational functions with degrees and orders of the numerators in $u_1$ being $4n$ and $-4n$, respectively, while the degrees and orders of the denominators are $2n-2$ and $-2n+2$, respectively.

\emph{\eqref{2sym},\eqref{2inv},\eqref{2even}, \eqref{2eval} and \eqref{2extra}:} The expression is obviously symmetric in $u_1,\ldots,u_n$, invariant under the transformation
$(u_1,\ldots,u_n,p) \mapsto (\bar u_1,\ldots,\bar u_n, \bar p)$ and even in $u_i$, $i=1,2,\ldots,n$.
The identities from Lemmas~\ref{eval1} and \ref{eval2} can be verified straightforwardly.
It is also obvious that the expression in \eqref{dastpartfunct1expr} has a zero at
$u_1 = \bar p q^2$.
\end{proof}

\subsection{Specialization of the $\qast$ partition function at $u_{n+1}=p$ and $q=e^{\frac{i \pi}{6}}$}

\begin{theo}
\label{dastschurtheo}
We have the following Schur function expression for $Z^{\uparrow}_n(q,u_1,\ldots,u_n;p)_{\qast}$ at $q=e^{\frac{i \pi}{6}}$.
\begin{equation}
\label{dastschur}
3^{-\binom{n+1}{2}} \prod_{i=1}^{n} (p ^2 u_i^2 + 1 + \bar p^2 \bar u_i^2)   \schur_{(n,n-1,n-1,n-2,n-2,\ldots,1,1)}(u_1^2,\bar u_1^2,\ldots,u_n^2, \bar u_n^2)
\end{equation}
\end{theo}

\begin{proof}
Use \eqref{dastpartfunct1expr} and observe that this expression is a specialization of one of two factors of the partition function of half-turn symmetric $\asm$s as noted in \cite[Theorem~10]{KuperbergRoof}. The result follows from \cite[Theorem~2.4 (2), 2nd eq.]{OkadaCharacters} or from \cite[Eq. (11), Eq. (31)]{StroReloaded,RazStroHTSymm}.
\end{proof}

Formulas for the specializations of
$Z_n(u_1,\ldots,u_n;p)_{\qast}$ at $q=e^{\frac{i \pi}{4}},e^{\frac{i \pi}{8}}$ are also provided in \cite[Theorem 2.4 (2)]{OkadaCharacters}. They correspond to $x$-enumerations of $\qast$s for $x=0,2$, respectively.

\begin{rem}
In \cite{AyyBehFis16b} it was shown that \eqref{dastschur} factorizes as
\begin{equation}
\label{DASTSchurfact}
\begin{aligned}
\schur_{(2n,2n-1,2n-1,2n-2,2n-2,\ldots,1,1)}(u_1,\bar u_1,u_2,\bar u_2,\ldots,u_{2n}, \bar u_{2n})
&= \\ (-1)^n \so_{(n,n-1,n-1,\ldots,1,1)}(u_1,&\ldots,u_{2n}) \so_{(n,n-1,n-1,\ldots,1,1)}(-u_1,\ldots,-u_{2n}), \\
\schur_{(2n+1,2n,2n,\ldots,1,1)}(u_1,\bar u_1,u_2,\bar u_2,\ldots,u_{2n+1},
\bar u_{2n+1})
&= \\ \sp_{(n,n-1,n-1,\ldots,1,1)}(u_1,&\ldots,u_{2n+1}) \oeven_{(n+1,n,n,\ldots,1,1)}(u_1,\ldots,u_{2n+1}).
\end{aligned}
\end{equation}
\end{rem}

\subsection{Proof of Theorem~\ref{maxone}}
We set $u_i=1$, $i=1,\ldots,n$, and $p=1$ in \eqref{dastschur}  to finally prove Theorem~\ref{maxone}. We use the following formula for the specialization of Schur functions.
\begin{equation}
\label{hookcontent}
\schur_{\lambda}(1,1,\ldots,1) = \prod_{1 \le i < j \le n} \frac{\lambda_i - \lambda_j +j-i}{j-i}
\end{equation}
This is, by the combinatorial interpretation of the Schur function as the generating function of semistandard tableaux of shape $\lambda$, also the number of these tableaux. It follows that
$$
3^{-\binom{n}{2}}\schur_{(n,n-1,n-1,n-2,n-2,\ldots,1,1)}(1^{2n})=\prod_{i=0}^{n-1} \frac{(3i+2) (3i)!}{(n+i)!}
$$
and the latter is also the number of cyclically symmetric plane partitions in an $n \times n \times n$ box, see \cite{AndrewsMacdonald}.

\subsection{Outlook}
A natural question to ask is whether there is a generalization of
Theorem~\ref{maxone} that is analogous to Theorem~\ref{maxmone}, i.e., what is a statistic on $n \times n \times n$ $\cspp$s that has the same distribution as the numbers of $-1$'s in the fundamental domain of $(2n+1) \times (2n+1)$ $\dasasm$s $A$ with $\n_{1}(A)=n+1$? In a forthcoming paper, we will show that this is the case for a statistic that was already introduced by Mills, Robbins and Rumsey in \cite[p. 47]{MRR87}, namely the number of special parts in a $\cspp$. In fact, this result will follow from a common generalization of Theorems~\ref{maxmone} and \ref{maxone} that involves an infinite family of alternating sign matrix objects, among which the objects from Theorems~\ref{maxmone} and \ref{maxone} are two special members. In addition, certain classes of extreme diagonally and antidiagonally symmetric alternating sign matrices of \emph{even order} will also be members of this family.

With regard to this family, it will also make sense to consider objects with vertical symmetry, and we will enumerate these symmetry classes. A special case of this result will be that vertically symmetric $\ast$s of order $n$ are equinumerous with vertically symmetric $\asm$s of order $n$. More generally, we will see that the generating functions of vertically symmetric $\asm$s and vertically symmetric $\ast$s with respect to the numbers of $-1$'s coincide.

\section{Maximal number of $0$'s: Off-diagonally and off-antidiagonally symmetric alternating sign matrices}
\label{qoosasm}

In this section, we provide the proof of Theorem~\ref{maxone} and generalizations.

\subsection{The $\qoosasm$ partition function}

Recall the definition of the \emph{Pfaffian} of a triangular array $A=(a_{i,j})_{1 \le i < j \le 2n}$
$$
\underset{1 \le i < j \le 2n}{\pf}(a_{i,j}) = \sum_{\pi = \{\{i_1,j_1\},\{i_2,j_2\},\ldots,\{i_n,j_{n}\}\} \atop \pi \text{ perfect matching of } K_{2n},  i_k < j_k} \sgn \pi \prod_{k=1}^{n} a_{i_k,j_k},
$$
where $\sgn \pi$ is the sign of the permutation $i_1 j_1 i_2 j_2 \ldots i_n j_n$.
Suppose $\widehat{A}=(a_{i,j})_{1 \le i, j \le 2n}$ is the skew symmetric extension of $A$. Then it is a well-known fact that
$$
\left(\underset{1 \le i < j \le 2n}{\pf}(a_{i,j}) \right)^2 = \det_{1 \le i, i \le 2n} (a_{i,j}).
$$

\begin{theo}
\label{qoosasmpartfunct}
We have the following Pfaffian formula for the $\qoosasm$ partition function
$$
Z_{n}(u_1,\ldots,u_n;u_{n+1})_{\qoosasm} =
P_{\lceil \frac{n}{2} \rceil}(u_1,\ldots,u_{2 \lceil \frac{n}{2} \rceil}) \, Q_{\lceil \frac{n+1}{2} \rceil}(u_1,\ldots,u_{2 \lceil \frac{n+1}{2} \rceil -1}),
$$
where
$$
P_m(u_1,\ldots,u_{2m}) = \sigma(q^4)^{-(m-1)2m} \prod_{1\le i<j\le2m}\frac{\sigma(q^2u_iu_j)\,\sigma(q^2 \bar u_i \bar u_j)}{\sigma(u_i \bar u_j)}
\underset{1\le i<j\le2m}{\pf}\left(\frac{\sigma(u_i \bar u_j)}{\sigma(q^2u_iu_j)\,\sigma(q^2 \bar u_i \bar u_j)}\right)
$$
and
\begin{multline*}
Q_m(u_1,\ldots,u_{2m-1}) = \sigma(q^4)^{-(m-1)(2m-1)} \prod_{1\le i<j\le2m-1}\frac{\sigma(q^2u_iu_j)\,\sigma(q^2
\bar u_i \bar u_j)}{\sigma(u_i \bar u_j)} \\   \times
\underset{1\le i<j\le2m}{\pf}
\left( \Large
\begin{cases} \frac{\sigma(u_i \bar u_j)}{\sigma(q^2u_iu_j)} + \frac{ \sigma(u_i \bar u_j)}{\sigma(q^2 \bar u_i \bar u_j)}, & j<2m \\
1, & j=2m \end{cases} \right).
\end{multline*}
\end{theo}

Curiously, $P_n(u_1,\ldots,u_{2n})$ is the partition function of $2n \times 2n$
$\operatorname{OSASM}$s using the bulk weights from Table~\ref{weights} and normalizing the boundary weights as in the
$\qoosasm$ case (i.e., all boundary weights are $1$).
To see this, set $a=q^2$ in Kuperberg's formula \cite[Theorem 10]{KuperbergRoof} and divide by $\sigma(q^4)^{\binom{2n}{2}}$ to account for the different normalization of the bulk weights.

\begin{proof}
We use Theorem~\ref{un}. The case $n=1$ is easy to verify.

Both $P_m(u_1,\ldots,u_{2m})$ and $Q_m(u_1,\ldots,u_{2m-1})$ are Laurent polynomials as the two Pfaffians involved are antisymmetric functions in
the variables $u_1,\ldots,u_{2m}$ and $u_1,\ldots,u_{2m-1}$, respectively, and odd in each $u_i$.

\emph{\eqref{1degree}:}
By the definition of the Pfaffian, $P_m(u_1,\ldots,u_{2m})$ is a sum of rational functions where the degrees and orders of the numerators in $u_1$ are $4m-4$ and $-4m+4$, respectively, and of the denominators
$2m-2$ and $-2m+2$, respectively. Hence the bounds for the degree and order of $P_m(u_1,\ldots,u_{2m})$ are $2m-2$ and $-2m+2$, respectively.  This shows that, in case of $n$ is odd, the degree and order in $u_{n+1}$ of the proposed expression for the partition function are at most
$n-1$ and at least $-n+1$, respectively ($m=\lceil \frac{n}{2} \rceil$).

On the other hand, the degrees (resp.\ orders) of the numerators of the summands of $Q_m(u_1,\ldots,\allowbreak u_{2m-1})$ in $u_1$ are either
$4m-4$ (resp.\ $-4m+4$) or $4m-5$ (resp.\ $-4m+5$), depending on whether or not $1$ is matched to $2m$ and the degrees (resp.\ orders) of the denominators are $2m-2$ (resp. $-2m+2$) and $2m-3$ (resp.\ $-2m+3$). Thus, the degree and
order of $Q_m(u_1,\ldots,u_{2m-1})$ are at most $2m-2$ and at least $-2m+2$, respectively. This shows that, in case of $n$ is even, the degree and order in $u_{n+1}$ of the proposed expression for the partition function are at most $n$ and at least $-n$, respectively ($m=\lceil \frac{n+1}{2} \rceil$).

\emph{\eqref{1sym}, \eqref{1inv} and \eqref{1even}:} The expression in the statement is symmetric in $u_1,\ldots,u_n$ as it is the product of two symmetric functions in $u_1,\ldots,u_n$. It is also invariant under the replacement $(u_1,\ldots,u_n) \mapsto (\bar u_1,\ldots, \bar u_{n+1})$ as $P_m(u_1,\ldots,u_{2m})$ is invariant under  $(u_1,\ldots,u_{2m}) \mapsto (\bar u_1,\ldots, \bar u_{2m})$ and $Q_m(u_1,\ldots,\allowbreak u_{2m-1})$
is invariant under $(u_1,\ldots,u_{2m-1}) \mapsto (\bar u_1,\ldots, \bar u_{2m-1})$, and even in $u_i$, $i=1,2,\ldots,n$.

\emph{\eqref{1eval}:} Let $X_n(u_1,\ldots,u_n;u_{n+1})$ denote the expression in the statement. Then---by Lemma~\ref{eval} and by taking into account the normalization we have chosen---we need to show
$$
X_n(u_1,\ldots,u_{n}; q^2 \bar u_1) = \prod_{i=2}^{n} \frac{\sigma(q^2 u_1 u_i) \sigma(q^4 \bar u_1 u_i)}{\sigma(q^4)^2}
X_{n-1}(u_2,\ldots,u_n; u_1).
$$
Thus it suffices to show
\begin{align*}
P_n(u_1,\ldots,u_{2n-1},  q^2 \bar u_1) & =
\prod_{i=2}^{2n-1} \frac{\sigma(q^2  u_1 u_i) \sigma(q^4 \bar u_1 u_i)}{\sigma(q^4)^2} P_{n-1}(u_2,\ldots,u_{2n-1}), \\
Q_n(u_1,\ldots,u_{2n-2},  q^2 \bar u_1) & =
\prod_{i=2}^{2n-2} \frac{\sigma(q^2 u_1 u_i) \sigma(q^4 \bar u_1 u_i)}{\sigma(q^4)^2} Q_{n-1}(u_2,\ldots,u_{2n-2})
\end{align*}
and this follows easily from the definition of the Pfaffian.
\end{proof}

\subsection{Specialization of the $\qoosasm$ partition function at $q=e^{\frac{i \pi}{6}}$}

Here we obtain a formula involving \emph{symplectic characters} $\sp_{\lambda}(x_1,\ldots,x_n)$ when specializing at $q=e^{\frac{i \pi}{6}}$. (See \cite[§ 24.2]{fultonharris} for a reference on symplectic characters.)
We use the determinantal formula
\begin{equation}
\label{symp_det}
\sp_{\lambda}(x_1,\ldots,x_n) = \frac{\det\limits_{1 \le i, j \le n} \left( x_i^{\lambda_j+n-j+1} - \bar x_i^{\lambda_j+n-j+1} \right)}{\prod_{i=1}^{n} (x_i - \bar x_i) \prod_{1 \le i < j \le n} (x_i + \bar x_i - x_j - \bar x_j)}
\end{equation}
if $\lambda=(\lambda_1,\lambda_2,\ldots,\lambda_n)$, and extend the partition with zero parts if its length is less than $n$.

Symplectic characters are also generating functions of certain tableaux:
A \emph{symplectic tableaux} of shape $\lambda$ is a semistandard tableaux of shape $\lambda$ with entries $1 < 1' < 2 < 2' < 3 < \ldots$ and the additional constraint that entries in row $i$ are no smaller than $i$,
see \cite[Theorem~2.3]{sundaram1990}. Now
$$
\sp_{\lambda}(x_1,\ldots,x_n) = \sum_{T} \prod_{i=1}^{n} x_i^{(\# \, i \in T) - (\# \, i' \in T)},
$$
where the sum is over all symplectic tableaux of shape $\lambda$ with entries in $\{1,1',2,2',\ldots,n,n'\}$.

\begin{theo}
\label{spq}
Assume $q=e^{\frac{i \pi}{6}}$.
We have the following formula for
$Z_{2n-1}(u_1,\ldots,u_{2n-1}; u_{2n})_{\qoosasm}$ in terms of symplectic characters
\begin{equation}
\label{sympodd}
3^{-(n-1)(2n-1)} \sp_{(n-1,n-1,n-2,n-2,\ldots,1,1)}(u_1^2,\ldots,u_{2n}^2)
\sp_{(n-1,n-2,n-2,n-3,n-3,\ldots,1,1)}(u_1^2,\ldots,u_{2n-1}^2)
\end{equation}
and for $Z_{2n}(u_1,\ldots,u_{2n}; u_{2n+1})_{\qoosasm}$
\begin{equation}
\label{sympeven}
3^{-n(2n-1)} \sp_{(n-1,n-1,n-2,n-2,\ldots,1,1)}(u_1^2,\ldots,u_{2n}^2)
\sp_{(n,n-1,n-1,n-2,n-2,\ldots,1,1)}(u_1^2,\ldots,u_{2n+1}^2).
\end{equation}
\end{theo}

We use Theorem~\ref{qoosasmpartfunct} to prove the theorem. As $P_n(u_1,\ldots,u_{2n})$ is the partition function of $\operatorname{OSASM}$s, we can take its specialization
at $q=e^{\frac{i \pi}{6}}$ from \cite[Theorem~2.5 (2)]{OkadaCharacters} where it was shown that
$$
P_n(u_1,\ldots,u_{2n})|_{q=e^{\frac{i \pi}{6}}}= 3^{-(n-1)n} \sp_{(n-1,n-1,n-2,n-2,\ldots,1,1)}(u_1^2,\ldots,u_{2n}^2).
$$
See also \cite[Theorem~5]{RazStroVSASM}.
(In \cite[Theorem~2.5 (2)]{OkadaCharacters} appears also a formula for the specialization of $P_n(u_1,\ldots,u_{2n})$ at $q=e^{\frac{i \pi}{4}}$.)
We still need to show
\begin{equation}
\label{qooq}
Q_n(u_1,\ldots,u_{2n-1})|_{q=e^{\frac{i \pi}{6}}} = 3^{-(n-1)^2}
\sp_{(n-1,n-2,n-2,n-3,n-3,\ldots,1,1)}(u_1^2,\ldots,u_{2n-1}^2).
\end{equation}

The Pfaffian identity stated next is useful in the following.
\begin{lem}
\label{lem:OkadaType}
We define
$$\mathcal{W}(x_1,\ldots,x_n;a_1,\ldots,a_n)= \det_{1 \le i, j \le n} \left( x_i^{j-1} + a_i \, x_i^{n-j} \right).$$ Then
\begin{equation}
\label{OkadaType}
\underset{1\le i < j \le 2n}{\pf} \left( \Large \begin{cases} \frac{\mathcal{W}(x_i,x_j;a_i,a_j)(1+a_i a_j)}{1-x_i x_j}, &
j<2n \\ 1-a_i^2, & j=2n \end{cases} \normalsize \right) = \frac{\mathcal{W}(x_1,\ldots,x_{2n-1};
-a_1^{2},-a_{2}^2,\ldots,-a_{2n-1}^{2})}{\prod_{1 \le i < j \le 2n-1} (1-x_i x_j)}.
\end{equation}
\end{lem}

\begin{proof}
The following special case of an identity due to Okada
\cite[Theorem 3.4., Eq. (20)]{OkadaCharacters} is applied.
\begin{equation}
\label{OkadaId}
\underset{1\le i < j \le 2n-2}{\pf} \left(
\frac{ \mathcal{W}(x_i,x_j,x_{2n-1};a_i,a_j,a_{2n-1})}{1-x_i x_j} \right) =
\frac{(1+a_{2n-1})^{n-2}}{\prod_{1 \le i < j \le 2n-2} (1 - x_i x_j)} \mathcal{W}(x_1,\ldots,x_{2n-1};a_1,\ldots,a_{2n-1})
\end{equation}
(Substitute $n$ for $n-1$, and set $b_1=\ldots=b_{2n-2}=0$, $z=x_{2n-1}$, $c=a_{2n-1}$ in Okada's identity.)
After replacing $a_i$ by $-a_i^2$ in \eqref{OkadaId}, the right-hand side is---up to the factor $(1-a_{2n-1}^2)^{n-2} \prod_{i=1}^{2n-2} (1-x_i x_{2n-1})$---equal to the right-hand side in the statement. It suffices to show
\begin{multline}
\label{oka}
(1-a_{2n-1}^2)^{n-2} \prod_{i=1}^{2n-2} (1-x_i x_{2n-1}) \underset{1\le i < j \le 2n}{\pf} \left( \Large \begin{cases} \frac{\mathcal{W}(x_i,x_j;a_i,a_j)(1+a_i a_j)}{1-x_i x_j}, &
j<2n \\ 1-a_i^2, & j=2n \end{cases} \normalsize \right)  \\
= \underset{1\le i < j \le 2n-2}{\pf} \left(
\frac{ \mathcal{W}(x_i,x_j,x_{2n-1};-a_i^2,-a_j^2,-a_{2n-1}^2)}{1-x_i x_j} \right).
\end{multline}
The fact that elementary row and column operations do not change the determinant of a matrix has an analogue for Pfaffians: For instance, suppose $A$ is an even-order skew-symmetric matrix, and $A'$ is obtained from $A$ by adding a multiple of row $i$ to row $j$, and simultaneously adding the same multiple of column $i$ to column $j$, for some $i \not= j$. Then $\pf(A) = \pf(A')$. The identity \eqref{oka} can now be proven using such operations.

More specifically, we use the identity
\begin{equation}\label{PfId}
c_{2n-1,2n}^{n-2}\:\underset{1\le i<j\le2n}{\pf}(c_{ij})=\underset{1\le i<j\le2n-2}{\pf}\left(\pf\begin{pmatrix}c_{i,j}&c_{i,2n-1}&c_{i,2n}\\
&c_{j,2n-1}&c_{j,2n}\\
&&c_{2n-1,2n}\end{pmatrix}\right),\end{equation}
for any $(c_{i,j})_{1\le i<j\le 2n}$, which is  a special case of a Pfaffian analogue of Sylvester's determinant identity. (Replace $n$ by $n-1$, and set
$\beta=\{1,\ldots,2n-2\}$ and $\alpha=\{2n-1,2n\}$ in an identity of Knuth~\cite[Eq.~(2.5)]{Knu96}.) Alternatively,~\eqref{PfId} can be obtained directly by transforming the first $2n-2$ entries in the last column of $(c_{i,j})_{1\le i<j\le 2n}$ to zeros
using certain row and column operations, and then
expanding along the last column. Now let $(c_{i,j})_{1\le i<j\le 2n}$ be the array in the Pfaffian on the left-hand side of~\eqref{oka}.
It can be checked straightforwardly that
\begin{equation*}\pf\begin{pmatrix}c_{i,j}&c_{i,2n-1}&c_{i,2n}\\
&c_{j,2n-1}&c_{j,2n}\\&&c_{2n-1,2n}\end{pmatrix}=\frac{\mathcal{W}(x_i,x_j,x_{2n-1};-a_i^2,-a_j^2,-a_{2n-1}^2)}
{(1-x_i x_j)(1-x_ix_{2n-1})(1-x_jx_{2n-1})}.\end{equation*}
The result now follows from \eqref{PfId}.
\end{proof}

\begin{proof}[Proof of Theorem~\ref{spq}]
We prove \eqref{qooq}. Assume  $q=e^{\frac{i \pi}{6}}$. As
$$
\sigma(q^4)= \sqrt{3} i, \quad \sigma(q^2 x) \sigma(q^2 \bar x) = - (1+x^2+\bar x^2), \quad \sigma(q^2 x) + \sigma(q^2 \bar x) = \sqrt{3} i (x + \bar x),
$$
we have that $Q_n(u_1,\ldots,u_{2n-1})$ is equal to
$$
(-3)^{-(n-1)^2} \prod_{1 \le i < j \le 2n-1}
\frac{1 + u_i^2 u_j^2 + \bar u_i^2 \bar u_j^2}{u_i \bar u_j - \bar u_i u_j}
\underset{1 \le i < j \le 2n}{\pf}
\left( \Large \begin{cases} \frac{u_i^2 - \bar u_i^2 - u_j^2 + \bar u_j^2}{1 + u_i^2 u_j^2 + \bar u_i^2 \bar u_j^2},
& j < 2n \\ 1, & j=2n \end{cases} \normalsize \right).
$$
By setting $x_i=u_i^6$ and $a_i=-u_i^2$ in Lemma~\ref{lem:OkadaType}, we see that the left-hand side of \eqref{OkadaType} is, up to the factor $\prod_{i=1}^{2n-1} (1-u_i^4)$ (which arises from a factor $b_ib_j$ in each entry of the triangular array, where $b_i=1/(1-u_i^4)$ for $i<2n$ and $b_{2n}=1$), equal to the Pfaffian in the previous expression. This implies that the previous expression is equal to
\begin{multline*}
(-3)^{-(n-1)^2} \prod_{i=1}^{2n-1} \frac{1}{1-u_i^4} \prod_{1 \le i < j \le 2n-1}
\frac{1 + u_i^2 u_j^2 + \bar u_i^2 \bar u_j^2}{(u_i \bar u_j - \bar u_i u_j)(1-u_i^6 u_j^6)} \det_{1 \le i, j \le 2n-1} \left(u_i^{6j-6} - u_{i}^{12n-4-6j} \right) \\
= (-1)^n\,3^{-(n-1)^2}\prod_{i=1}^{2n-1} \frac{u_i^{6n-4}}{1-u_i^4} \prod_{1 \le i < j \le2n-1}
\frac{1 + u_i^2 u_j^2 + \bar u_i^2 \bar u_j^2}{(u_i \bar u_j - \bar u_i u_j)(1-u_i^6 u_j^6)} \\
\times 
\det_{1 \le i, j \le 2n-1} \left(u_i^{2(3n-3j+1)}-\bar u_i^{2(3n-3j+1)} \right).
\end{multline*}
The required right-hand side of~\eqref{qooq} can now be obtained from this expression by
multiplying columns $n+1,\ldots,2n-1$ of the matrix by $-1$,
reordering the columns as $1,2n-1,2,2n-2,\ldots,n-1,n+1,n$, simplifying the prefactor, and
applying the determinant formula~\eqref{symp_det} for symplectic characters
\end{proof}

\begin{rem}
A different approach to proving Theorem~\ref{spq} would be to show that the function in this theorem fulfills the properties from Theorem~\ref{un} in the special case $q=e^{\frac{i \pi}{6}}$. (The same idea would clearly also lead to direct proofs of Theorems~\ref{astschurtheo} and \ref{dastschurtheo}, where in this case we would have to show that the properties from Theorem~\ref{u1} are satisfied. This approach was previously used, see for instance \cite{StroNewWay}.) It turns out that it suffices to show
\begin{align*}
\frac{\sp_{(n-1,n-1,n-2,n-2,\ldots,1,1)}(u_1,u_2,\ldots,u_{2n-1},q^4 \bar u_1)}{\sp_{(n-2,n-2,\ldots,1,1)}(u_2,u_3,\ldots,u_{2n-1})} &= \prod_{i=2}^{2n-1} ({\bar q}^2 u_1 +  q^2 \bar u_1 + u_i + \bar u _i), \\
\frac{\sp_{(n,n-1,n-1,\ldots,1,1)}(u_1,u_2,\ldots,u_{2n}, q ^4 \bar u_1)}{\sp_{(n-1,n-2,n-2,\ldots,1,1)}(u_2,u_3,\ldots,u_{2n})} &= \prod_{i=2}^{2n} ({\bar q}^2 u_1 +  q^2 \bar u_1 + u_i + \bar u _i)
\end{align*}
at $q=e^{\frac{i \pi}{6}}$.
Since $\sp_{\lambda}(u_1,\ldots,u_n)$ is a certain generating function of symplectic tableaux, it would be interesting to explore whether these identities have a combinatorial proof.
\end{rem}

\subsection{Proof of Theorem~\ref{maxzero}} We set $u_i=1$ for $i=1,\ldots,n$ in \eqref{sympodd} and \eqref{sympeven} to compute the number of $\qoosasm$s. Here we need the following product formula, see \cite[Exercise 24.20]{fultonharris}.
$$
\sp_{\lambda}(1,1,\ldots,1) = \prod_{i=1}^{n} \frac{\lambda_i+n+1-i}{n+1-i}
\prod_{1 \le i < j \le n} \frac{(\lambda_i - \lambda_j + j - i)(\lambda_i + \lambda_j +2n+2-i-j)}{(j-i)(2n+2-i-j)}
$$
We obtain
$$
| \qoosasm(4n-1)| = \prod_{i=0}^{n-1} \frac{(3i+2)! (3n+3i)!}{(2n+i)! (3n+i)!},
$$
and this is also the number of $(4n+1) \times (4n+1)$ $\vhsasm$s. Finally,
$$
| \qoosasm(4n+1)| = \prod_{i=1}^{n} \frac{(3i-1)! (3n+3i)!}{(2n+i)! (3n+i+1)!},
$$
and this is also the number of $(4n+3) \times (4n+3)$ $\vhsasm$s, see \cite{OkadaCharacters}.

\section*{Acknowledgement}
We thank Matja\v{z} Konvalinka for very helpful discussions.

The authors acknowledge hospitality and support from the Galileo Galilei Institute, Florence, Italy, during the programme
``Statistical Mechanics, Integrability and Combinatorics'' held in May--June 2015. The first author (A.A.) is partially supported by UGC Centre for Advanced Studies and acknowledges support from DST grant DST/INT/SWD/VR/P-01/2014.
The third author (I.F.) acknowledges support from the Austrian Science Foundation FWF, START grant Y463.

\begin{appendix}

\section{$\asm$s and $\ast$s with a single $-1$}
\label{k1}

We count $n \times n$ $\asm$s with precisely one $-1$:
Such matrices are uniquely determined by the following information.
\begin{itemize}
\item The columns of the unique $-1$ and the two $1$'s that are situated in the same row. There are $\binom{n}{3}$ choices.
\item The rows of the $-1$ and the two $1$'s that are situated in the same column as the $-1$. There are $\binom{n}{3}$ choices.
\item The permutation matrix of order $n-3$ that is obtained by deleting the three chosen rows and the three chosen columns. There are $(n-3)!$ choices.
\end{itemize}
Therefore, there is a total of $\binom{n}{3}^2 (n-3)!$ $\asm$s of order $n$ with precisely one $-1$. We note that formulas for the numbers of $\asm$s with more $-1$'s are given in \cite{aval} and \cite{le-gac-2011}, and that $\asm$s with one $-1$ were studied by Lalonde~\cite{Lal02,Lal06}.

\begin{figure}
\scalebox{0.65}{
\psfrag{1}{\Huge $1$}
\psfrag{-1}{\Huge$-1$}
\psfrag{a}{\Large $-j_1$}
\psfrag{b}{\Large $j_2$}
\psfrag{c}{\Large $j_3$}
\psfrag{d}{\Large $i_1$}
\psfrag{e}{\Large $i_2$}
\includegraphics{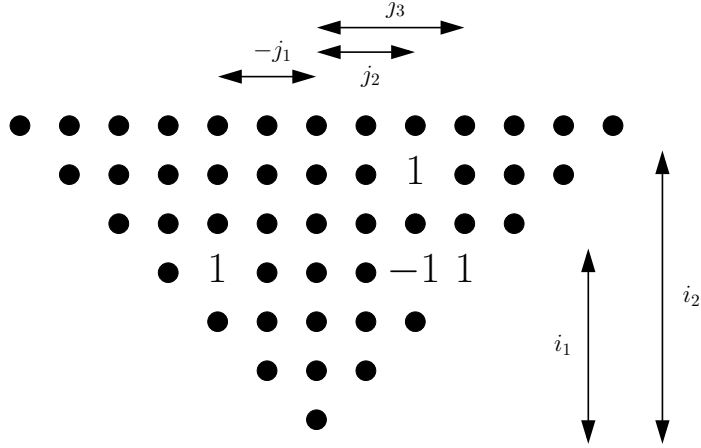}}
\caption{\label{onem1} Definition of $i_1,i_2,j_1,j_2,j_3$}
\end{figure}

On the other hand, $\ast$s with precisely one $-1$ are less accessible: Let $i_1 \in \{2,3,\ldots,n-1\}$ be the row of the
$-1$, counted from the bottom, and $j_2 \in \{-n+2,-n+3,\ldots,n-2\}$ its column, counted from the central column and where we use the negative sign if the column is left of the central column. Let $j_1,j_3, i_2$ encode the positions of the three $1$'s to the left, to the right and above the $-1$ as indicated in Figure~\ref{onem1}.
For integers $a,b$, we use the following notation.
$$
p(a,b) = \begin{cases} a(a+1) \cdots b & \text{if $a \le b$,} \\
                         1 & \text{otherwise.} \end{cases}
$$
We count all possible subarrays consisting of the $i_1-1$ bottom rows of such an $\ast$: Since there is neither a $1$ in column $j_1$ nor in column $j_3$, the number is
\begin{equation}
\label{total}
p(1,\min(|j_1|,|j_3|)) p(\min(|j_1|,|j_3|),\max(|j_1|,|j_3|)-1) p(\max(|j_1|,|j_3|)-1,i_1-3)= F_1(j_1,j_3,i_1).
\end{equation}
(The argument is analogous to that used for counting $\ast$s without $-1$'s.)
If we impose the condition that there is also no $1$ in column $j_2$ in this subarray then this number is
\begin{equation}
\label{no}
p(1,\min) p(\min,\mathrm{mid}-1) p(\mathrm{mid}-1,\max-2) p(\max-2,i_1-4)= F_2(j_1,j_2,j_3,i_1),
\end{equation}
where  $\min=\min(|j_1|,|j_2|,|j_3|)$, $\max=\max(|j_1|,|j_3|)$,
$\mathrm{mid}=|j_1|+|j_2|+|j_3|-\min-\max$. There are
$p(i_1-1,i_2-3) p(i_2,n-1)$
possibilities for extending the latter type of subarray to a permissible $\ast$ of order $n$. Combining \eqref{total}
and \eqref{no}, it is clear that there are
$$
F_1(j_1,j_3,i_1) - F_2(j_1,j_2,j_3,i_1)
$$
possibilities for the bottom $i_1-1$ rows that have a $1$ in column $j_2$. Here, there are $p(i_1,i_2-2) p(i_2+1,n)$ ways to extend this to a permissible $\ast$ of order $n$. In total, there are
\begin{align*}
\sum_{-n+1 < j_1 < j_2 < j_3 < n-1} \sum_{\max(|j_1|,|j_3|)+1 \le i_1 < i_2 \le n}
& \big[ F_2(j_1,j_2,j_3,i_1) p(i_1-1,i_2-3) p(i_2,n-1)  \\
&  + \left( F_1(j_1,j_3,i_1) - F_2(j_1,j_2,j_3,i_1) \right) p(i_1,i_2-2) p(i_2+1,n) \big]
\end{align*}
configurations. It is tedious but straightforward to show that this is indeed equal
to $\binom{n}{3}^2 (n-3)!$.

\end{appendix}

\bibliographystyle{alpha}

\end{document}